\documentclass[11pt,oneside,english]{amsart}
\usepackage[T1]{fontenc}
\usepackage[latin9]{inputenc}
\usepackage{geometry}
\usepackage{mathrsfs}
\usepackage{amsthm}
\usepackage{amsbsy}
\usepackage{amssymb}
\usepackage{esint}

\makeatletter

\newcommand{\noun}[1]{\textsc{#1}}

\numberwithin{equation}{section}
\numberwithin{figure}{section}
\theoremstyle{plain}
\newtheorem{thm}{\protect\theoremname}
  \theoremstyle{definition}
  \newtheorem{defn}[thm]{\protect\definitionname}
  \theoremstyle{plain}
  \newtheorem{cor}[thm]{\protect\corollaryname}
  \theoremstyle{plain}
  \newtheorem{prop}[thm]{\protect\propositionname}
  \theoremstyle{plain}
  \newtheorem{lem}[thm]{\protect\lemmaname}

\usepackage{wrapfig}
\usepackage[all]{xy}
\usepackage{bbm}

\newcommand{\pru}{\Pr_\Gamma^u}

\newcommand{\one}{\mathbbm{1}}
\newcommand{\lp}{\left(}
\newcommand{\rp}{\right)}
\newcommand{\moy}[1]{\langle #1 \rangle}
\newcommand{\lva}{\left|}
\newcommand{\rva}{\right|}
\newcommand{\lno}{\left\|}
\newcommand{\rno}{\right\|}
\newcommand{\fr}{\frac}

\newcommand{\ti}[1]{\tilde {#1}}
\newcommand{\eps}{\varepsilon}
\renewcommand{\d}{\partial}
\newcommand{\defeq}{\stackrel{\rm{def}}{=}}

\newcommand{\cB}{{\mathcal B}}
\newcommand{\cC}{{\mathcal C}}

\newcommand{\cE}{\mathcal E}
\newcommand{\cF}{\mathcal F}

\newcommand{\cJ }{\mathcal J}

\newcommand{\cL}{{\mathcal L}}

\newcommand{\cO}{{\mathcal O}}

\newcommand{\cR}{\mathcal R}
\newcommand{\cS}{\mathcal S}

\newcommand{\cU}{{\mathcal U}}

\newcommand{\cZ}{\mathcal Z}

\newcommand{\scP}{\mathscr P}
\newcommand{\scR}{\mathscr R}
\newcommand{\scL}{\mathscr L}

\newcommand{\bT}{\boldsymbol T}
\newcommand{\bD}{\boldsymbol D}

\newcommand{\C}{{\mathbb C}}
\newcommand{\N}{{\mathbb N}}

\newcommand{\R}{{\mathbb R}}
\newcommand{\Z}{{\mathbb Z}}
\newcommand{\SP}{{\mathbb S}}

\renewcommand{\H}{\mathbb H}

\newcommand{\supp}{\operatorname {supp}}

\newcommand{\Tr}{\operatorname{Tr}}

\newcommand{\Spec}{\operatorname{Spec}}
\newcommand{\Span}{\operatorname{Span}}
\newcommand{\e}{\operatorname{e}}
\newcommand{\Op}{\operatorname{Op}}
\newcommand{\oph}{\operatorname{Op}_{h}}
\newcommand{\Id}{\operatorname{Id}}

\newcommand{\diam}{\operatorname{diam}}

\renewcommand{\Im }{\operatorname{Im}}
\renewcommand{\Re}{\operatorname{Re}}

\renewcommand{\Pr}{\operatorname{Pr}}
\newcommand{\sgn}{\operatorname{sgn}}
\renewcommand{\div}{\operatorname{div}}
\newcommand{\loc}{\operatorname{loc}}
\renewcommand{\i}{\operatorname{i}}
\renewcommand{\top}{\operatorname{top}}
\newcommand{\Hess}{\operatorname{Hess}}

\newcommand{\Graph}{\operatorname{Graph}}

\newcommand{\sub}{\operatorname{sub}}
\newcommand{\comp}{\operatorname{comp}}

\usepackage{babel}

\providecommand{\theoremname}{Theorem}

\usepackage{babel}
\providecommand{\corollaryname}{Corollary}
  \providecommand{\definitionname}{Definition}
  \providecommand{\lemmaname}{Lemma}
  \providecommand{\propositionname}{Proposition}
\providecommand{\theoremname}{Theorem}

\usepackage{amssymb}

\makeatother

\usepackage{babel}
  \providecommand{\corollaryname}{Corollary}
  \providecommand{\definitionname}{Definition}
  \providecommand{\lemmaname}{Lemma}
  \providecommand{\propositionname}{Proposition}
\providecommand{\theoremname}{Theorem}

\begin{document}
\global\long\def\scal#1#2{\langle#1,#2\rangle}

\title{Resonances near the real axis for manifolds with hyperbolic trapped
sets}

\author{Emmanuel Schenck}

\address{Laboratoire d'analyse, géométrie et applications, Université Paris
13, CNRS UMR 7539, 93430 Villetaneuse, France.}

\email{schenck@math.univ-paris13.fr }
\begin{abstract}
For manifolds Euclidian at infinity and compact perturbations of the
Laplacian, we show that under assumptions involving hyperbolicity
of the classical flow on the trapped set and its period spectrum,
there are strips below the real axis where the resonance counting
function grows sub-linearly. We also provide an inverse result, showing
that the knowledge of the scattering poles can give some information
about the Hausdorff dimension of the trapped set when the classical
flow satisfies the Axiom-A condition. 
\end{abstract}
\maketitle

\section{Introduction}

Let $M$ be a $C^{\infty}$ manifold of dimension $n\geq2$ which
agrees with $\R^{n}$ outside a compact set: 
\[
M=M_{0}\sqcup M_{1},\qquad M_{0}\simeq\R^{n}\setminus B(0,R_{0}),\quad R_{0}>1.
\]
Here, $M_{1}$ is a smooth compact manifold with connected boundary
$\d M_{1}\simeq\SP^{n-1}$, and we can then see $M$ as a compact
perturbation of $\R^{n}$. We assume that $M$ is equipped with a
positive density $dx$ which coincides with the Lebesgue measure on
$M_{0}$. We will be interested in the scattering theory on $M$ for
a positive elliptic self-adjoint pseudodifferential operator $P$
of order 2. The operator $P$ is supposed to agree with the Euclidian
Laplacian on $M_{0}$: 
\[
P|_{M_{0}}=-\Delta|_{\R^{n}\setminus B(0,R_{0})}\,.
\]
The total symbol $\sigma_{P}(x,\xi)$ of $P$ is assumed to be classical,
in the Kohn-Nirenberg class $S^{2}(M)$, and with vanishing sub-principal
symbol : 
\[
\sigma_{\sub}(P)=0.
\]
If $p\in C^{\infty}(T^{*}M)$ denotes the principal symbol of $P$,
the classical Hamiltonian flow $\e^{tH_{p}}:T^{*}M\to T^{*}M$ is
given by 
\[
\left(\begin{array}{c}
\dot{x}\\
\dot{\xi}
\end{array}\right)=H_{p}(x,\xi)=\left(\begin{array}{c}
\d_{\xi}p(x,\xi)\\
-\d_{x}p(x,\xi)
\end{array}\right)\,,\quad(x,\xi)\in T^{*}M.
\]
It is symplectic and preserves the energy layers $p^{-1}(E)\subset T^{*}M$,
$E>0$. By definition, the trapped set $\Gamma_{E}$ at energy $E>0$
is the set of points that do not escape to infinity in the future,
nor in the past: 
\[
\Gamma_{E}=\{\rho\in T^{*}M:\ p(\rho)=E\ \mbox{and}\ \e^{tH_{p}}(\rho)\nrightarrow\infty,\ t\to\pm\infty\}\subset T^{*}M.
\]
Our principal assumption in this article is that $\Gamma_{E}\neq\emptyset$
for $|E-1|$ small enough, and if $\Gamma\defeq\Gamma_{1}$, then
\[
\e^{tH_{p}}:\Gamma\to\Gamma\ \ \mbox{is hyperbolic.}
\]
The standard situation to be kept in mind is the metric scattering,
where $M$ is endowed with a Riemannian metric $g$ and 
\[
P=-\Delta_{g},\quad g|_{M_{0}}=g_{\mathrm{Eucl}},\quad g|_{M_{1}\setminus\d M_{1}}\ \mbox{is negatively curved}.
\]
For our main result, we add an assumption about the period spectrum
of $\e^{tH_{p}}$, which we denote by 
\[
\scL\defeq\{\ell>0,\ \exists\rho\in\cE^{*}M\ \mbox{with}\ \e^{\ell H_{p}}(\rho)=\rho\}.
\]
We will assume that $\scL$ is not too much clustered, insofar as
for some constant $\nu>0$ and any $T>0$ large enough, there is at
least one $\ell\in\scL\cap[T-1,T]$ such that $[\ell-\e^{-\nu T},\ell+\e^{-\nu T}]\cap\scL=\ell$.
Our precise assumption is slightly weaker and will be discussed after
Definition \ref{def: minimal sep} below. For the moment, we just
point out that in constant negative curvature, the period spectrum
satisfies stronger separation properties \cite{DolJak16_gaps}. 

The first result of this article deals with the resonance counting
function in strips of finite size $s>0$ below the real axis. Under
the precise assumptions in Theorem \ref{thm: MLPC}, we show that
this resonance counting function satisfies: 
\[
\exists\alpha,s_{0}>0,\ \forall s\ge s_{0},\quad\sharp\left\{ \lambda\mbox{ resonance of }P:\ |\Im\lambda|\leq s,\ |\Re(\lambda)|\leq r\right\} \geq Cr^{1-\frac{\alpha}{s}}+C'.
\]
Resonance counting functions in strips have been intensively studied
since the conjecture of Lax and Phillips \cite{LaxPhi89_Scatt}, and
we will briefly review below the literature on this subject.

\medskip{}

The resolvent $(P-\lambda^{2})^{-1}$ continues meromorphically as
an operator $L_{\comp}^{2}\to L_{\loc}^{2}$ from $\Im z>0$, $z^{2}\notin\Spec_{pp}(P)$
where it is analytic to $\C$ when $n$ is odd, and the logarithmic
cover $\Lambda_{\C}$ of $\C$ when $n$ is even \cite{SjoZwo91_CplxScal}.
The poles of the meromorphic continuation of the resolvent are called
the resonances, or scattering poles of $P$, and they are the objects
that replace the usual spectrum of $P$ when $M$ is compact. We will
denote by $\mathscr{R}_{M}(P)$ the resonances of $P$ on the manifold
$M$. Remark also that there is $k\in\N$ such that $\one_{\R^{n}\setminus B(0,R_{0})}(P+i)^{-k}$
is of trace class. In this way, the black box scattering formalism
of Sjöstrand and Zworski applies, in particular there are Poisson
formulæ for resonances in this context, both for odd and even dimensions,
see \cite{SjoZwo94_LowII,Zwo98_even} and Section \ref{sec: Poisson formula}
below.

Resonances close to the real axis are of particular interest, for
instance in studying the local energy decay for the solutions of the
wave equation $(\d_{t}^{2}+P)u=0$. The trapping properties of the
classical flow $\e^{tH_{p}}$ have a direct consequence on the repartition
of these scattering poles in the lower half plane: in the non-trapping
case $\Gamma=\emptyset$, it has been shown that there are pole-free
regions of logarithmic sizes below the real axis \cite{MelSjo78_1,MelSjo82_2}.
On the other hand, when Lax and Phillips first published their monograph
\cite{LaxPhi89_Scatt}, they conjectured that if $\Gamma\neq\emptyset$,
there should exist a sequence of resonances $(\lambda_{i})_{i\in\N}$
converging to the real axis, namely 
\[
\lim_{i\to+\infty}\Im\lambda_{i}=0.
\]

Ikawa \cite{Ika82_decay} showed that this conjecture was generally
incorrect, as he established a strip with no resonances below the
real axis in the case of Euclidian scattering in $\R^{3}$ by two
compact, disjoint convex obstacles. Ikawa formulated then what is
known to be the modified Lax-Phillips conjecture: for a trapping scattering
problem where by convention the resonances are located in the lower
half plane, there is $\alpha<0$ such that the strip 
\[
S_{\alpha}=\{z\in\C:\alpha\leq\Im z\leq0\}
\]
contains infinitely many resonances. A good amount of results have
been obtained so far concerning this conjecture, which turned out
to be true in various settings \cite{Ger88_bull,Far95_low,SteVod96_Neu,Sto09_LPC,Pet02_low},
see also the survey article of Sjöstrand \cite{Sjo97_traceRev}.

However, the precise distribution of the scattering poles near the
real axis is still badly understood. Unlike the counting function
inside disks of radius $r\to\infty$, which in our settings reads
\cite{SjoZwo91_CplxScal}: 
\[
C^{-1}r^{n}\leq\sharp\{\lambda\in\scR_{M}(P),\ |\lambda|\leq r\}\leq Cr^{n},\quad C>0,
\]
there is no such asymptotics for the counting function in a fixed
strip below the real axis. Sjöstrand \cite{Sjo90} first proved an
upper bound of fractal type for potential scattering in a semiclassical
framework. Building on this work, Guillopé, Lin and Zworski \cite{GuiLinZw04}
have proven geometric upper bounds for $P=-\Delta_{g}$ acting on
$L^{2}(G\backslash\H^{n+1})$ where $G$ is a convex, co-compact Schottky
group: 
\[
\forall\alpha>0,\ \exists C>0,\quad\sharp\{\lambda\in\scR_{G\backslash\H^{n+1}}(-\Delta_{g}):\ |\Im\lambda|\leq\alpha,\ |\Re\lambda|\leq r\}\leq Cr^{1+\delta}
\]
where $\delta=\dim\Lambda(G)$ is the dimension of the limit set of
$G$. This upper bound has led to conjecture a lower bound of the
same order, conjecture which is known as the fractal Weyl law. Unfortunately,
few lower bounds for the counting function in a strip are known. For
$X$ a surface of constant negative curvature, Guillopé and Zworski
\cite{GuiZwo99_Tr} have shown that for any (small) $\eps>0$ and
(large) $A>0$, there is a constant $C_{\eps}>0$ and a sequence $r_{i}\to\infty$
such that 
\begin{equation}
\sharp\{\lambda\in\scR_{X}(-\Delta):\ |\Im\lambda|\leq C_{\eps},\ |\Re\lambda|\leq r_{i}\}\geq Ar_{i}^{1-\eps}.\label{eq: Lower bound Gui-Zwo}
\end{equation}
Even if this lower bound can be generalized in higher dimensions for
the same type of manifolds, it is not sensitive to the size of the
trapped set, and is not even optimal in the elementary case of the
hyperbolic cylinder (with a single trapped orbit), where a lower bound
is computable and linear as the resonances are distributed on a lattice.

For scattering by disjoint convex obstacles in dimension $n=3$, lower
bounds on the number of scattering poles in strips below the real
axis have been obtained by Farhy \cite{Far95_low} and later by Petkov
\cite{Pet02_low} who proved bounds similar to \eqref{eq: Lower bound Gui-Zwo}
and \eqref{eq: Lower bound Gui-Zwo}, building on earlier works of
Ikawa \cite{Ika85_osa}. 

All of the above results about estimating the counting function in
strips rely on a trace formula, which connects resonances to periodic
orbits of the flow $\e^{tH_{p}}:\Gamma\to\Gamma$. Under a separation
assumption for the period spectrum (see the definition immediately
below), our first result can be stated as follows:
\begin{thm}
\label{thm: MLPC}Let $M$ be a $C^{\infty}$ manifold of dimension
$n\geq2$ such that 
\[
M=M_{0}\sqcup M_{1},\qquad M_{0}\simeq\R^{n}\setminus B(0,R_{0}),\quad R_{0}>1,\quad M_{1}\mbox{ compact.}
\]
Let $P$ be a second order elliptic, positive, selfadjoint pseudodifferential
operator as above, and assume that:

(i) if $M$ has even dimension, 0 is not a resonance of $P$,

(ii) the Hamiltonian flow $\e^{tH_{p}}:\cE^{*}M\to\cE^{*}M$ has a
non-empty trapped set on which it is hyperbolic, and its period spectrum
is minimally separated.

\medskip{}
 Then there are positive constants $C$,$\Theta$,$\epsilon_{0}$
such that for any $\epsilon\leq\epsilon_{0}$, there is $r_{0}(\epsilon)>0$
with 
\[
\sharp\{\lambda\in\scR_{M}(P):\ |\Im\lambda|\leq\frac{2n}{\epsilon}\ \mbox{and }\ |\Re\lambda|\leq r\}\geq Cr^{1-\epsilon\Theta},\qquad\forall r\geq r_{0}(\epsilon).
\]

\end{thm}
\medskip{}

More can be said about the constant $\Theta$, for this we refer to
Section \ref{sub:Counting-poles}. These lower bounds are however
still far from the already known fractal upper bounds, but they are
the first explicit lower bounds for such systems. Let us discuss now
our assumption about the period spectrum. 
\begin{defn}
\emph{\label{def: minimal sep}We will say that $\scL$ is minimally
separated, if for some $\nu>0$ the following property holds true
:} 
\end{defn}
\emph{For any $C_{0}>0$, there is $T_{0}>0$ such that if $T\geq T_{0}$,
any interval $[T-1/2,T]$ contains at least one sequence of consecutive
periods $\ell_{1}<\dots<\ell_{k}$ with $k\geq3$ and 
\[
\ell_{2}-\ell_{1}\geq\e^{-\nu T},\qquad\ell_{k}-\ell_{k-1}\geq\e^{-\nu T},\qquad\ell_{k-1}-\ell_{2}\leq\e^{-C_{0}T}.
\]
}

\medskip{}

In other words, given $C_{0}>0$, for $T$ large enough we can always
find in $[T-1/2,T]$ two gaps of size $\e^{-\nu T}$ in the period
spectrum with between them either a single period, or a group of periods
that spread over a distance at most $\e^{-C_{0}T}$. The value $1/2$
is irrelevant in this definition and has been chosen for later convenience.
What is important is that it is independent of $T$.

This assumption is rather reasonable for hyperbolic flows. As an important
example, much more is true in dimension $\geq3$ for the length spectrum
(i.e. the set of lengths of closed geodesics) of a finite volume hyperbolic
manifold :\emph{ every }pair of closed geodesics have lengths exponentially
bounded from below \cite{DolJak16_gaps}. This means that there is
$\nu>0$ such that $|\ell(\gamma)-\ell(\gamma')|\geq\e^{-\nu T}$
for every closed geodesics $\gamma\neq\gamma'$ with lengths in $[T-1/2,T]$
for $T$ large enough. For compact surfaces in constant negative curvature,
Dolgopyat and Jakobson note that this strong property remains true
for a dense set in the corresponding Teichmüller space. In particular,
all this would apply to our settings if $M_{1}$ is built from such
hyperbolic manifolds : the period spectrum would be minimally separated,
since in $[T-1/2,T]$ the number of non-primitive periods (corresponding
to iterations of closed geodesics) is negligible compared to the number
of primitive ones. 

For Anosov flows on basic sets, the number of periodic orbits with
length $\leq T$ grows exponentially fast as $T\to\infty$. For instance,
if $\Gamma$ is a basic set and $\e^{tH_{p}}:\Gamma\to\Gamma$ is
weak mixing, then \cite{MarSha04_anosov}: 
\[
\lim_{T\to\infty}\frac{1}{T}\log\sharp\{\gamma\in\scP:\ \ell(\gamma)\in[T-1/2,T]\}=h_{\top}>0,
\]
where $h_{\top}$ denotes the topological entropy and $\scP$ set
of periodic orbits : 
\[
\scP\defeq\{\rho\in\cE^{*}M:\ \exists t\in\R\ \mbox{with}\ \e^{tH_{p}}(\rho)=\rho\}.
\]
 Hence, if $s>h_{\top}$, by a box principle the number of gaps of
size $\e^{-sT}$ in $\scL\cap[T-1/2,T]$ grows exponentially as $T\to\infty$.
This means that many occurrences of two consecutive periods separated
by at least $\e^{-sT}$ appear, however it could happen that (possibly
many) periods cluster together with arbitrary small gaps between them.
Note that $\scL$ being minimally separated does not exclude this
possibility either.

It is plausible that the minimal separation assumption is satisfied
for a reasonably large set of Anosov flows : these flows are structurally
stable, which means that small $C^{r}$ perturbations ($r\geq1$)
of the flow are still Anosov, while the lengths of the periodic orbits
would not be preserved in general \cite{Sch}.

\medskip{}

Our second result deals with a topological pressure related to the
infinitesimal unstable Jacobian $J^{u}$ on the trapped set, defined
by 
\[
J^{u}(\rho)\defeq\left.\frac{d}{dt}\right|_{t=0}\log\det d\Phi_{\rho}^{t}|_{E_{\rho}^{u}},
\]
where $E_{\rho}^{u}$ is the unstable subbundle at $\rho\in\Gamma$,
see Section \ref{sec: Anosov}. We will consider the pressure of $-\frac{1}{2}J^{u}$
with respect to the flow $\e^{tH_{p}}:\Gamma\to\Gamma$, which we
will denote by 
\[
\Pr_{\Gamma}^{u}\defeq\Pr_{\Gamma}(-\frac{1}{2}J^{u}).
\]
A precise definition of this pressure is given in \eqref{eq: def pressure}
below, see also \cite{NoZw09_1}, Section 3.3 for a general discussion
in our context. This pressure is always a real number, the sign of
which may vary according to the geometry of the trapped set: it is
heuristically considered that $\Gamma$ is large, or ``thick'',
if $\Pr_{\Gamma}^{u}\geq0$, and small, or ``filamentary'', if $\pru<0$.
This consideration comes from the fact that the map $t\mapsto\Pr_{\Gamma}(-tJ^{u})$
is decreasing for $t\geq0$, and that in the case $\dim M=2$, the
unique root $t_{u}>0$ of Bowen's equation is related to $d_{H}(\Gamma)$,
the Hausdorff dimension of $\Gamma$: 
\[
\Pr_{\Gamma}(-t_{u}J^{u})=0\ \Leftrightarrow\ d_{H}(\Gamma)=2t_{u}+1,\qquad\dim M=2.
\]
In dimension $n\geq3$, there is no simple relation between $d_{H}(\Gamma)$
and $t_{u}$, unless the flow $\e^{tH_{p}}$ is conformal in the stable
and unstable directions \cite{PeSa01}. Nonnenmacher and Zworski \cite{NoZw09_1}
have established the important role played by this pressure concerning
the presence of a spectral gap near the real axis for hyperbolic scattering
systems similar to those we consider in this article, but in a semiclassical
framework. Translated to our settings, they showed that if $\Pr_{\Gamma}^{u}<0$,
then there is a spectral gap near the real axis, namely for any $\eps>0$
and $\alpha_{\eps}=\Pr_{\Gamma}^{u}+\eps$, 
\[
\sharp\{\lambda\in\scR_{M}(P):\ \lambda\in S_{\alpha_{\eps}}\}<\infty.
\]
To state our second result, let us denote by $\Tr u(t)$ the distributional
trace of the wave group $u(t)=\cos t\sqrt{P}$, see Section \ref{sec: Poisson formula}
below. As a result, the pressure $\Pr_{\Gamma}^{u}$ can be computed
from the knowledge of this trace when the periodic orbits are dense
in $\Gamma$ (here no assumption on the length spectrum is needed)
: 
\begin{thm}
\label{thm: spectral invariant}Let $M,P$ be as above, and assume
that the flow $\e^{tH_{p}}:\Gamma\to\Gamma$ is Axiom-A. We have:
\begin{equation}
\Pr_{\Gamma}(-\fr12J^{u})=\lim_{T\to+\infty}\frac{1}{T}\log\left(\limsup_{\Xi\to\infty}\sup_{\lambda\in[\e^{\exp T},\e^{\exp\Xi T}]}\langle\Tr u,f_{\lambda,T}\rangle\right)\,.\label{eq: Spectral Invariant}
\end{equation}
The function $f_{\lambda,T}\in C_{0}^{\infty}(\R_{+})$ is given by
$f_{\lambda,T}(t)=\cos(\lambda t)\phi\left(\frac{t-T+1/2}{1/2}\right)$,
where $\phi\in C_{0}^{\infty}(\R)$ is positive and equal to 1 near
0. 
\end{thm}
In other words, the pressure $\Pr_{\Gamma}^{u}$ is a scattering invariant
in our geometric settings: for instance if $n$ is odd, the trace
part in the right hand side can be computed directly from the resonances
$\scR_{M}(P)$ using the Poisson formula. The situation is however
more subtle if $n$ is even, as the Poisson formula is not exact,
see \eqref{eq:Poisson Test even}.

This theorem has also an interesting inverse scattering result for
corollary once we have in mind the relationship between $\Pr_{\Gamma}^{u}$
and the Hausdorff dimension of $\Gamma$. For this, define first the
stable infinitesimal Jacobian by 
\[
J^{s}(\rho)\defeq\left.\frac{d}{dt}\right|_{t=0}\log\det d\Phi_{\rho}^{t}|_{E_{\rho}^{s}},
\]
and denote by $t^{s}\geq0$ the root of Bowen's equation $\Pr_{\Gamma}(t_{s}J^{s})=0$.
We recall that $t_{u}+t_{s}+1=d_{H}(\Gamma)$ if $\e^{tH_{p}}:\Gamma\to\Gamma$
is conformal in the stable and unstable directions \cite{PeSa01}.
Theorem \ref{thm: spectral invariant} implies: 
\begin{cor}
Let $M$ and $P$ be as above, with $\e^{tH_{p}}:\Gamma\to\Gamma$
Axiom-A. Let $\Tr u(t)$ be the distributional wave trace of $P$,
$d_{H}(\Gamma)$ the Hausdorff dimension of $\Gamma$, and $f_{\lambda,T}$
be as in Theorem \ref{thm: spectral invariant}. Assume either that:

(i) $M$ has dimension 2, or

(ii) $M$ has dimension $n\geq3$, the flow $\e^{tH_{p}}$ is conformal
in the stable and unstable directions, and $t_{u}=t_{s}$.

Then, 
\[
\lim_{T\to+\infty}\frac{1}{T}\log\left(\limsup_{\Xi\to\infty}\sup_{\lambda\in[\e^{\exp T},\e^{\exp\Xi T}]}\langle\Tr u,f_{\lambda,T}\rangle\right)<0\ \Leftrightarrow\ d_{H}(\Gamma)<2.
\]

\end{cor}
In particular if $\dim M$ is odd or $P$ has no resonances at $0$,
the simple knowledge of the scattering poles gives an information
about the Hausdorff dimension of the trapped set in this geometrical
framework. Indeed, as noted above the distribution $\Tr u(t)$ is
either completely determined by the resonances in odd dimensions,
or in even dimensions with no resonances at 0, the terms in $\Tr u(t)$
which are not determined by the resonances produce an irrelevant $\cO(1)$
term in the formula $\langle\Tr u,f_{\lambda,T}\rangle$ in the limit
$T\to+\infty$, see \eqref{eq: Poisson even} below.

\section{The Trace Formula for resonances\label{sec: Poisson formula}}

As stated in the introduction, we assume that the symbol $\sigma_{P}(x,\xi)$
of $P$ belongs in the class $S^{2}(M),$ where 
\[
S^{m}(M)\defeq\{a\in C^{\infty}(T^{*}M):\ |\d_{x}^{\alpha}\d_{\xi}^{\beta}a(x,\xi)|\leq C_{\alpha\beta}\moy\xi^{m-|\beta|}\}.
\]
The hypothesis that $P$ is classical means that in each coordinate
chart, $\sigma_{P}$ admits an asymptotic expansion 
\[
\sigma_{P}(x,\xi)\sim\sum_{j=0}^{\infty}p_{2-j}(x,\xi)
\]
with $p_{2-j}\in C^{\infty}(T^{*}M)$ homogeneous of degree $2-j$.
If $P$ satisfies $\sigma_{\sub}(P)=0,$ it implies in particular
that if $Q=\sqrt{P}$ is the square root of the unique self-adjoint
extension of $P$ in $L^{2}(M)$ given by the spectral theorem, then
$\sigma_{\sub}(Q)=0$ as well.

The operator $Q=\sqrt{P}$ is then a pseudodifferential operator of
order 1, with principal symbol $q=\sqrt{p}$. The Hamiltonian vector
field satisfies $H_{q}=\fr12H_{p}$ on $T^{*}M$. We will denote by
\[
\Phi^{t}\defeq\e^{tH_{q}}:T^{*}M\to T^{*}M
\]
the Hamiltonian flow of $H_{q}$, and $\Phi^{t}:\Gamma\to\Gamma$
is still hyperbolic. We now recall the various Poisson formulæ available
in our context. Let $T>R_{0}>0$ and $\chi=\chi_{T}\in C_{0}^{\infty}(M)$
with $\chi\geq0$, $\chi=1$ in $M\setminus(\R^{n}\setminus B(0,R_{0}+T))$,
$\chi=0$ on $\R^{n}\setminus B(0,R_{0}+2T)\subset M_{0}$. In particular,
$\chi=1$ near $M_{1}$, and $P$ coincides with $-\Delta$ on the
support of $1-\chi$. We will be interested in the wave group 
\[
u(t)\defeq\cos t\sqrt{P}\,,
\]
and the truncated free wave group 
\[
(1-\chi)u_{0}(t)(1-\chi)\,,\quad u_{0}(t)=\cos t\sqrt{-\Delta}\,.
\]

Denote by $C_{T}^{\infty}(\R_{+})$ the set of functions in $C_{0}^{\infty}(\R_{+})$
supported on the interval $[0,T]$. Because of the properties of $\chi$
and the finite speed of propagation of the waves in $\R^{n}$, the
operator

\begin{equation}
\int(u(t)-(1-\chi)u_{0}(1-\chi))\rho(t)dt\label{eq:Trace class}
\end{equation}
is a smoothing operator with compactly supported kernel and defines
a tempered distribution $\Tr u(t)$ on $C_{T}^{\infty}(\R_{+})$ by
the formula 
\[
\Tr u:\begin{cases}
C_{T}^{\infty}(\R_{+})\to\C\\
\rho\mapsto\langle\Tr u,\rho\rangle\defeq2\Tr\int\lp u(t)-(1-\chi)u_{0}(t)(1-\chi)\rp\rho(t)dt.
\end{cases}
\]
The idea that it is possible to relate the above trace to the resonances
of $P$ originated in the work of Bardos, Guillot and Ralston \cite{BarGuiRal82_poisson},
followed by Melrose \cite{Mel82_trace}, see also Sjöstrand and Zworski
\cite{SjoZwo94_LowII} for the general presentation in the black box
formalism. The Poisson formula established by these authors make this
relation explicit, and the formula required by the framework of this
article is due to Sjöstrand and Zworski in odd dimensions \cite{SjoZwo94_LowII},
and Zworski \cite{Zwo98_even} in even dimensions.

If $n$ is odd, the formula reads 
\begin{equation}
\Tr u(t)=\sum_{\lambda\in\scR_{M}(P)}m_{\lambda}\e^{-\i|t|\lambda}\,,\qquad t\neq0,\label{eq: Poisson odd}
\end{equation}
where $m_{\lambda}$ is the multiplicity of $\lambda$ as a pole of
the resolvent of $P$.

When $n$ is even, the formula is no longer as simple, and we follow
\cite{Zwo98_even} for the presentation. Let $\rho\in]0,\pi/2[$ and
denote by 
\[
\Lambda_{\rho}=\{r\e^{\i\theta}:r>0,\ |\theta+2k\pi|<\rho,\ k\in\Z\}\cup\{r\e^{\i\theta}:r>0,\ |\theta-\pi+2k\pi|<\rho,\ k\in\Z\}
\]
a conic open neighborhood of the real axis on the logarithmic plane
with cut $\i\R_{+}$. Finally let $\sigma(\lambda)$ be the scattering
phase of $P$ and $\psi\in C_{0}^{\infty}(\R,[0,1])$ be equal to
1 near 0. Then 
\begin{align}
\Tr u(t) & =\sum_{\lambda\in\Lambda_{\rho}}m(\lambda)\e^{\i|t|\lambda}+\sum_{{\lambda^{2}\in\Spec_{pp}(P)\cap\R_{-}\atop \Im\lambda<0}}m(\lambda)\e^{\i|t|\lambda}\nonumber \\
 & +m(0)+2\int_{0}^{\infty}\psi(\lambda)\frac{d\sigma}{d\lambda}(\lambda)\cos t\lambda d\lambda+v_{\rho,\psi}(t)\,,\ n\ \mbox{even},\label{eq: Poisson even}
\end{align}
where $v_{\rho,\psi}\in C^{\infty}(\R\setminus\{0\})$ satisfies $\partial_{t}^{k}v_{\rho,\psi}=\cO(t^{-N})$
for all $k,N$ and $t\to\infty$. We refer to \cite{Zwo98_even} and
the references given there for the definition and properties of the
scattering phase $\sigma(\lambda)$. In particular, if $0$ is not
a resonance of $P$, then $\sigma'(\lambda)$ is smooth near 0.

On the other hand, the Duistermaat-Guillemin trace formula \cite{DuGu75}
gives informations on the left hand side of \eqref{eq: Poisson odd}
and \eqref{eq: Poisson even} and show precisely that the distribution
$\Tr u(t)$ has singularities located at times that correspond to
length of periodic orbits of the flow $\Phi^{t}$ in $\Gamma$. The
next paragraph makes this more precise.

\subsection{Semiclassical formulation and localization of the trace\label{sub:Semiclassical formulation}}

In this article we will consider test functions $\rho$ in \eqref{eq:Trace class}
which are supported inside a bounded interval near a time $T>0$,
where ultimately $T\to+\infty$. The finite speed of propagation implies
that if $\Pi\in C_{0}^{\infty}(M)$ is equal to 1 in a neighborhood
of 
\[
B_{T}\defeq M_{1}\sqcup(B(0,R_{0}+2T)\setminus B(0,R_{0})),
\]
then the operator \eqref{eq:Trace class} has a smooth Schwartz kernel
included in $B_{T}\times B_{T}$, so we might simply look at $\Pi(u-(1-\chi)u_{0}(1-\chi))\Pi$
to compute the trace: 
\begin{align*}
\langle\Tr u,\rho\rangle & =2\iint\Pi(x)(u(t,x,x)-(1-\chi(x))u_{0}(t,x,x)(1-\chi(x)))\Pi(x)\rho(t)dtdx.
\end{align*}
Two specific test functions $\rho$ will be used: the first one for
Theorem \ref{thm: MLPC} and the second one for Theorem \ref{thm: spectral invariant}.
We will choose them of the form 
\[
\rho(t)=\cos(\lambda t)\phi^{(i)}(t),\qquad i\in\{1,2\}
\]
where $\lambda>0$ will be large and $\phi^{(i)}$ is a real, compactly
supported function. To define $\phi^{(i)}$, let first $\phi\in C_{0}^{\infty}([-1,1])$
be a positive function, with $0\leq\phi\leq1=\phi(0)$. We will assume
that $\phi=1$ on $[-3/4,3/4]$ and define 
\begin{equation}
\phi^{(i)}(t)=\phi\left(\frac{t-b_{i}}{a_{i}}\right)\,,\quad a_{i},b_{i}>0,\quad i=1,2.\label{eq: Tests functions}
\end{equation}
For the first test function $\phi^{(1)}$, we will choose a parameter
$\beta>0$ such that for some $\epsilon>0$ small but fixed, we have
\begin{equation}
T=\epsilon\log\beta\leq\epsilon\log\lambda\quad\mbox{and}\quad\beta\leq\lambda\leq\beta+1.\label{eq: Relation beta lambda}
\end{equation}
Hence the two parameters $\lambda,\beta$ are not independent: in
practice, for the proof of Theorem \ref{thm: MLPC} we will choose
an arbitrary (large) value of $\beta>0$, and then adjust $\lambda$
accordingly, as it will be explained in Section \ref{sub:Counting-poles}.
The value of $a_{1}$ and $b_{1}$ will be chosen such that 
\begin{equation}
b_{1}\in]T-1,T[,\qquad a_{1}=\e^{-J_{+}T}=\beta^{-\epsilon J_{+}},\qquad T-1\leq b_{1}\pm a_{1}\leq T\label{eq: Choice param phi1}
\end{equation}
where $J_{+}>0$ is a fixed constant that depends only on $M$ and
$p$ that will be defined later in \eqref{eq: choice J +}. The precise
value of $b_{1}$ needed for Theorem \ref{thm: MLPC} will also be
defined in Section \ref{sub:Counting-poles}.

For the second test function $\phi^{(2)}$, we will only assume that
$T\leq\epsilon\log\lambda$ and and set 
\begin{equation}
b_{2}=T-\fr12,\qquad a_{2}=\fr12.\label{eq: Choice param phi2}
\end{equation}

For convenience, we will rather consider the unitary groups 
\[
U(t)=\e^{-\i tQ},\quad U_{0}(t)=\e^{-\i t\sqrt{-\Delta}}
\]
and work with the (half) wave equation written on the form $(\frac{1}{\i}\frac{\d}{\d t}+Q)f=0.$
The next proposition establishes a correspondence for the wave groups
$U$, $u$ and $u_{0}$. 
\begin{prop}
\label{prop: Trace Simplifiee}Let us write $U(t,x,y)$ for the Schwartz
kernel of $U(t)$. Then, 
\begin{align}
2\Tr\int(u(t)-u_{0}^{\chi}(t))\cos(\lambda t)\phi^{(i)}(t)dt & =\Re\iint\Pi(x)U(t,x,x)\Pi(x)\e^{\i\lambda t}\phi^{(i)}(t)dxdt\label{eq: Integral Trace}\\
 & +\cO(\lambda^{-\infty}),\nonumber 
\end{align}
where $u_{0}^{\chi}(t)\defeq(1-\chi)u_{0}(t)(1-\chi)$ and $\lambda\to+\infty$.
\end{prop}
The proof is postponed in Section \ref{sub: proof trace simple}.
Roughly speaking, in the limit $\lambda\to\infty$, it is sufficient
to consider the term involving $u$ in the trace, since when subtracting
the cut-off free wave group, the part involving $u_{0}$ does not
contribute significantly to the trace as $(1-\chi)u_{0}(t)(1-\chi)$
has a singular support in time reduced to 0. For convenience, in the
next section we will write 
\[
\langle\Tr U,\rho\rangle\defeq\iint\Pi(x)U(t,x,x)\Pi(x)\rho(t)dxdt
\]
and the above proposition simply states that 
\[
\scal{\Tr u(t)}{\cos(\lambda t)\phi^{(i)}(t)}=\Re\scal{\Tr U}{\e^{\i\lambda t}\phi^{(i)}(t)}+\cO(\lambda^{-\infty}).
\]

The study of the above traces in the limit $\lambda\to\infty$ is
more easily treated with semiclassical microlocal analysis, so we
present now a semiclassical formulation of the problem -- see \cite{JaPoTo07_low}
for an similar reformulation, and \cite{Zwo10_semiAMS} for a complete
introduction to semiclassical analysis. We note that if $h\defeq\lambda^{-1}$,
then $hQ=\sqrt{h^{2}P}$ and $hQ_{0}=\sqrt{-h^{2}\Delta}$ are classical
$h-$semiclassical operators of order 1. Our symbol classes for the
semiclassical calculus can now depend on $h$, and we will use the
following notation: 
\[
S_{\delta}^{m,k}(M)\defeq\{a\in C^{\infty}(T^{*}M\times]0,1]):\ |\d_{x}^{\alpha}\d_{\xi}^{\beta}a(x,\xi,h)|\leq C_{\alpha\beta}h^{k-\delta(|\alpha|+|\beta|)}\moy\xi^{m-|\beta|}\}.
\]
The assumption that the subprincipal symbol of $P$ vanishes implies
that 
\[
hQ=\Op_{h}(\sqrt{p})+h^{2}\Op_{h}(q_{2}),\qquad q_{2}\in S_{0}^{-1,0}(M),
\]
where $\Op_{h}(\cdot)$ denotes the Weyl quantization. In the following
we will write $q=\sqrt{p}$ for the principal symbol of $Q$. As a
result, the wave groups 
\[
u(t)=u_{h}(t)=\cos(h^{-1}t\sqrt{h^{2}P}),\quad u_{0,h}(t)=\cos(h^{-1}t\sqrt{-h^{2}\Delta})
\]
as well as $U(t)$ and $U_{0}(t)$ can be seen as semiclassical Fourier
integral operators (albeit with trivial dependance on $h$), and we
are lead to conduct the semiclassical analysis of the oscillatory
integral 
\[
\iint\Pi(x)(u_{h}(t,x,x)-(1-\chi(x))u_{0,h}(t,x,x,)(1-\chi(x))\Pi(x)\cos(h^{-1}t)\phi^{(i)}(t)dxdt\,,\quad h\to0.
\]

We conclude this section with the microlocalisation of the trace near
$\cE^{*}M$. From the fact that the principal symbol of $Q$ and $Q_{0}$
is homogeneous of degree 1, it is easy to see (\cite{DuGu75}, Section
1) that the operators 
\[
\int\Pi\e^{-\frac{\i}{h}t(hQ-1)}\Pi\phi^{(i)}(t)dt\quad\mbox{and}\quad\int\Pi(1-\chi)\e^{-\frac{\i}{h}t(hQ_{0}-1)}(1-\chi)\Pi\phi^{(i)}(t)dt
\]
are microlocalized in phase space around $\cE^{*}M$. As a result,
if $f\in C_{0}^{\infty}(\R)$ is supported in an interval of the form
$[1-\delta/2,1+\delta/2]$ for some $\delta>0$, and is identically
equal to 1 near 1, then we have 
\begin{multline*}
\Tr\lp\int\Pi f(hQ)(U(t)-(1-\chi)U_{0}(t)(1-\chi))\Pi\e^{\i\frac{t}{h}}\phi^{(i)}(t)dt\rp\\
=\Tr\lp\int\Pi(U(t)-(1-\chi)U_{0}(t)(1-\chi))\Pi\e^{\i\frac{t}{h}}\phi^{(i)}(t)dt\rp+\cO(h^{\infty}).
\end{multline*}
Here the remainder is actually $\cO_{N,\delta}(h^{N(1-\epsilon J_{+})})$
for all $N\in\N$, due to the fact that 
\[
\lno\frac{d^{N}}{dt^{N}}\phi^{(1)}(t)\rno_{L^{\infty}}=\cO(\lambda^{N\epsilon J_{+}}),\qquad\lno\frac{d^{N}}{dt^{N}}\phi^{(2)}(t)\rno_{L^{\infty}}=\cO(1).
\]

\subsection{Long-time Trace formula}

To prove Theorems \ref{thm: MLPC} and \ref{thm: spectral invariant},
our main tool is a long-time trace formula, applied to a well-chosen
test function. This formula is simply a generalization of the Duistermaat-Guillemin
trace formula for test functions with support that can vary with $\lambda\to+\infty$,
in a non-compact setting. 
\begin{prop}
\label{prop: Long Time Trace Formula} Denote by $\scP$ the set of
periodic bicharacteristics of the flow $\Phi^{t}:\Gamma\to\Gamma$,
and for $\gamma\in\scP$, call $\ell^{\sharp}(\gamma)$ its primitive
length and $P_{\gamma}$ its Poincaré map. Let $\phi^{(i)}$, $i\in\{1,2\}$
be as in \eqref{eq: Tests functions} with $a_{i},b_{i}$ satisfying
\eqref{eq: Choice param phi1}, \eqref{eq: Choice param phi2}. There
is $\ti\epsilon>0$ and $C=C_{M,p,\phi}>0$ such that for any $\epsilon\leq\ti\epsilon$,
we have the following expansion when $\lambda\to+\infty$: 
\begin{equation}
\moy{\Tr U,\e^{\i t\lambda}\phi^{(i)}}=\sum_{\gamma\in\scP}\e^{\i\lambda\ell(\gamma)}\frac{\ell^{\sharp}(\gamma)}{\sqrt{|1-P_{\gamma}|}}\phi^{(i)}(\ell(\gamma))+\cO_{M,p}(\lambda^{-\mu})\,,\quad\mu=1-C\epsilon>0.\label{eq: Trace W}
\end{equation}

\end{prop}
The particularity of this trace formula, albeit fairly classical,
is twofolds: the number of periodic orbits that are considered can
be large, for instance it can grow exponentially when $i=2$, and
their length diverge with $\lambda\to+\infty$. Also, the remainder
in $\lambda$ is precisely controlled with respect to $\epsilon$.
In \cite{JaPoTo07_low}, a similar formula is derived for the Laplacian
on a compact surface with negative curvature.

\section{Hyperbolic and Axiom-A Hamiltonian flows\label{sec: Anosov}}

Hyperbolic dynamical systems have been extensively studied in the
last decades, see \cite{KaHa95} for a comprehensive introduction
and many additional properties of such systems. Our main assumption
in this article is that 
\[
\Phi^{t}:\cE^{*}M\to\cE^{*}M
\]
is hyperbolic when restricted to the trapped set $\Gamma$. The set
$\Gamma$ is a compact space, invariant under the flow. By definition
of hyperbolicity, for any $\rho\in\Gamma$, the tangent space $T_{\rho}\cE^{*}M$
splits into \emph{flow}, \emph{stable} and \emph{unstable} subspaces
\[
T_{\rho}\cE^{*}M=\R H_{q}\oplus E_{\rho}^{s}\oplus E_{\rho}^{u}\,.
\]
The spaces $E_{\rho}^{s}$ and $E_{\rho}^{u}$ are $n-1$ dimensional,
and are preserved under the flow map: 
\[
\forall t\in\R,\ \ d\Phi_{\rho}^{t}(E_{\rho}^{s})=E_{\Phi^{t}(\rho)}^{s},\quad d\Phi_{\rho}^{t}(E_{\rho}^{u})=E_{\Phi^{t}(\rho)}^{u}.
\]
Moreover, there exist $\mu,C>0$ such that on $\cE^{*}M$, 
\begin{eqnarray}
i) &  & \|d\Phi_{\rho}^{t}(v)\|\leq C\e^{-\mu t}\|v\|,\ \mbox{\ for\ all\ }v\in E_{\rho}^{s},\ t\geq0\nonumber \\
ii) &  & \|d\Phi_{\rho}^{-t}(v)\|\leq C\e^{-\mu t}\|v\|,\mbox{\ \ for\ all\ }v\in E_{\rho}^{u},\ t\geq0.\label{eq:Unstable}
\end{eqnarray}
There is a metric near $\Gamma$, called the \emph{adapted metric,
}such that one can take $C=1$ in the above equations. This metric
can be extended to the total energy layer $\cE^{*}M$ in a way that
it coincides with the standard Euclidian metric outside $T^{*}M_{1}$.

The adapted metric on $\cE^{*}M$ induces a volume form $\Omega_{\rho}$
on any $n-1$ dimensional subspace of $T(\cE_{\rho}^{*}M)$. Using
$\Omega_{\rho}$, we can define the unstable Jacobian at $\rho$ for
time $t$. We set 
\[
\det d\Phi^{t}|_{E_{\rho}^{u}\to E_{\Phi^{t}(\rho)}^{u}}=\frac{\Omega_{\Phi^{t}(\rho)}(d\Phi^{t}v_{1}\wedge\dots\wedge d\Phi^{t}v_{n-1})}{\Omega_{\rho}(v_{1}\wedge\dots\wedge v_{n-1})}\,,
\]
where $(v_{1},\dots,v_{n-1})$ can be any basis of $E_{\rho}^{u}$.
The infinitesimal unstable Jacobian is defined by

\begin{equation}
J^{u}(\rho)=\left.\frac{d}{dt}\right|_{t=0}\log\det d\Phi_{\rho}^{t}|_{E_{\rho}^{u}},\quad\mbox{where\ }d\Phi_{\rho}^{t}:E_{\rho}^{u}\to E_{\Phi^{t}(\rho)}^{u}.\label{eq: Def J u}
\end{equation}
The unstable Jacobian at $\rho\in T^{*}M$ will be denoted by 
\[
\e^{\lambda_{t}^{+}(\rho)}=\det d\Phi_{\rho}^{t}|_{E_{\rho}^{u}}=\e^{\int_{0}^{t}J^{u}(\Phi^{s}(\rho))ds}\xrightarrow{t\to+\infty}+\infty.
\]
The flow $\Phi^{t}:\cE^{*}M\to\cE^{*}M$ is said to be \emph{Axiom-A
}if periodic orbits are dense in $\Gamma$. We recall that hyperbolic
sets are structurally stable, namely there is $\delta>0$ such that
$\forall E\in[1-\delta,1+\delta]$, $\Gamma_{E}$ is a hyperbolic
set for $\Phi^{t}|_{p^{-1}(E)}:\Gamma_{E}\to\Gamma_{E}$. Hence in
the thickened unit energy layer 
\[
\cE^{*}M^{\delta}\defeq\{\rho\in T^{*}M:\ |q(\rho)-1|\leq\delta\},
\]
the dynamics is uniformly hyperbolic \cite{KaHa95}, provided $\delta$
is sufficiently small.

Let $d$ be the distance function associated with the adapted metric.
This will lead to consider the open balls around $\rho_{0}\in\cE^{*}M$
defined by $B_{\rho_{0}}(\eps)\defeq\{\rho\in T^{*}M:d(\rho,\rho_{0})\leq\eps\}$,
where the distance is measured with the adapted metric and $\eps>0$
is small enough so that $B_{\rho_{0}}(\eps)\subset\cE^{*}M^{\delta}$.
We end this paragraph by noting that there are some constants $\eps_{0},C_{+},K_{+}>0$
depending only on $p$ and $M$ such that for any $\rho_{0}\in\cE^{*}M$
and $\eps<\eps_{0}$ we have 
\begin{equation}
\diam\Phi^{t}(B_{\rho_{0}}(\eps))\leq C_{+}\e^{K_{+}t}\eps\label{eq: Growth Bowen}
\end{equation}
where the distance is again measured with the adapted metric.

\subsection{The topological pressure for Axiom A flows\label{sub: Pressure}}

Let $f\in C^{0}(T^{*}M)$. For every $\eps>0$ and $T>0$, a set $E\subset\Gamma\subset T^{*}M$
is $(\eps,T)$ separated if for every $x,y\in E$ with $x\neq y$,
$d(\Phi^{t}(x),\Phi^{t}(y))>\eps$ for some $t\in[0,T]$. Since $\Gamma$
is compact, the cardinal of such a set is always bounded, but it may
grow exponentially with $T$. Set 
\[
Z_{T}(\Phi,f,\eps)=\sup\left(\sum_{x\in E}\int_{0}^{T}f(\Phi^{s}(x))ds\right)
\]
where the supremum is taken over all $(\eps,T)$ separated subsets
of $\Gamma$. The topological pressure of the flow $\Phi^{t}:\Gamma\to\Gamma$,
with respect to the function $f$ is defined by 
\begin{equation}
\Pr_{\Gamma}(f)=\lim_{\eps\to0}\limsup_{T\to\infty}\frac{1}{T}\log Z_{T}(\Phi,f,\eps).\label{eq: def pressure}
\end{equation}
If the flow $\Phi^{t}$ is Axiom A on $\Gamma$, it is possible to
express the topological pressure by using the periodic orbits (see
\cite{KaHa95}, Chapter 18 and 20). More precisely, 
\[
\Pr_{\Gamma}(f)=\lim_{T\to\infty}\frac{1}{T}\log\sum_{\gamma\in\scP:\ell(\gamma)\leq T}\e^{\int_{\gamma}f},\qquad\Phi^{t}:\Gamma\to\Gamma\mbox{\ is Axiom A}.
\]
If $\Pr_{\Gamma}(f)>0$, then one has the even more precise asymptotics
\cite{Bow72_per}: 
\[
\sum_{\gamma\in\scP:\ell(\gamma)\leq T}\e^{\int_{\gamma}f}=\frac{\e^{T\Pr_{\Gamma}(f)}}{\Pr_{\Gamma}(f)}(1+o(1))\,,\qquad T\to+\infty.
\]

\section{Counting scattering poles \label{sub:Counting-poles}}

This entire section is devoted to the proof of Theorem 1, assuming
Proposition \ref{prop: Long Time Trace Formula}. The trace formula
\eqref{eq: Trace W} can be used to derive a lower bound on the number
of scattering poles near the real axis. This approach is the main
method available to obtain lower bounds in resonance counting problems.
However, the dependence on $\beta$ and $\lambda$ in the test function
here is not standard as we will potentially deal with many periodic
orbits of diverging length when these parameters tend to $+\infty$.
Through this section, we then assume that the length spectrum of $\Phi^{t}$
is minimally separated.

\subsection{Lower bound for the trace formula up to Ehrenfest time.}

Let $h_{\top}\geq0$ denote the topological entropy of the flow, and
define 
\begin{equation}
\Theta_{+}^{u}=\limsup_{t\to\infty}\frac{1}{2t}\log\sup_{\rho\in\Gamma\subset\cE^{*}M}\det\lva1-d\Phi^{t}|_{E_{\rho}^{u}}\rva>0.\label{eq: J u +}
\end{equation}
Let $\nu$ be as in Definition \ref{def: minimal sep} and choose
a positive constant $J_{+}$ such that 
\begin{equation}
J_{+}>\max(\Theta_{+}^{u},h_{\top},\nu)>0.\label{eq: choice J +}
\end{equation}
Now fix $\epsilon>0$ small enough so that Proposition \ref{prop: Long Time Trace Formula}
holds with a function $\phi^{(1)}$ as in \eqref{eq: Tests functions},
with $a_{1}=\beta^{-\epsilon J_{+}}$ and $b_{1}\in]T-1,T[$ to be
defined. Take then 
\[
C_{0}\defeq2\epsilon^{-1},\quad\mbox{hence}\quad T=\epsilon\log\beta\Rightarrow\e^{-C_{0}T}=\beta^{-2}.
\]
Since for Theorem \ref{thm: MLPC} we assumed that $\scL$ is minimally
separated, for $\beta$ large enough we can always find an interval
$I\subset[T-1,T]$ such that 
\[
I=I_{-1}\sqcup I_{0}\sqcup I_{1}
\]
with $I_{-1},I_{1}$ open of size $\e^{-J_{+}T}$, $I_{\pm1}\cap\scL=\emptyset$,
$I_{0}=[\ell_{0},\ell_{1}]$ with $\ell_{0},\ell_{1}\in\scL$ and
either 
\[
\ell_{0}=\ell_{1}\quad\mbox{or}\quad\ell_{1}-\ell_{0}\leq\beta^{-2}.
\]

Let now set $b_{1}=\ell_{0}$. It is then clear that the test function
$\phi^{(1)}$ defined as in \eqref{eq: Tests functions} is supported
on $I$, and it satisfies $\phi^{(1)}(\ell_{0})=1$. The next proposition
gives an estimate on the right hand side of \eqref{eq: Trace W},
and to keep in mind that $\phi^{(1)}$ depends on the parameter $\beta$
we write $\phi_{\beta}=\phi^{(1)}$ in the following. 
\begin{prop}
\label{prop: Lower bound trace} Let $\phi_{\beta}=\phi^{(1)}$ be
as in \eqref{eq: Tests functions} with $b_{1},a_{1}$ as above. Then
there are $c_{0},\beta_{0}>0$ such that for any $\beta\geq\beta_{0}$
and $\epsilon$ sufficiently small, there is a choice of the parameter
$\lambda$ such that \eqref{eq: Relation beta lambda} holds true
and furthermore, 
\[
\langle\Tr u(t),\cos(\lambda t)\phi_{\beta}(t)\rangle\geq c_{0}\lambda^{-\epsilon\Theta_{+}^{u}}
\]
where $\Theta_{+}^{u}>0$ is given in \eqref{eq: J u +}. The constant
$c_{0}$ depends on $M$ and $p$. \end{prop}
\begin{proof}
From Proposition \ref{prop: Trace Simplifiee} it is enough to estimate
$\Re\scal{\Tr U}{\e^{\i\lambda t}\phi_{\beta}}$. Applying the trace
formula to the test function $\e^{\i\lambda t}\phi_{\beta}(t)$ and
taking the real part of both sides yields to

\begin{equation}
\Re\langle\Tr U,\e^{\i\lambda t}\phi_{\beta}\rangle=\sum_{\gamma\in\scP}\cos(\lambda\ell(\gamma))\frac{\ell^{\sharp}(\gamma)}{\sqrt{|1-P_{\gamma}|}}\phi_{\beta}(\ell(\gamma))+\cO_{M,p,\phi}(\lambda^{-\mu}).\label{eq:Trace real part}
\end{equation}
The difficulty to get lower bounds from this formula usually come
from the highly oscillatory terms $\cos(\lambda\ell(\gamma))$. However
from our choice of $a_{1}$ and $b_{1}$, the number of periods in
the support of $\phi_{\beta}$ is either one, or the periods spread
out over a size $\leq\beta^{-2}$ near $\ell_{0}$. It follows that
once $\beta$ is chosen, the value of $\ell_{0}$ is determined and
we can easily adjust $\lambda=\lambda(\beta)$ such that for $\beta$
sufficiently large, then 
\begin{equation}
\cos(\lambda\ell_{0})=1,\quad\lambda,\beta\mbox{\ satisfy \eqref{eq: Relation beta lambda}}.\label{eq: lower bound cosine}
\end{equation}
Note that if $\{\ell\in\scL\cap\supp\phi_{\beta}\}=\ell_{0}\leq\dots\leq\ell_{k}$
and $k>0$, then for $0<i\leq k$ we have 
\[
|\cos(\lambda\ell_{i})-\cos(\lambda\ell_{0})|\leq\lambda\beta^{-2}<\frac{1}{2},\qquad\beta\ \mbox{large enough}.
\]
From this we conclude that $\forall\ell\in\scL\cap\supp\phi_{\beta}$
we have 
\[
\cos(\lambda\ell)\geq\frac{1}{2}.
\]
\emph{From now on, and until the end of Section \ref{sub:Counting-poles},
we work with this precise value of $\lambda=\lambda(\beta)$. }

We then deduce from \eqref{eq:Trace real part} that 
\[
\Re\langle\Tr U,\e^{\i\lambda t}\phi_{\beta}\rangle\geq\frac{1}{2}\sum_{\ell(\gamma)\in\supp\phi_{\beta}}\frac{\ell^{\sharp}(\gamma)}{\sqrt{|1-P_{\gamma}|}}\phi_{\beta}(\ell(\gamma))+\cO_{M,p,\phi}(\lambda^{-\mu}).
\]
Let us analyze further the right hand side when $\beta,\lambda\to\infty$.
By the definition of $\Theta_{+}^{u}$ in \eqref{eq: choice J +},
we have that 
\[
\sqrt{|1-P_{\gamma}|}=\lva\det(1-d\Phi^{\ell(\gamma)}|_{E^{s}})\det(1-d\Phi^{\ell(\gamma)}|_{E^{u}})\rva^{\fr12}<\frac{1}{c_{0}}\e^{T\Theta_{+}^{u}}
\]
for some $c_{0}>0$. Since $\phi_{\beta}>0$, $\phi_{\beta}=1$ near
$\ell_{0}$ and recalling that $T\leq\epsilon\log\lambda$, we can
keep only the orbits such that $\ell(\gamma)=\ell_{0}$ to get 
\begin{align*}
\sum_{\ell(\gamma)\in\supp\phi_{\beta}}\frac{\ell^{\sharp}(\gamma)}{\sqrt{|1-P_{\gamma}|}}\phi_{\beta}(\ell(\gamma)) & \geq c_{0}\sharp\{\gamma:\ell(\gamma)=\ell_{0}\}\lambda^{-\epsilon\Theta_{+}^{u}}.
\end{align*}
Now, from proposition \ref{prop: Long Time Trace Formula} we see
that the term $\cO_{M,p,\phi}(\lambda^{-\mu})$ is $o(\lambda^{-\epsilon\Theta_{+}^{u}})$
if $\epsilon$ is sufficiently small and can then be absorbed in the
constant $c_{0}$, which concludes the proof. 
\end{proof}

\subsection{Proof of Theorem \ref{thm: MLPC}: $n$ odd.}

In this section we still write $\phi_{\beta}=\phi^{(1)}$. If $n$
is odd, then applying the distributional trace $\Tr u$ to $\cos(\lambda t)\phi_{\beta}(t)$,
we get from \eqref{eq: Poisson odd}: 
\begin{align*}
\moy{\Tr u,\cos(\lambda t)\phi_{\beta}(t)} & =\langle\sum_{\lambda_{j}\in\scR_{M}(P)}\e^{-{\it \lambda_{j}}t},\cos(\lambda t)\phi_{\beta}(t)\rangle\\
 & =\sum_{\lambda_{j}\in\scR_{M}(P)}\sum_{\iota=\pm1}a_{1}\hat{\phi}(a_{1}(\lambda_{j}+\iota\lambda))\e^{-\i b_{1}(\lambda_{j}+\iota\lambda)}\\
 & =\sum_{\lambda_{j}\in\scR_{M}(P)}\sum_{\iota=\pm1}\hat{\phi}_{\beta}(\lambda_{j}+\iota\lambda)\,,
\end{align*}
where 
\[
\hat{\phi}_{\beta}(\tau)\defeq a_{1}\hat{\phi}(a_{1}\tau)\e^{-\i b_{1}\tau}.
\]
We are now in position to prove our main result. The Paley-Wiener
estimate for $\hat{\phi}_{\beta}$ reads: 
\begin{align*}
|\hat{\phi}_{\beta}(\tau)| & \leq C_{K}a_{1}\e^{(b_{1}-a_{1})\Im\tau}(1+|a_{1}\tau|)^{-K}\,,\quad\mbox{for any }K>0,\ \Im\tau\leq0\,.
\end{align*}
Recall that $a_{1}=\beta^{-\epsilon J_{+}}$ and $b\geq\epsilon\log\beta-1$.
This implies that for any $K>0$ there is $C_{K}>0$ such that for
$\beta$ large enough, 
\[
|\hat{\phi}_{\beta}(\tau)|\leq C_{K}a_{1}\e^{\frac{\epsilon}{2}\log\beta\Im\tau}\frac{1}{(1+a_{1}|\tau|)^{K}}.
\]
Let $s>0$, and define 
\[
\cB_{s}\defeq\{z\in\C:-s\leq\Im z\leq0\}\,.
\]
We rewrite the sum over $j$ as 
\begin{align*}
\sum_{\lambda_{j}\in\scR_{M}(P)}\hat{\phi}_{\beta}(\lambda_{j}-\lambda)+\hat{\phi}_{\beta}(\lambda_{j}+\lambda) & =\sum_{\lambda_{j}\in\cB_{s}}\sum_{\iota=\pm1}\hat{\phi}_{\beta}(\lambda_{j}+\iota\lambda)+\sum_{\lambda_{j}\notin\cB_{s}}\sum_{\iota=\pm1}\hat{\phi}_{\beta}(\lambda_{j}+\iota\lambda)\,.
\end{align*}
For the deep resonances outside the strip $\cB_{s}$, we use the Paley-Wiener
estimate to write, for $\beta>0$ large enough, 
\begin{align*}
|\hat{\phi}_{\beta}(\lambda_{j}\pm\lambda)| & \leq C_{K,\epsilon,s}\beta^{-J_{+}\epsilon}\e^{-\frac{\epsilon}{2}|\Im(\lambda_{j})|\log\beta}\lp1+\beta^{-\epsilon J_{+}}|\lambda_{j}\pm\lambda|\rp^{-K}\\
 & \leq C_{K,\epsilon,s}\beta^{-J_{+}\epsilon}\e^{-\frac{\epsilon}{2}|\Im(\lambda_{j})|\log\beta}\lp1+\beta^{-J_{+}\epsilon}\lva|\lambda_{j}|-\lambda\rva\rp^{-K}.
\end{align*}
Denote now the counting function in the lower half-plane by 
\[
N(\lambda)\defeq\sharp\{\lambda_{j}\in\scR_{M}(P):\ |\lambda_{j}|\leq\lambda\}
\]
and introduce the counting measure 
\[
dN(t)=\sum_{\lambda_{j}\in\scR_{M}(P)}\delta(t-|\lambda_{j}|)dt\,.
\]
The sum over the resonances outside $\cB_{s}$ can be rewritten using
this counting measure: 
\begin{multline*}
\lva\sum_{\lambda_{j}\notin\cB_{s}}\hat{\phi}_{\beta}(\lambda_{j}\pm\lambda)\rva\leq\sum_{\lambda_{j}\notin\cB_{s}}C_{K,\epsilon,s}\beta^{-J_{+}\epsilon}\e^{-\frac{\epsilon}{2}|\Im(\lambda_{j})|\log\beta}\left(1+(|\lambda_{j}|-\lambda)\beta^{-J_{+}\epsilon}\right)^{-K}\\
\leq C_{K,\epsilon,s}\int_{s}^{+\infty}\frac{\beta^{-s\frac{\epsilon}{2}}\beta^{-J_{+}\epsilon}}{(1+|t-\lambda|\beta^{-J_{+}\epsilon})^{K}}dN(t)\leq C_{K,\epsilon,s}\int_{s}^{+\infty}\frac{t^{n}\beta^{-(J_{+}+\frac{s}{2})\epsilon}}{(1+|t-\lambda|\beta^{-J_{+}\epsilon})^{K}}dt\\
\leq C_{K,\epsilon,s}\lambda^{n}\beta^{-(J_{+}+\frac{s}{2})\epsilon}\leq C'_{K,\epsilon,s}\lambda^{n-(J_{+}+\frac{s}{2})\epsilon}.
\end{multline*}
In the last line, we have used a standard estimate $dN([0,t])\leq Ct^{n}$
valid in our geometric settings and the fact that $|\lambda-\beta|\leq1$.
Finally, it is easily checked that 
\begin{equation}
s>2(\frac{n}{\epsilon}+\Theta_{+}^{u}-J_{+})\ \Rightarrow\lva\sum_{\lambda_{j}\notin\cB_{s}}\hat{\phi}_{\beta}(\lambda_{j}\pm\lambda)\rva\leq C_{K,\epsilon,s}\lambda^{n-(\frac{s}{2}+J_{+})\epsilon}=o(\lambda^{-\epsilon\Theta_{+}^{u}})\,.\label{eq: Choice width strip}
\end{equation}
As we will show below, this property will be required to get a lower
bound for the scattering poles in the strip $\cB_{s}$, which then
needs to have a width of order $1/\epsilon$ in practice, since $\epsilon>0$
has to be chosen small. If $s$ is chosen according to \eqref{eq: Choice width strip},
we have obtained 
\[
\sum_{\lambda_{j}\in\scR_{M}(P)}\sum_{\iota=\pm1}\hat{\phi}_{\beta}(\lambda_{j}+\iota\lambda)=\sum_{\lambda_{j}\in\cB_{s}}\sum_{\iota=\pm1}\hat{\phi}_{\beta}(\lambda_{j}+\iota\lambda)+o(\lambda^{-\epsilon\Theta_{+}^{u}})
\]
and it remains to estimate the sum over resonances in $\cB_{s}$.
For this we adapt the technique of \cite{SjoZwo93_Low}. Note first
that if $\tau=r+\i\sigma$ with $-s\leq\sigma\leq0$ and $|r|\to\infty$,
then 
\[
\hat{\phi}_{\beta}(r+\i\sigma)=\hat{\phi}_{\beta}(r)+\cO_{K,s}\left((r\beta^{-\epsilon J_{+}}){}^{-K}\right)\,.
\]
Then, if $x\in\R$ and still $-s\leq\sigma\leq0$, we have 
\begin{align*}
|\hat{\phi}_{\beta}(x+\i\sigma)| & \leq a_{1}\e^{\sigma b_{1}}|\hat{\phi}(a_{1}(x+\i\sigma))|\\
 & \leq a_{1}|\hat{\phi}(a_{1}x)|+\lp|\hat{\phi}(a_{1}(x+\i\sigma))-\hat{\phi}(a_{1}x)|\rp.
\end{align*}
A straightforward computation using the Paley-Wiener estimate and
\eqref{eq: Choice width strip} shows that the second term of the
right hand side of the above equation is $\cO_{\epsilon,M,p}(a_{1}\langle a_{1}x\rangle^{-K})$
for any $K\in\N$. So we have 
\[
|\hat{\phi}_{\beta}(x+\i\sigma)|\leq a_{1}|\hat{\phi}(a_{1}x)|+\cO_{\epsilon,M,p}(a_{1}\langle a_{1}x\rangle^{-K}),\quad\forall K\in\N.
\]
In particular, if we write 
\[
\psi(x)\defeq\sup_{-s\leq\sigma\leq0}|a_{1}\hat{\phi}(a_{1}(x+\i\sigma))|
\]
it is clear that $\psi(x)\geq a_{1}|\hat{\phi}(a_{1}x)|$, so we have
obtained 
\begin{equation}
\psi(x)=a_{1}|\hat{\phi}(a_{1}x)|+\cO_{\epsilon,M,p}(a_{1}\langle a_{1}x\rangle^{-K}),\quad\forall K\in\N.\label{eq: Bound Fourier}
\end{equation}
Let us define now a measure on the real line that counts the resonances
in the strip $\cB_{s}$ by

\begin{equation}
d\mu_{s}(x)={\displaystyle \sum_{\lambda_{j}\in\scR_{M}(P)\cap\cB_{s}}\delta(x-r_{j})},\qquad r_{j}=\Re\lambda_{j}.\label{eq: Counting measure}
\end{equation}
We can then write 
\begin{align*}
\lva\sum_{\lambda_{j}\in\scR_{M}(P)\cap\cB_{s}}\sum_{\iota=\pm1}\hat{\phi}_{\beta}(\lambda_{j}+\iota\lambda)\rva & \leq\sum_{\lambda_{j}\in\scR_{M}(P)\cap\cB_{s}}\sum_{\iota=\pm1}\psi(r_{j}+\iota\lambda)\\
 & =\int_{\R}\sum_{\iota=\pm1}\psi(x+\iota\lambda)d\mu_{s}(x)\\
 & =\psi*\mu_{s}(-\lambda)+\psi*\mu_{s}(\lambda)\,.
\end{align*}
We first examine the term $\psi*\mu_{s}(\lambda)$. For $\nu>0$ arbitrary,
decompose this term as 
\begin{align*}
\psi*\mu_{s}(\lambda)= & \int_{\lambda-\lambda^{\nu+\epsilon J_{+}}}^{\lambda+\lambda^{\nu+\epsilon J_{+}}}\psi(x-\lambda)d\mu_{s}(x)+\int_{-\infty}^{\lambda-\lambda^{\nu+\epsilon J_{+}}}\psi(x-\lambda)d\mu_{s}(x)\\
 & +\int_{\lambda+\lambda^{\nu+\epsilon J_{+}}}^{+\infty}\psi(x-\lambda)d\mu_{s}(x)\,.
\end{align*}
The Paley-Wiener estimate for $\hat{\phi}_{\beta}$ implies that 
\[
\int_{-\infty}^{-\lambda^{\nu+\epsilon J_{+}}}a_{1}|\hat{\phi}(a_{1}x)|dx=\cO_{K}(\lambda^{-K\nu}),\quad\int_{\lambda^{\nu+\epsilon J_{+}}}^{+\infty}a_{1}|\hat{\phi}(a_{1}x)|dx=\cO_{K}(\lambda^{-K\nu})\,.
\]
Since $\nu>0$, notice that these terms are $\cO(\lambda^{-\infty})$.
These equations, together with \eqref{eq: Bound Fourier} and the
bound $d\mu_{s}([-t,t])=\cO(t^{n})$, imply now that 
\[
\int_{-\infty}^{\lambda-\lambda^{\nu+\epsilon J_{+}}}\psi(x-\lambda)d\mu_{s}(x)=\cO_{K}(\lambda^{n-K\nu}),\qquad\int_{\lambda+\lambda^{\nu+\epsilon J_{+}}}^{+\infty}\psi(x-\lambda)d\mu_{s}(x)=\cO_{K}(\lambda^{n-K\nu})\,.
\]
We could write exactly the same type of estimates for $\psi*\mu_{s}(-\lambda)$
but now for intervals of the form $]-\infty,-\lambda-\lambda^{\nu+\epsilon J_{+}}]$
and $[-\lambda+\lambda^{\nu+\epsilon J_{+}},+\infty[$, and altogether
this yields to

\[
c_{0}\lambda^{-\epsilon\Theta_{+}^{u}}\leq\int_{\lambda-\lambda^{\nu+\epsilon J_{+}}}^{\lambda+\lambda^{\nu+\epsilon J_{+}}}\psi(x-\lambda)d\mu_{s}(x)+\int_{-\lambda-\lambda^{\nu+\epsilon J_{+}}}^{-\lambda+\lambda^{\nu+\epsilon J_{+}}}\psi(x+\lambda)d\mu_{s}(x)+\cO_{K}(\lambda^{n-K\nu})+o(\lambda^{-\epsilon\Theta_{+}^{u}})\,.
\]
Using \eqref{eq: Bound Fourier}, if $\beta$ (and then $\lambda$)
is large enough this can be rewritten as 
\[
\lambda^{\epsilon(J_{+}-\Theta_{+}^{u})}\leq C_{M,p}\mu_{s}([\lambda-\lambda^{\nu+\epsilon J_{+}},\lambda+\lambda^{\nu+\epsilon J_{+}}]\cup[-\lambda-\lambda^{\nu+\epsilon J_{+}},-\lambda+\lambda^{\nu+\epsilon J_{+}}]),\quad C_{M,p}>0.
\]
At this point, we can not use directly a Tauberian argument to conclude,
since the above counting estimate is a priori valid only for the precise
value of $\lambda$ such that \eqref{eq: lower bound cosine} holds
true once $\beta$ is chosen. However, it is not hard to see that
for an arbitrary value of $\beta$ sufficiently large, we have 
\[
[\lambda-\lambda^{\nu+\epsilon J_{+}},\lambda+\lambda^{\nu+\epsilon J_{+}}]\subset[\beta-\beta^{2\nu+\epsilon J_{+}},\beta+\beta^{2\nu+\epsilon J_{+}}]
\]
and we finally get that 
\begin{align}
\exists C_{M,p},\beta_{0} & >0,\quad\forall\beta\geq\beta_{0},\quad\forall\nu>0,\nonumber \\
 & \beta^{\epsilon(J_{+}-\Theta_{+}^{u})}\leq C_{M,p}\mu_{s}([-\beta-\beta^{\nu+\epsilon J_{+}},-\beta+\beta^{\nu+\epsilon J_{+}}]\cup[\beta-\beta^{\nu+\epsilon J_{+}},\beta+\beta^{\nu+\epsilon J_{+}}]).\label{eq: main counting estimate}
\end{align}
We are now in position to recall a general Tauberian Lemma due to
Sjöstrand and Zworski: 
\begin{lem}
(\cite{SjoZwo93_Low}, Lemma p. 849) Let $\mu_{s}$ be a discrete
counting measure as in \eqref{eq: Counting measure}. Assume that
for some $\delta\in(0,1)$, $c,r_{0}>0$ and $\kappa+\delta\geq0$
we have for $r\geq r_{0}$, 
\[
\mu([r,r+r^{\delta}])\geq cr^{\kappa+\delta}.
\]
Then there is $c_{1},c_{2}>0$ such that 
\[
\mu([0,r])\geq c_{1}r^{1+\kappa}-c_{2}.
\]

\end{lem}
This lemma can be used to conclude the proof of Theorem \ref{thm: MLPC}
in the following way. First if $I$ is a positive interval, we denote
by $-I$ the symmetric interval with respect to 0, and we define $\ti\mu_{s}(I)=\mu_{s}(I)+\mu_{s}(-I)$
which is now a measure on the real axis. The main counting estimate
\eqref{eq: main counting estimate} can now be rewritten 
\[
C_{M,p}\ti\mu_{s}([\beta-\beta^{\nu+\epsilon J_{+}},\beta+\beta^{\nu+\epsilon J_{+}}])\geq\beta^{\epsilon(J_{+}-\Theta_{+}^{u})}.
\]
For any $\eps_{0}>0$, if $r=\beta-\beta^{\nu+\epsilon J_{+}}$, $\delta=\nu+\epsilon J_{+}+\eps_{0}$
and $\beta$ is sufficiently large, then $[\beta-\beta^{\nu+\epsilon J_{+}},\beta+\beta^{\nu+\epsilon J_{+}}]\subset[r,r+r^{\delta}]$
so we can write 
\[
\ti\mu_{s}([r,r+r^{\delta}])\geq\frac{1}{C_{M,p}}r^{\epsilon(J_{+}-\Theta_{+}^{u})}.
\]
A straightforward application the above lemma with $\kappa=-\nu-\epsilon\Theta_{+}^{u}-\eps_{0}$
shows that there is $r_{0},C>0$ such that if $r\geq r_{0}$ then

\[
\mu_{s}([-r,r])=\ti\mu_{s}([0,r])\geq Cr^{1-\nu-\epsilon\Theta_{+}^{u}-\eps_{0}}\geq Cr^{1-3\epsilon\Theta_{+}^{u}}
\]
since $\nu,\eps_{0}$ are arbitrary and can then be chosen to be equal
to $\epsilon\Theta_{+}^{u}$. If we set $s=\frac{2n}{\epsilon}$,
$\Theta=3\Theta_{+}^{u}$ and recall \eqref{eq: choice J +}, \eqref{eq: Choice width strip},
Theorem \ref{thm: MLPC} is then proved.

\subsection{Proof of Theorem \ref{thm: MLPC}, $n$ even. }

Again in this section, $\phi_{\beta}=\phi^{(1)}$. If $n$ is even,
using \eqref{eq: Poisson even}, the trace applied to the test function
$\cos(\lambda t)\phi_{\beta}(t)$ gives

\begin{align}
\moy{\Tr u,\cos(\lambda t)\phi_{\beta}(t)} & =\sum_{\lambda_{j}\in\Lambda_{\rho}}\sum_{\iota=\pm1}\hat{\phi}_{\beta}(\lambda_{j}+\iota\lambda)\label{eq:Poisson Test even}\\
 & +\sum_{{\lambda^{2}\in\Spec_{pp}(P)\cap\R_{-}\atop \Im\lambda<0}}\sum_{\iota=\pm1}\hat{\phi}_{\beta}(\lambda_{j}+\iota\lambda)+\cO(\lambda^{-\epsilon J_{+}})\nonumber 
\end{align}
since the three last terms in \eqref{eq: Poisson even} are controlled
by the size of the support of $\phi_{\beta}$ if $P$ has no resonance
at 0, see \cite{Zwo98_even}, Section 3.

The principle is exactly the same as for $n$ odd. We split the sum
over the resonances $\lambda\in\Lambda_{\rho}$ and $\lambda^{2}\in\Spec_{pp}(P)\cap\R_{-}$,
$\Im\lambda<0$ in two parts. The first one deals with the resonances
located outside a strip of width $s$ so that \eqref{eq: Choice width strip}
holds true with the same arguments. The second one deals with the
resonances inside this strip, where we get now an equality of the
form 
\begin{align*}
0<\lambda^{-\epsilon\Theta_{u}^{+}} & \leq C_{M,p}\left(\mu_{s}([\lambda-\lambda^{\nu+\epsilon J_{+}},\lambda+\lambda^{\nu+\epsilon J_{+}}]\cup[-\lambda-\lambda^{\nu+\epsilon J_{+}},-\lambda+\lambda^{\nu+\epsilon J_{+}}])\right)\\
 & +\cO(\lambda^{-\infty})+\cO(\lambda^{-\epsilon J_{+}})
\end{align*}
where the last remainder comes from the terms in the second line in
\eqref{eq: Poisson even}. Since we have chosen $J_{+}>\Theta_{+}^{u}$,
the last two terms can be absorbed into the left hand side, and we
conclude the proof as for the case $n$ odd.

\section{The Topological Pressure as a scattering invariant\label{sec: Spectral Invarian}}

In this section we give the proof of Theorem \ref{thm: spectral invariant}.
Let us fix some $T>0$, and for a given value of $\Xi>1$ such that
\[
\e^{T}+\e^{h_{\top}T}\log2\leq\e^{\Xi T},
\]
consider a range of values for $\lambda$ given by 
\begin{equation}
T\leq\log\log\lambda\leq\Xi T.\label{eq: interval lambda}
\end{equation}
For convenience, we write 
\[
I_{\Xi}(T)=[\e^{\exp T},\e^{\exp\Xi T}].
\]
If $T$ is large enough, for any choice of $\lambda$ in this interval
we have $T\leq\epsilon\log\lambda$ and we can apply Proposition \ref{prop: Long Time Trace Formula},
with $\phi^{(2)}(t)$ as it has been defined in \eqref{eq: Tests functions}
and \eqref{eq: Choice param phi2}. Applying the long-time trace formula
and taking the real part, we obtain: 
\[
\Re\langle\Tr U,\e^{\i\lambda t}\phi^{(2)}(t)\rangle=\sum_{\gamma\in\scP}\cos(\lambda\ell(\gamma))\frac{\ell^{\sharp}(\gamma)}{\sqrt{|1-P_{\gamma}|}}\phi^{(2)}(\ell(\gamma))+\cO_{M,p}(\lambda^{-\mu}).
\]
An upper bound for the left hand side is easily derived: the cosine
and $\phi^{(2)}(\ell(\gamma))$ terms are bounded by 1, which gives
an upper bound of the right hand side independent of $\lambda$ and
$\Xi$, so we finally get 
\begin{equation}
\limsup_{\Xi\to\infty}\sup_{\lambda\in[\e^{\exp T},\e^{\exp\Xi T}]}\Re\langle\Tr U,\e^{\i\lambda t}\phi^{(2)}\rangle\leq\sum_{\gamma\in\scP:\ell(\gamma)\in[T-1,T]}\frac{\ell^{\sharp}(\gamma)}{\sqrt{|1-P_{\gamma}|}}+C_{M,p}\e^{-\mu\e^{T}}.\label{eq: upper bound Pressure}
\end{equation}
To obtain a lower bound, we must again control the oscillating terms.
Note that with the above choice of $a_{2}$ and $b_{2}$, $\phi^{(2)}$
has support in $[T-1,T]$, and is equal to 1 in $[T-5/4,T-1/4]$.
Define 
\[
\cJ_{T}=\left\{ \gamma\in\scP:\ell(\gamma)\in[T-1,T]\right\} ,\quad\cJ_{T}'=\left\{ \gamma\in\scP:\ell(\gamma)\in[T-5/4,T-1/4]\right\} 
\]
and denote by $\nu(T)$ the number of distinct lengths of the orbits
in $\cJ_{T}$. Observe now that 
\[
\nu(T)\leq\sharp\cJ_{T}\leq\e^{h_{\top}T}\leq\e^{\Xi T}.
\]
A Dirichlet box principle can now be used \cite{JaPoTo07_low}: if
$\left\{ r_{1},\dots,r_{\nu(T)}\right\} $ are $\nu(T)$ distinct
positive real numbers, for any constant $m>0$ there is $\lambda_{0}\in[m,2^{\nu(T)}m]$
such that the following equation holds true: 
\[
1\leq j\leq\nu(T),\quad|\lambda_{0}r_{j}\mod2\pi|\leq\frac{1}{2}.
\]
This immediately yields to $\cos(\lambda_{0}r_{j})\geq\frac{1}{2}$
for any $r_{j}$. Now take $m=\e^{\alpha T}$ with $\alpha>0$. In
particular, with $m=\e^{\e^{T}}$ this implies that $\lambda_{0}$
varies in the interval 
\[
[\e^{\e^{T}},\e^{\e^{T}+\log2\e^{h_{\top}T}}]
\]
which is included in \eqref{eq: interval lambda}. For this special
value of $\lambda=\lambda_{0}$, we have 
\begin{align*}
\frac{1}{2}\sum_{\gamma\in\cJ_{T}'}\frac{\ell^{\sharp}(\gamma)}{\sqrt{|1-P_{\gamma}|}}-C'_{M,p}\e^{-\mu\e^{T}} & \leq\Re\langle\Tr U,\e^{\i\lambda_{0}t}\phi^{(2)}\rangle\\
 & \leq\limsup_{\Xi\to\infty}\sup_{\lambda\in I_{\Xi}(T)}\Re\langle\Tr U,\e^{\i\lambda t}\phi^{(2)}\rangle.
\end{align*}
Combining this equation with \eqref{eq: upper bound Pressure}, we
obtain

\begin{align*}
\frac{1}{2}\sum_{\gamma\in\cJ_{T}'}\frac{\ell^{\sharp}(\gamma)}{\sqrt{|1-P_{\gamma}|}}-C'_{M,p}\e^{-\mu\e^{T}} & \leq\limsup_{\Xi\to\infty}\sup_{\lambda\in I_{\Xi}(T)}\Re\langle\Tr U,\e^{\i\lambda t}\phi^{(2)}\rangle\\
 & \leq\sum_{\gamma\in\cJ_{T}}\frac{\ell^{\sharp}(\gamma)}{\sqrt{|1-P_{\gamma}|}}+C_{M,p}\e^{-\mu\e^{T}}.
\end{align*}
The sums over periodic orbits in the above equation can be related
to the topological pressure $\Pr_{\Gamma}^{u}$. First, recall that
by definition, we have 
\[
\det d\Phi_{\rho}^{t}|_{E_{\rho}^{u}}=\e^{\int_{0}^{t}J^{u}(\Phi^{s}(\rho))ds}.
\]
On the other hand,  hyperbolicity of the flow implies that there are
constants $K_{i},C_{i}>0$, $1\leq i\leq4$, independent of the point
$\rho,$ with 
\[
C_{1}\e^{K_{1}t}\leq\det d\Phi_{\rho}^{t}|_{E_{\rho}^{u}}\leq C_{2}\e^{K_{2}t},\qquad C_{3}\e^{-K_{3}t}\leq\det d\Phi_{\rho}^{t}|_{E_{\rho}^{s}}\leq C_{4}\e^{-K_{4}t}.
\]
This means that if $\gamma$ is a periodic orbit of length $t$ (large),
then $\det(1-P_{\gamma})$ is dominated by the unstable Jacobian,
namely 
\begin{equation}
\det(1-P_{\gamma})=\e^{\int_{0}^{t}J^{u}(\Phi^{s}(\rho))ds}(1+\cO(\e^{-\ti Kt}))\label{eq: approx poincare J u}
\end{equation}
for some $\ti K>0$ independent of $\gamma$. Now, for some $\delta>0$
consider the function $f\in C^{0}(\Gamma)$ given by 
\[
f=-\frac{1}{2}J^{u}+(1+\delta)|\Pr_{\Gamma}^{u}|
\]
where we have written $\Pr_{\Gamma}^{u}\defeq\Pr_{\Gamma}(-\fr12J^{u})$\textbf{.
}By construction we have $\Pr_{\Gamma}(f)>0,$ so from the results
recalled in section \ref{sub: Pressure}, we deduce the existence
of $C>0$ such that for $T$ large enough, 
\[
\frac{1}{C}\e^{T\Pr_{\Gamma}(f)}\leq\sum_{\gamma\in\scP:\ell(\gamma)\in\cJ_{T}}\e^{\int_{\gamma}f}\leq C\e^{T\Pr_{\Gamma}(f)}\,,
\]
and this can be rewritten as 
\[
\frac{1}{C}\e^{T(\Pr_{\Gamma}^{u}+(1+\delta)|\Pr_{\Gamma}^{u}|)}\leq\sum_{\gamma\in\scP:\ell(\gamma)\in\cJ_{T}}\e^{\int_{\gamma}f}\leq C\e^{T(\Pr_{\Gamma}^{u}+(1+\delta)|\Pr_{\Gamma}^{u}|)}.
\]
Thus, there is $C>0$ such that for $T$ large enough, 
\[
\frac{1}{C}\e^{T\Pr_{\Gamma}^{u}}\leq\sum_{\gamma\in\scP:\ell(\gamma)\in\cJ_{T}}\e^{\int_{\gamma}-\frac{1}{2}J^{u}}\e^{(\ell(\gamma)-T)(1+\delta)|\Pr_{\Gamma}^{u}|}\leq C\e^{T\Pr_{\Gamma}^{u}}
\]
and we obtain for some other $C=C_{M,p,\delta}>0$ that 
\[
\frac{1}{C}\e^{T\Pr_{\Gamma}^{u}}\leq\sum_{\gamma\in\scP:\ell(\gamma)\in\cJ_{T}}\e^{\int_{\gamma}-\frac{1}{2}J^{u}}\leq C\e^{T\Pr_{\Gamma}^{u}}\,,
\]
again for $T$ sufficiently large. Finally in view of \eqref{eq: approx poincare J u}
and the fact that $\inf_{\gamma\in\scP}\ell(\gamma)\leq\ell^{\sharp}(\gamma)\leq T$,
we conclude that 
\[
\frac{1}{C}\left(\inf_{\gamma\in\scP}\ell(\gamma)\right)\e^{T\Pr_{\Gamma}^{u}}\leq\sum_{\gamma\in\scP:\ell(\gamma)\in\cJ_{T}}\frac{\ell^{\sharp}(\gamma)}{\sqrt{1-P_{\gamma}}}\leq CT\e^{T\Pr_{\Gamma}^{u}}.
\]
It is easy to see that the same inequality (again with a different
constant $C$) holds for periodic orbits in $\cJ_{T'}$. Combining
these inequalities, and noting that $\e^{-\mu\e^{T}}=o(\e^{T\Pr_{\Gamma}^{u}})$
if $T$ is sufficiently large, we end up with 
\[
\frac{1}{C}\left(\inf_{\gamma\in\scP}\ell(\gamma)\right)\e^{T\Pr_{\Gamma}^{u}}\leq\limsup_{\Xi\to\infty}\sup_{\lambda\in I_{\Xi}(T)}\Re\langle\Tr U,\e^{\i\lambda t}\phi^{(2)}\rangle\leq CT\e^{T\Pr_{\Gamma}^{u}}.
\]
We only need to take the logarithm of the above equation and divide
by $T$ to recover \eqref{eq: Spectral Invariant} in the limit $T\to\infty$,
and Theorem \ref{thm: spectral invariant} is proved since we can
replace $\Re\langle\Tr U,\e^{\i\lambda t}\phi^{(2)}\rangle$ by $\langle\Tr u,\cos(\lambda t)\phi^{(2)}\rangle$
thanks to Proposition \ref{prop: Trace Simplifiee}.

\section{Hyperbolic trapped set and long-time generating functions \label{sec: Classical dynamics}}

\subsection{Separation of periodic orbits}

In this section, we consider periodic orbits of the flow in $\cE^{*}M$
with length in the interval $\ell(\gamma)\in[T-1,T]$, where $T\leq\epsilon\log h^{-1}$.
For $\gamma$ a periodic orbit and $\eps>0$, denote by 
\[
\Theta(\gamma,\eps)\defeq\{\rho\in\cE^{*}M,\ \ d(\rho,\gamma)<\eps\}
\]
an open tubular $\eps-$neighborhood $\gamma$ in the adapted metric
for $\cE^{*}M$. We first state a fact which is essentially a direct
consequence of the well known Anosov shadowing Lemma \cite{KaHa95}. 
\begin{lem}
\label{lem: Separation Orbits}There are constants $\delta_{0},C>0$
depending only on $M$ and $p$ such that if $\gamma,\gamma'$ are
periodic orbits of the flow $\Phi^{t}$ in $\cE^{*}M$ with length
in the interval $[\ell,\ell+\delta_{0}]$, then 
\[
\Theta(\gamma,C\e^{-K_{+}\ell})\cap\Theta(\gamma',C\e^{-K_{+}\ell})=\emptyset\,,
\]
unless $\gamma'=-\gamma$, in which case the two sets are identical. \end{lem}
\begin{proof}
From Darboux theorem, there is a symplectic chart $(y,\eta)=(y^{1},\dots y^{n};\eta^{1},\dots\eta^{n})$
near every $\rho\in\cE^{*}M$ such that $\rho\equiv(0,0)$, $\eta_{1}=q-1$
and $\frac{\d}{\d y_{1}}(0,0)=H_{q}(\rho)$. The unit energy layer
is obtained by imposing $\eta_{1}=0$, so for $\eps,\ti\eps>0$ two
small enough constants, every $\rho\in\Gamma\subset\cE^{*}M$ has
a neighborhood in $\cE^{*}M$ diffeomorphic to the flow-box 
\[
\cF(\ti\eps,\eps)=]-\ti\eps,\ti\eps[\times U_{\rho}(\eps)\subset\R^{2n-1},
\]
where $]-\ti\eps,\ti\eps[$ denotes the local coordinate along the
flow direction, and $U_{\rho}(\eps)\subset\R^{2n-2}$ is a cross-section
of the flow made of an open ball of radius $\eps>0$. Let now $\gamma\in\cE^{*}M$
be a periodic orbit of length $\ell$, and we assume that for some
$\tau\in]0,\ti\eps[$, the tubular neighborhood $\Theta(\gamma,\eps)$
contains another orbit $\gamma'\neq\gamma$ of length $\ell+\tau$.
We will show that if $\eps,\ti\eps$ are sufficiently small, we will
get a contradiction. We can choose two points $\rho,\rho'\in\Theta(\gamma,\eps)$
such that $\rho\in\gamma$, $\rho'\in\gamma'$ and in the flow box
$\cF(\ti\eps,\eps)$ centered around $\rho=(0,0)$, the point $\rho'$
has coordinates $(0,x')$, and then $\Phi^{\ell}(\rho')=(-\tau,x').$
On the other hand, we know that 
\begin{equation}
d(\Phi^{\ell}(\rho),\Phi^{\ell}(\rho'))\leq C_{+}\e^{\ell K_{+}}d(\rho,\rho')\leq C_{+}\e^{\ell K_{+}}\eps.\label{eq: flow tube}
\end{equation}
Hence if $\eps\leq C_{+}^{-1}\e^{-\ell K_{+}}\ti\eps$, then $d(\Phi^{t}\rho,\Phi^{t}\rho')\leq\ti\eps$
for $0\leq t\leq\ell$. From these data, we will construct an infinite,
discrete $\ti\eps-$pseudo orbit of the flow near $\gamma$.

For this, let us divide the interval $[0,\ell+\tau]$ into subintervals
$[i_{0},i_{1}]\cup\cdots\cup[i_{N-1},i_{N}]$ such that $i_{0}=0$,
$i_{N}=\ell+\tau$, $i_{k+1}-i_{k}=\tau$ if $k\leq N-1$ and $i_{N}-i_{N-1}\leq\tau$.
We define a sequence of points $(z_{k})_{k\in\N}$ in $\cE^{*}M$,
and a sequence of associated times $(\delta_{k})_{k\in\N}$ by 
\[
z_{k}=\Phi^{(k\mod N)\tau}(\rho'),\quad\delta_{k}=\tau,\quad k\in\Z.
\]
By construction, $(z_{k},\tau)_{k\in\Z}$ is an $\ti\eps-$pseudo
orbit for the flow, as $d(z_{k+1},\Phi^{\tau}(z_{k}))=0$ if $k\neq N-1\mod N$,
and if $k=N-1\mod N$, we see that 
\[
d(\Phi^{\tau}(z_{k}),z_{k+1})=d(\Phi^{r}(\rho'),\rho')\leq\ti\eps,\qquad r=\tau-|i_{N}-i_{N-1}|.
\]
On the other hand, we have clearly $d(z_{k},\gamma')=0\leq\ti\eps$
for all $k\in\Z$. Recalling \eqref{eq: flow tube} and the choice
of $\eps$, we also see that $d(z_{k},\gamma)\leq\ti\eps$ for any
$k$. Finally, $(z_{k},\tau)_{k\in\Z}$ is an infinite $\ti\eps-$pseudo-orbit
which is $\ti\eps-$shadowed by $\gamma$ and $\gamma'$ which are
supposedly distinct. But the Anosov shadowing Lemma (\cite{KaHa95},
Chapter 18) ensures the existence of $\delta_{0}>0$ such that if
$\ti\eps\leq\delta_{0}$, then any $\ti\eps-$pseudo orbit is $\ti\eps-$shadowed
by a unique genuine orbit. Hence if we chose $\ti\eps\leq\delta_{0}$
we must have $\rho=\rho'$ and the Lemma is proved up to shrink $\delta_{0}$
further, so that no orbits multiple of each other can have lengths
in an interval of size $\delta_{0}$. 
\end{proof}

\subsection{Hamilton-Jacobi around Ehrenfest time\label{sub: Ham Jac}}

Consider the flow $\Phi^{t}:T^{*}M\to T^{*}M$ for $0\leq t\leq T$
where again $T\leq\epsilon\log h^{-1}$. In this section, examine
how the (local) Hamilton-Jacobi theory of generating functions for
the flow $\Phi^{t}$ apply when $T\to+\infty.$

Since we will work locally, we first need to control the size of open
sets evolved by the flow until such large times. As above, let $\rho_{0}\in\Gamma$.
From the growth of balls under the dynamics \eqref{eq: Growth Bowen},
we notice that if we consider $\eps'>0,$ $\eps(t)=\e^{-Lt}$ with
$L>0$ fixed but such that $B_{\rho_{0}}(\eps(t))\subset\cE^{*}M^{\delta}$
for all $t\geq0$, then 
\[
\diam\lp\Phi^{t}(B_{\rho_{0}}(\eps(t))\rp\leq\eps'
\]
if $C_{+}\e^{(K_{+}-L)t}\leq\eps'$, which can be satisfied for sufficiently
large times only if $L>K_{+}$. Hence if $L$ is sufficiently large,
then for any $t\in[T-1,T]$, $\Phi^{t}(B_{\rho_{0}}(\e^{-LT}))\subset T^{*}M$
can be parametrized by a single local coordinates patch in $\R^{2n}$.

We now recall a particular choice of coordinates near $\rho_{0}$
which is well adapted to the dynamics of the flow $\Phi^{t}$ for
times $t\leq\epsilon\log h^{-1}$, essentially built on classical
results of symplectic geometry, see for instance \cite{GrSj94}, Chapter
5 and 9. For $U\subset T^{*}M$ we denote by 
\[
\Graph\Phi^{t}|_{U}\defeq\{(\rho,\sigma),\ \rho\in U,\ \sigma=\Phi^{t}(\rho)\}\subset T^{*}M\times T^{*}M.
\]
Let $\eps_{h},t_{h}\in(0,1)$, possibly depending on $h$, and $0\leq t_{0}<T$
be such that $t_{0}+t_{h}\leq T$, and consider for $\rho_{0}=(y_{0},\eta_{0})\in\cE^{*}M$
the open set 
\begin{align*}
\cU(\rho_{0},t_{0},\eps_{h}) & =B_{\rho_{0}}(\eps_{h})\times\bigcup_{t_{0}-t_{h}<s<t_{0}+t_{h}}\Phi^{s}(B_{\rho_{0}}(\eps_{h}))\times]t_{0}-t_{h},t_{0}+t_{h}[\\
 & =\bigcup_{t_{0}-t_{h}<s<t_{0}+t_{h}}\Graph\Phi^{s}|_{B_{\rho_{0}}(\eps_{h})}\times]t_{0}-t_{h},t_{0}+t_{h}[\subset T^{*}M\times T^{*}M\times\R.
\end{align*}
Without loss of generality, for $\eps_{h},t_{h}>0$ sufficiently small,
we can perform a symplectic change of variables in $B_{\rho_{0}}(\eps_{h})$
such that the projection map 
\begin{equation}
\pi_{G}:\Graph\Phi^{t}|_{B_{\rho_{0}}(\eps_{h})}\owns(y,\eta,x,\xi,t)\mapsto(x,\eta,t)\,,\quad\Phi^{t}(y,\eta)=(x,\xi)\label{eq: projection graph}
\end{equation}
has an invertible differential at the point $(y_{0},\eta_{0},x_{0},\xi_{0},t_{0})$,
and this is equivalent to say that the upper-left block of $d\Phi^{t_{0}}(y_{0},\eta_{0})$
is invertible. By the local inverse theorem, $\pi_{G}$ is then invertible
near $(y_{0},\eta_{0},x_{0},\xi_{0},t_{0})$ but we want to estimate
quantitatively the size of such a neighborhood. More precisely, we
have : 
\begin{prop}
\label{prop: Local projection} There is $L_{0}>0$ such that for
any $\rho_{0}\in\cE^{*}M$, $t_{0}>0$ with $t_{0}+h^{\epsilon L_{0}}\leq\epsilon\log h^{-1}$,
there are local coordinates on $U_{0}=B_{\rho_{0}}(h^{\epsilon L_{0}})\subset\cE^{*}M^{\delta}$
and near $\Phi^{t_{0}}(U_{0})$ with the property that the projection
map 
\[
\Graph\Phi^{t}|_{U_{0}}\owns(y,\eta,x,\xi)\mapsto(x,\eta)\,,\quad\Phi^{t}(y,\eta)=(x,\xi)
\]
is a diffeomorphism for $|t-t_{0}|\leq h^{\epsilon L_{0}}$.\end{prop}
\begin{proof}
Set $\Phi^{t_{0}}(y_{0},\eta_{0})=(x_{0},\xi_{0})$. In a local chart
in $T^{*}\R^{n}\times T^{*}\R^{n}$, let $U\subset\R^{2n+1}$ be an
open ball of radius $r$ around $(y_{0},\eta_{0},t_{0})$ and $V\subset\R^{2n+1}$
an open neighborhood of $(x_{0},\eta_{0},t_{0})$. Here if $z,z'\in U$,
there are $c,c'>0$ depending only on $M$ and $p$ such that if $\|\cdot\|$
denotes the usual Euclidian norm on $\R^{2n+1}$ and $d$ the adapted
metric on $T^{*}M$, then 
\begin{equation}
c\|z-z'\|\leq d(\rho,\rho')+|t-t'|\leq c'\|z-z'\|.\label{eq: comparaison des normes}
\end{equation}
What we want is an estimate on $r$ for which $\boldsymbol{F}\defeq\pi_{G}\circ\Phi^{t}:U\to V$
is one-to-one on its image, knowing that $d_{\rho_{0},t_{0}}\boldsymbol{F}\in GL(2n+1)$.
We will follow the lines of the usual proof of the local inverse theorem.
Let us write $z=(y,\eta,t)$ and $\zeta=(x,\eta,t)$. Without loss
of generality, assume that $z_{0}=(y_{0},\eta_{0},t)=0$, $\zeta_{0}=(x_{0},\eta_{0},t_{0})=0$.
In local coordinates, it makes sense to write, for $\zeta$ fixed
near $\boldsymbol{F}(z_{0})$ : 
\[
\boldsymbol{G}(z)=(d_{z_{0}}\boldsymbol{F})^{-1}(\boldsymbol{F}(z)-\zeta)-z,\qquad z\in U.
\]
Note that $d_{z_{0}}\boldsymbol{G}=0$, so by continuity there is
$r'>0$ such that $\|z\|\leq r'\Rightarrow\|d_{z}\boldsymbol{G}\|\leq1/4$.
To estimate $r'$, we first recall some standard flow estimates. Since
$d_{t}\Phi^{t}=H_{p}\Phi^{t}$, estimates using Gronwall inequality
\cite[Lemma 11.11]{Zwo10_semiAMS} show that for $t\geq0$, the derivatives
of $\Phi^{t}$ with respect to the phase space and time variables
grow at most exponentially in time, namely if $(x,\xi)\in T^{*}M$
satisfies $p(x,\xi)\leq R$ for some $R>0$, then there is constant
$C_{M,p,R}$ such that for any multi-index $\alpha\in\N^{2n+1}$,
\begin{equation}
\sup_{(x,\xi)\in T^{*}M:\, p(x,\xi)\leq R}|\d^{\alpha}\Phi^{t}(x,\xi)|\leq C_{\alpha}\e^{C_{M,p,R}|\alpha||t|}\,,\qquad C_{\alpha}>0,\ t\in\R.\label{eq: flow estimates}
\end{equation}
This implies using a Taylor expansion of $d\boldsymbol{F}$ around
$z_{0}$ together with \eqref{eq: flow estimates} that for some $C>0$,
we have $\|d_{z}\boldsymbol{F}-d_{z_{0}}\boldsymbol{F}\|\leq C\e^{CT}$,
$\|(d_{z_{0}}\boldsymbol{F})^{-1}\|\leq C\e^{CT}$ and finally, up
to increase $C$ again, 
\[
\|z\|\leq r'\defeq C^{-1}\e^{-CT}\Rightarrow\|d_{z}\boldsymbol{G}\|\leq\frac{1}{4}.
\]
Now $\|\boldsymbol{G}(z)\|\leq\|d_{z_{0}}\boldsymbol{F}^{-1}\zeta\|+\frac{1}{4}\|z\|<1/2$
if, say, $\|d_{z_{0}}\boldsymbol{F}^{-1}\zeta\|\leq1/4$. This is
for instance satisfied as soon as $\zeta\in B(0,\frac{1}{4C}\e^{-CT})$.
Hence $\boldsymbol{G}$ is contracting on a ball of radius $r'$ and
the unique fixed point satisfies $\boldsymbol{F}(z)=\zeta$. We then
obtain an invertible mapping from a ball $\{\|z\|<r\}$ on its image
as soon as $r<r'$ and $\|\zeta\|=\|\boldsymbol{F}(z)\|\leq r'/4$
which is achieved for $r=L^{-1}\e^{-LT}$ and some $L>C$. Hence for
\[
U=B(z_{0},r)\subset\R^{2n+1},\quad\boldsymbol{F}:U\to\boldsymbol{F}(U)
\]
is one-to-one and it is actually a diffeomorphism. Coming back to
$T\leq\epsilon\log h^{-1}$, this finally implies in view of \eqref{eq: comparaison des normes}
the existence of $L_{0}>0$ depending only on $p$ and $M$ such that
if $h$ is small enough, the projection $\pi_{G}$ defined in \eqref{eq: projection graph}
is invertible at $(\rho,\Phi^{t}(\rho),t)\equiv(y,\eta,x,\xi,t)$
if 
\[
\rho\in T^{*}M,\quad d(\rho_{0},\rho)\leq\e^{-L_{0}T}\leq\eps_{h}=\e^{-L_{0}\epsilon\log h^{-1}}=h^{\epsilon L_{0}},\quad|t-t_{0}|\leq t_{h}=h^{\epsilon L_{0}}.
\]

\end{proof}
The open set $B_{\rho_{0}}(\eps_{h})$ is thens of size $\cO(h^{L_{0}\epsilon})$
in phase space. It follows that if $\epsilon>0$ is small enough,
we will be able to (semiclassically) microlocalize operators in such
neighborhoods and use the semiclassical calculus.

A classical consequence of the preceding proposition is the existence
of a local generating function $\varphi_{0}(t,x,\eta)$ for the flow
$\Phi^{t}$ defined near $\rho_{0}\in\Gamma$ up to times $\epsilon\log h^{-1}$.
For $(y,\eta)\in B_{\rho_{o}}(h^{\epsilon L_{0}})$, and $t=t_{0}$,
Proposition \ref{prop: Local projection} implies that there is a
first generating function $\ti\varphi_{0}(x,\eta)$ such that 
\[
\Phi^{t_{0}}:(\d_{\eta}\ti\varphi_{0}(x,\eta),\eta)\mapsto(x,\d_{x}\ti\varphi_{0}(x,\eta))
\]
and using $(x,\eta,t)$ as coordinates on $\cU(\rho_{0},t_{0},\eps)$,
we can define 
\[
\varphi_{0}(t,x,\eta)=\ti\varphi_{0}(x,\eta)-\int_{t_{0}}^{t}q(x,\xi(s,x,\eta))ds
\]
where as above, $\Phi^{t}(y,\eta)=(x,\xi)$. We readily check that
$\varphi_{0}(t,x,\eta)$ is now a generating function for the flow
for times $|t-t_{0}|\leq h^{\epsilon L_{0}}$,

\begin{equation}
y=\d_{\eta}\varphi_{0}(t,x,\eta),\ \xi=\d_{x}\varphi_{0}(t,x,\eta),\quad\Phi^{t}:(\d_{\eta}\varphi_{0}(t,x,\eta),\eta)\mapsto(x,\d_{x}\varphi_{0}(t,x,\eta))\label{eq: Flow Gen Fn}
\end{equation}
and furthermore $\varphi_{0}$ satisfies the Hamilton-Jacobi equation
\begin{equation}
\d_{t}\varphi_{0}(t,x,\eta)+q(x,\d_{x}\varphi_{0}(t,x,\eta))=0.\label{eq: ham Jac}
\end{equation}

The above construction can be done in a particular choice of coordinates
in $\cU(\rho_{0},t_{0},h^{\epsilon L_{0}})$ which is very useful
when considering the hyperbolic dynamics at stake near the trapped
set. In \cite{NoZw09_1}, Lemmas 4.3 and 4.4, Nonnenmacher and Zworski
show the existence of a coordinate chart $(y,\eta)=(y^{1},\dots,y^{n},\eta^{1},\dots,\eta^{n})$
near $\rho_{0}\equiv(y_{0},\eta_{0})\in\Gamma$ and $(x,\xi)=(x^{1},\dots,x^{n},\xi^{1},\dots,\xi^{n})$
near $(x_{0},\xi_{0})\equiv\Phi^{t_{0}}(\rho_{0})$ for which the
projection $\pi_{G}$ has a bijective differential near $(\rho_{0},\Phi^{t_{0}}(\rho_{0}),t_{0})$
and which is furthermore well adapted to the dynamics, in the sense
that 
\[
\frac{\d}{\d y^{1}}(\rho_{0})=H_{q}(\rho_{0}),\ \eta^{1}=q-1,\quad\frac{\d}{\d\xi^{1}}(\Phi^{t_{0}}(\rho_{0}))=H_{q}(\Phi^{t_{0}}(\rho_{0})),
\]
and moreover the unstable and stable spaces at $(y_{0},\eta_{0})$
are given by 
\[
E_{\rho_{0}}^{u}=\Span\left(\frac{\d}{\d y^{2}}(y_{0},\eta_{0}),\dots,\frac{\d}{\d y^{n}}(y_{0},\eta_{0})\right),\quad E_{\rho_{0}}^{s}=\Span\left(\frac{\d}{\d\eta^{2}}(y_{0},\eta_{0}),\dots,\frac{\d}{\d\eta^{n}}(y_{0},\eta_{0})\right),
\]
and similar equations for the unstable and stable subspaces at $\Phi^{t_{0}}(\rho_{0})$.
In these charts, we have time and energy coordinates for the index
1, and ``stable'' and ``unstable'' coordinates for the indices
$\geq2$. We usually write $u=(y^{2},\dots,y^{n})$ and $s_{0}=(\eta^{2},\dots,\eta^{n})$
for the stable and unstable coordinates. Now consider in these coordinates
the following family of horizontal Lagrangian leaves 
\[
\Lambda_{\eta}=\{(y,\eta),\ y\in\pi B_{\rho_{0}}(h^{\epsilon L_{0}})\},\quad(y,\eta)\in B_{\rho_{0}}(h^{\epsilon L_{0}}).
\]
Note that these Lagrangian manifolds are isoenergetic since $\Lambda_{\eta}\subset p^{-1}(1+\eta^{1})\subset T^{*}M$.
They also have nice properties when evolved by the flow $\Phi^{t}$,
and we extract from \cite{NoZw09_1} the ones which will be used in
the present article. 
\begin{prop}
\label{prop: NZ 5}(adapted from \cite{NoZw09_1}, Section 5) Let
$\rho_{0}\in\Gamma$, and $\Lambda_{\eta}$ the above family of Lagrangian
manifolds in the coordinate system adapted to the dynamics. If $0\leq t\leq T$,
the map 
\[
\pi\circ\Phi^{t}|_{\Lambda_{\eta}}:\begin{cases}
\pi\Lambda_{\eta}\to\pi\Phi^{t}(\Lambda_{\eta})\\
y\mapsto y(t)
\end{cases}
\]
is well-defined and invertible. Moreover, the differential matrix
$\frac{\d y}{\d y(t)}$ is uniformly bounded in time: 
\begin{equation}
\exists C_{M,p}>0,\ \forall t\in[0,T],\quad\lno\frac{dy}{dy(t)}\rno\leq C_{M,p}\label{eq: Uniform bounds horiz}
\end{equation}
and satisfies the following estimate on its domain of definition:
\[
C\e^{-\lambda_{t}^{+}(\rho_{0})}\leq\lva\det\frac{dy}{dy(t)}\rva\leq\frac{1}{C}\e^{-\lambda_{t}^{+}(\rho_{0})},\qquad0\leq t\leq T
\]
where $C=C(M,p)>0$ and the unstable Jacobian $\lambda_{t}^{+}$ has
been defined in Section \ref{sec: Anosov}. 
\end{prop}
This proposition is similar to what is known as the inclination lemma
\cite{KaHa95}: in the chosen coordinate system, these horizontal
Lagrangian are transverse to the stable foliation, which explains
their stretching along the unstable manifold when evolved by the flow.
This property will be of crucial importance when determining estimates
in $C^{k}$ norm for the symbol of an oscillatory integral representation
of $U(t)$ up to time $T$.

Our dynamical setup near the trapped set is then as follows. Let $\rho_{0}\in\Gamma$,
and consider an orbit $\gamma$ issuing from $\rho_{0}$ and of length
$\ell$, with $\ell\leq\epsilon\log h^{-1}$. Let $V_{0}=B_{\rho_{0}}(h^{\epsilon L_{0}})\subset T^{*}M$
be an open neighborhood of $\rho_{0}$ of size $h^{L_{0}\epsilon}$
for some $L_{0}>0$ large enough, depending on $M$ and $p$ but independent
of $\epsilon$, so that Proposition \ref{prop: Local projection}
applies. In particular, we can find a sequence of times $(t_{k})_{0\leq k\leq K}$
with $t_{0}=0$, $t_{K}=\ell$, $K=\cO(h^{-\epsilon L_{0}}\log h^{-1})$
and a chain of neighborhoods 
\begin{equation}
V_{k}=\bigcup_{t\in[t_{k}-h^{\epsilon L_{0}},t_{k}+h^{\epsilon L_{0}}]}\Phi^{t}(V_{0})\label{eq: open V}
\end{equation}
with a coordinate set $J_{k}:V_{0}\times V_{k}\to\R^{2n}\times\R^{2n}$
such that there is a local generating function $\varphi^{k}(t,x,\eta)$
for the flow for times near $t_{k}$, more precisely: 
\begin{equation}
\Phi^{t}:\begin{cases}
V_{0}\mapsto\Phi^{t}(V_{0}),\ \ t\in[t_{k}-h^{\epsilon L_{0}},t_{k}+h^{\epsilon L_{0}}]\\
(\d_{\eta}\varphi^{k}(t,x,\eta),\eta)\to(x,\d_{x}\varphi^{k}(t,x,\eta)
\end{cases}\label{eq: varphi k}
\end{equation}
and $\varphi^{k}$ also satisfies the Hamilton-Jacobi equation \eqref{eq: ham Jac}.

\section{Proof of Proposition \ref{prop: Long Time Trace Formula}\label{sec: Big Stat Phase} }

To prove Proposition \ref{prop: Long Time Trace Formula}, we will
adopt the same method as the original one developed by Duistermaat
and Guillemin \cite{DuGu75}, namely we will represent the Schwartz
kernel of $U(t)$ locally by oscillatory integrals, and perform a
stationary phase expansion in the integral on the right hand side
of \eqref{eq: Integral Trace}. There is however an important difference
with the situation considered in \cite{DuGu75}, since our test function
has support localized around Ehrenfest times $\leq\epsilon\log\lambda$.
In this section, without further notifications, $C$ will denote a
positive constant depending only on $M$ and $p$.

\subsection{Microlocalization and local integral representations\label{sub: Local-integral-representations}}

Let $\rho_{0}\in\Gamma$, and $\gamma$ the orbit issuing from $\rho_{0}$
for positive times. Let $\chi\in C_{0}^{\infty}(T^{*}M)$ being equal
to one near $\rho_{0}$. We assume that $\chi$ is supported in a
neighborhood of $\rho_{0}$ of size $\cO(h^{\epsilon L_{0}})$ where
$L_{0}>0$ is as in the preceding section, namely it is large enough
so that Propositions \ref{prop: Local projection} and \ref{prop: NZ 5}
hold true. We again stress that $L_{0}$ is independent of $\epsilon$,
so if $\epsilon<\delta L_{0}^{-1}$, we can ensure that $\chi\in S_{\delta}^{0,0}(T^{*}M)$
for some $\delta<1/2$. Defining $V_{0}=\supp\chi$, we can find local
generating function for the flow restricted to $V_{0}$, up to times
$T\leq\epsilon\log h^{-1}$. Let us call $A_{0}^{w}=\Op_{h}(\chi)$.
The goal of this section is to build a local integral representation
of $U(t)A_{0}^{w}$ for times $t\in[0,T]$.

If $u\in L^{2}(M)$, we can use a partition of unity and assume without
loss of generality that $u$ is compactly supported, so that in a
local chart, 
\[
A_{0}^{w}u(y)=\frac{1}{(2\pi h)^{n}}\int\e^{\frac{\i}{h}\scal{y-z}{\eta}}\chi(\frac{y+z}{2},\eta)u(z)d\eta dz=\int\e^{\i\frac{\scal y{\eta}}{h}}u_{\eta}(y)d\eta
\]
where 
\[
u_{\eta}(y)\defeq\frac{1}{(2\pi h)^{n}}\int\e^{-\frac{\i}{h}\scal z{\eta}}\chi(\frac{y+z}{2},\eta)u(z)dz.
\]
In this way, we have decomposed $A_{0}^{w}u$ into a sum of ``momentum''
Lagrangian states $\e^{\i\frac{\scal y{\eta}}{h}}u_{\eta}(y)$ depending
on a parameter $\eta$. Note that because $\chi\in C_{0}^{\infty}(T^{*}M)$,
the parameter $\eta$ belongs to a compact set so this integral is
always well defined. These Lagrangian states are purely horizontal,
their Lagrangian manifold being given by 
\[
\Lambda_{\eta}^{0}=\{(y,\eta)\in T^{*}M,\ y\in\supp u_{\eta}\}.
\]
Hence $V_{0}=\supp\chi$ is foliated by such horizontal Lagrangian
manifolds, and this foliation is used on the microlocal level in the
above decomposition of $A_{0}^{w}u(y)$. The important fact here is
that in the coordinates adapted to the dynamics around $\rho_{0}$,
the manifolds $\Lambda_{\eta}^{0}$ are transverse to the stable foliation
and the inclination property applies when evolving these manifolds
under the Hamiltonian flow.

For each time-window $t\in[\tau_{k},\tau_{k+1}]$ we will find an
oscillatory integral representation of the operator $U(t)A_{0}^{w}$,
which will be denoted by $U_{k}(t)A_{0}^{w}$, such that 
\[
\Tr\left(\Pi(U(t)-U_{k}(t))A_{0}^{w}\right)=\cO(h^{\infty}).
\]
The next proposition makes this much more precise, and is the key
technical tool needed to prove our main theorem. 
\begin{prop}
\label{prop: The big one}Let $\gamma$ be an orbit of length $\ell\leq T$
starting at $\rho_{0}\in\cE^{*}M$. Let $\chi\in C_{0}^{\infty}(T^{*}M)$,
$V_{0}=\supp\chi$ and $A_{0}^{w}$ be as above. Let $(V_{k})_{0\leq k\leq K}$,
$(\varphi^{k})_{0\leq k\leq K}$ be as in \eqref{eq: open V} and
\eqref{eq: varphi k}. For any $N\in\N$ large enough, $0\leq k\leq K$
and $t\in[t_{k}-h^{\epsilon L_{0}},t_{k}+h^{\epsilon L_{0}}]$, there
is a sequence of local Fourier integral operators $U_{k}(t)$ and
$\cR_{k}^{N}(t)$ such that 
\begin{equation}
U_{k}(t)A_{0}^{w}=U(t)A_{0}^{w}+\cR_{k}^{N}(t)A_{0}^{w}\quad\mbox{and}\quad\Tr(\Pi\cR_{k}^{N}(t)A_{0}^{w})=\cO(h^{N/3}).\label{eq:Approx FIO}
\end{equation}
Furthermore, the operators $U_{k}(t)$ are of the form 
\begin{equation}
U_{k}(t)A_{0}^{w}u(x)=\int\e^{\frac{\i}{h}\varphi^{k}(t,x,\eta)-\scal y{\eta}}a_{h}^{k}(t,x,\eta)(A_{0}^{w}u)(y)dyd\eta,\quad t\in[t_{k}-h^{\epsilon L_{0}},t_{k}+h^{\epsilon L_{0}}],\label{eq: FIO Local Int}
\end{equation}
where the symbol $a_{h}^{k}$ is compactly supported and admits the
asymptotic expansion: 
\[
a_{h}^{k}(t,x,\eta)=\sum_{j=0}^{N-1}h^{j}a_{j}^{k}(t,x,\eta),\qquad N\geq1.
\]
The principal symbol $a_{0}^{k}$ is given by 
\[
a_{0}^{k}=\i^{\sigma_{k}}|\det\d_{x\eta}^{2}\varphi^{k}|^{\fr12}
\]
where $\sigma_{k}$ is an integer of Maslov type (see \eqref{eq: principal symbol}
below). 
\end{prop}

\subsubsection{Sketch of proof of Proposition \ref{prop: The big one} }

The proof of the preceding proposition relies on a standard WKB method
adjusted for our purpose. Let $\gamma(t)\subset\cE^{*}M$ with $0\leq t\leq\epsilon\log h^{-1}$.
From Section \ref{sub: Ham Jac}, we have at our disposal a sequence
of open patches in $T^{*}M$ that cover $\gamma(t)$, centered at
$t_{0},\dots,t_{K}$ and for which there is a generating function
of the flow in the sense of \eqref{eq: Flow Gen Fn}. If $A_{0}^{w}$
is an operator that microlocalizes in the first patch around $\rho_{0}=\gamma(0)$,
then $U(t_{i})A_{0}^{w}$ is microlocalized in the patch centered
around $\Phi^{t_{i}}(\rho_{0})$. In each patch around $\gamma(t_{i})$,
the WKB procedure can produce a local integral representation for
$U(t_{i})$ thanks to the generating function and the usual Duistermaat-Hörmander
transport equations for the symbols. By identifying two such representations
at the intersection of two consecutive patches modulo a remainder
term, we can obtain at the end a local integral representation valid
up to time $\epsilon\log h^{-1}$. For our purpose to take a trace,
the principal symbol must be computed exactly for all times (namely,
including the Maslov factors), and the remainders must also be small
in the trace class norm. The main difficulty in this construction
is due to the fact that $t$ can be of order $\epsilon\log h^{-1}$,
which means that all the symbols and remainder terms must be controlled
as $t\sim\epsilon\log h^{-1}$. In particular, the hyperbolicity of
the trapped set requires the use of symbol classes $S_{\epsilon C}^{0,0}$
for a given constant $C>0$ that depends only on $M$ and the symbol
$p$, as it has already been noted above. But provided that if $\epsilon$
is small enough, we are always in tractable symbol classes $S_{1/2^{-}}^{0,0}$
and microlocal calculus is then always available.

\subsubsection{Changing symbols and phase functions along a bicharacteristic}

The content of the proposition is standard if $k=0$ and $t=\cO(1)$,
see \cite{Zwo10_semiAMS}, Chapter 10. Writing $u_{0}=A_{0}^{w}u$
we have 
\[
U(t)u_{0}=\frac{1}{(2\pi h)^{n}}\int\e^{\frac{\i}{h}(\varphi^{0}(t,x,\eta)-\langle y,\eta\rangle)}a^{0}(t,x,\eta,h)u_{0}(y)dyd\eta
\]
where 
\begin{equation}
a^{0}\sim\sum_{j=0}^{\infty}h^{j}a_{j}^{0}\,,\quad a_{0}^{0}(t,x,\eta)=|\det\d_{x\eta}^{2}\varphi^{0}(t,x,\eta)|^{\fr12},\quad\sigma_{0}=0,\label{eq: a_0 first}
\end{equation}
and the symbols $a_{j}^{0}$ are smooth and compactly supported in
all variables (in particular, $x\in\pi\Phi^{t}(V_{0})$, $\pi:T^{*}M\to M$).
We take 
\[
U_{0}(t)u_{0}=\frac{1}{(2\pi h)^{n}}\int\e^{\frac{\i}{h}(\varphi^{0}(t,x,\eta)-\langle y,\eta\rangle)}\sum_{j=0}^{N-1}h^{j}a_{j}^{0}(t,x,\eta,h)u_{0}(y)d\eta dy
\]
for arbitrary $N\in\N$. As a result, there is a (semiclassical) local
Fourier integral operator $\cZ_{0}^{N}(t)$ such that 
\begin{equation}
(hD_{t}+hQ)U_{0}(t)A_{0}^{w}=\cZ_{0}^{N}(t)A_{0}^{w},\quad t\in[-h^{\epsilon L_{0}},h^{\epsilon L_{0}}].\label{eq: Chain 2-1}
\end{equation}
The operator $\cZ_{0}^{N}(t)$ arises from the usual transport equations,
it depends on the symbols $a_{0}^{0}(t),\dots,a_{N-1}^{0}(t)$ and
their derivatives and since its symbol is uniformly of order $\cO(h^{N})$
on its support, we get that $\Tr(\Pi\cZ_{0}^{N}(t)A_{0}^{w})=\cO(h^{N-n})=\cO(h^{N/2})$
if $N$ is sufficiently large. Duhamel formula gives 
\[
\cR_{0}^{N}(t)=\int_{0}^{t}U(t-s)\cZ_{0}^{N}(s)ds
\]
and \eqref{eq:Approx FIO} follows from the trace estimate of $\Pi\cZ_{0}^{N}(t)A_{0}^{w}$
.

We will now construct $\cU_{k}$ and $\cR_{k}^{N}$ by induction,
assuming that $U_{k-1},\cR_{k-1}^{N}$ have been constructed for time
$t\in[t_{k-1}-h^{\epsilon L_{0}},t_{k-1}+h^{\epsilon L_{0}}]$. For
this we will regularly change the phase function along the chain of
neighborhoods $V_{k}$ defined above, so that locally in $t,x,\eta$,
we still have a good oscillatory integral representation. For all
$k\leq K$, the times $t_{k}$ can be chosen so that $V_{k-1}$ and
$V_{k}$ always intersect, for instance around a point $\rho_{k-1}$
such that 
\[
\rho_{k-1}=\Phi^{\tau_{k-1}}(\rho_{0}),\qquad\tau_{k-1}\in[t_{k-1}-h^{\epsilon L_{0}},t_{k-1}+h^{\epsilon L_{0}}].
\]
Note that $\Phi^{\tau_{k-1}}(V_{0})\subset V_{k-1}\cap V_{k}$.

We start by defining the operator $U_{k}(\tau_{k-1})$, which will
be our initial data from which $U_{k}(\tau_{k-1}+s)$ for $s\geq0$
will be constructed. For $N>1$, our induction hypothesis allows to
write

\[
(U_{k-1}(t)u_{0})(x)=\frac{1}{(2\pi h)^{n}}\int\e^{\frac{\i}{h}(\varphi^{k-1}(t,x,\eta)-\langle y,\eta\rangle)}a_{h}^{k-1}(t,x,\eta,h)u_{0}(y)dyd\eta
\]
for $t\in[t_{k-1}-h^{\epsilon L_{0}},t_{k-1}+h^{\epsilon L_{0}}]$,
where 
\[
a_{h}^{k-1}=\sum_{j=0}^{N-1}h^{j}a_{j}^{k-1}
\]
is compactly supported in $(x,\eta)$. We want first to change the
phase function in the operator $U_{k-1}(\tau_{k-1})$, and use $\varphi^{k}$
instead of $\varphi^{k-1}$: this is possible since at time $\tau_{k-1}$,
these phase functions are both generating functions for the flow with
the same initial conditions.

To change the phase function from $\varphi^{k-1}$ to $\varphi^{k}$,
we apply the original method developed by Hörmander (reduction of
the number of fibre variables, fiber-preserving mappings and adjunction
of quadratic forms). This process is fairly long, so we refer the
reader to the thorough exposition in the original article \cite{Hor71_Fio1},
in particular Sections 3.1 and 3.2 -- see also \cite{Dui96}. This
is also precisely the method used in \cite{DuGu75}, p. 68. As a result,
there is a sequence of symbols $(\ti a_{j}^{k})_{0\leq j<N}$ which
are determined by equations of the form 
\[
\ti a_{j}^{k}(x,\eta)=\sum_{\nu=0}^{j}Z_{j,\nu}^{k}a_{j-\nu}^{k-1}(\tau_{k-1},x,\eta)
\]
where the differential operators $Z_{j,\nu}^{k}$ are of order $2\nu$
for $0\leq\nu\leq j<N$ with coefficients independent of $h$ and
uniformly bounded with respect to $k$, such that we can define a
local Fourier integral $\ti U_{k}$ by the formula 
\[
\ti U_{k}(x,y)=\frac{1}{(2\pi h)^{n}}\int\e^{\frac{\i}{h}(\varphi^{k}(\tau_{k-1},x,\eta)-\langle y,\eta\rangle)}\sum_{j=0}^{N-1}h^{j}\ti a_{j}^{k}(x,\eta)d\eta,
\]
which defines locally the same distribution as $U_{k-1}(\tau_{k-1})$
up to a remainder of arbitrary order. In particular, there is another
Fourier integral operator $\cS_{k-1}^{N}$ with a compactly supported
symbol $s_{k-1}^{N}(x,\eta,h)$ that satisfies 
\begin{equation}
\lp\ti U_{k}-U_{k-1}(\tau_{k-1})\rp A_{0}^{w}=\cS_{k-1}^{N}A_{0}^{w}\label{eq: def SkN}
\end{equation}
and its symbol $s_{k-1}^{N}$ obey estimates of the form: 
\begin{equation}
\|s_{k-1}^{N}\|_{C^{0}}\leq h^{N}\sum_{j=0}^{N-1}C_{j}\|a_{j}^{k-1}\|_{C^{2(N-j)+m}},\quad0\leq m\leq n+1,\quad C_{j}>0.\label{eq: SNk symbol}
\end{equation}
The celebrated transition equation for the principal symbols given
by 
\begin{align}
\ti a_{0}^{k} & =\e^{\frac{\i\pi}{4}(\sgn\d_{\eta}^{2}\varphi^{k-1}-\sgn\d_{\eta}^{2}\varphi^{k})}a_{0}^{k-1}(\tau_{k-1})\sqrt{D_{k,k-1}},\label{eq: Transition Principal}
\end{align}
where $\sqrt{D_{k,k-1}}$ is a positive factor which has an interpretation
as a quotient of half densities on the Lagrangian manifold $\Lambda\subset T^{*}(M\times M)$
formed by the local graph of $\Phi^{\tau_{k}}$, where $T^{*}(M\times M)$
is endowed with the usual twisted symplectic form \cite[Section 3.2]{Hor71_Fio1}.
In particular, since $a_{0}^{0}>0$ we have 
\[
\ti a_{0}^{k}=|\ti a_{0}^{k}|\prod_{i=1}^{k}\e^{\frac{\i\pi}{4}(\sgn\d_{\eta}^{2}\varphi^{i-1}-\sgn\d_{\eta}^{2}\varphi^{i})}.
\]
We now simply set 
\[
U_{k}(\tau_{k-1})\defeq\ti U_{k},\qquad a_{j}^{k}(\tau_{k-1},x,y)\defeq\ti a_{j}^{k}(x,\eta),\quad0\leq j\leq N-1.
\]
We have then a transition at time $\tau_{k-1}$ given by 
\begin{equation}
U_{k}(\tau_{k-1})A_{0}^{w}=U_{k-1}(\tau_{k-1})A_{0}^{w}+\cS_{k-1}^{N}A_{0}^{w}\label{eq: Chain 1}
\end{equation}
where $U_{k}(\tau_{k-1})$ is exactly of the form \eqref{eq: FIO Local Int}.

\subsubsection{Computation of the symbol from the transport equations}

It remains to define $U_{k}(\tau_{k-1}+t)$ for $\tau_{k-1}+t\in[t_{k}-h^{\epsilon L_{0}},t_{k}+h^{\epsilon L_{0}}]$.
In practice, we will use the symbol $\ti a^{k}=a^{k}(\tau_{k-1})$
as an initial data, and show that there is a local Fourier integral
operator $\cZ_{k}^{N}(t)$ such that we have the following evolution
problem for the unknown symbol $a^{k}(t)$: 
\[
\begin{cases}
(hD_{t}+hQ)U_{k}(t)u_{0}=\cZ_{k}^{N}(t)u_{0}\\
U_{k}(\tau_{k-1})u_{0}=U_{k-1}(\tau_{k-1})u_{0}+\cS_{k-1}^{N}u_{0}.
\end{cases}
\]
The method to solve these equations is standard \cite{Dui96,Zwo10_semiAMS},
but note however that in our case $t$ can ultimately depends on $h$,
so we have to control the whole process in time as well. To estimate
$(hD_{t}+hQ)U_{k}$, we write 
\[
\psi^{k}(t,x,y,\eta)=\varphi^{k}(t,x,\eta)-\scal y{\eta}
\]
and compute: 
\[
hD_{t}\e^{\frac{\i}{h}\psi^{k}(t,x,y,\eta)}a^{k}(t,x,\eta)=\d_{t}\varphi^{k}\e^{\frac{\i}{h}\psi^{k}(t,x,y,\eta)}a^{k}(t,x,\eta)+\e^{\frac{\i}{h}\psi^{k}(t,x,y,\eta)}\frac{h}{\i}\d_{t}a^{k}(t,x,\eta).
\]
From now on we fix $\eta$ and $y$ as parameters and consider $hQ\e^{\frac{\i}{h}\psi^{k}}a^{k}$.
Let us write 
\[
hQ=q^{w}+h^{2}q_{2}^{w}
\]
where the exponent $w$ is a shorthand to denote the Weyl quantization,
$q=\sqrt{p}\in S_{0}^{1,0}$ is the principal symbol of $hQ$, and
$q_{2}\in S_{0}^{-1,0}$. Since $q$ is homogeneous of order 1 in
$\xi$, a straightforward computation using the Weyl quantization
of linear symbols gives 
\begin{align*}
q^{w}\e^{\frac{\i}{h}\psi^{k}}a^{k} & =\e^{\frac{\i}{h}\psi^{k}}\left(q(x,\d_{x}\varphi^{k})+\frac{h}{\i}(X+\frac{1}{2}\div X)\right)a^{k}
\end{align*}
where 
\[
X=\sum_{j=1}^{n}\partial_{\xi_{j}}q(x,\partial_{x}\varphi^{k})\partial_{x_{j}}.
\]
Taking into account the eikonal equation \eqref{eq: ham Jac}, we
have altogether 
\[
(hD_{t}+hQ)\e^{\frac{\i}{h}\psi^{k}}a^{k}=\e^{\frac{\i}{h}\psi^{k}}(hD_{t}a^{k}+\frac{h}{\i}(X+\frac{1}{2}\div X)a^{k}+h^{2}q_{2}^{w}a_{k}).
\]
The last term is of lower order term compared to $h(D_{t}a^{k}-\i(X+\frac{1}{2}\div X)a^{k})$,
so we can impose the Duistermaat-Hörmander transport equations that
must be satisfied by the amplitude $a^{k}$: 
\[
\begin{cases}
(\d_{t}+X+\frac{1}{2}\div X)a_{0}^{k}=0\\
(\d_{t}+X+\frac{1}{2}\div X)a_{j}^{k}=\frac{h}{\i}q_{2}^{w}a_{j-1}^{k}
\end{cases}
\]
with initial conditions 
\[
a_{j}^{k}(\tau_{k-1})=\sum_{\nu=0}^{j}Z_{j,\nu}^{k}a_{j-\nu}^{k-1}(\tau_{k-1}).
\]
If the above transport equations are satisfied, we end up with 
\begin{align}
(hD_{t}+hQ)\left(\frac{1}{(2\pi h)^{n}}\int\e^{\frac{\i}{h}\psi^{k}(t,x,y,\eta)}\sum_{j=0}^{N-1}h^{j}a_{j}^{k}(t,x,\eta)d\eta\right)\nonumber \\
=\frac{h^{N}}{(2\pi h)^{n}}\int\e^{\frac{\i}{h}\psi^{k}(t,x,y,\eta)}q_{2}^{w}a_{N-1}^{k}d\eta & \defeq\cZ_{k}^{N}(t).\label{eq: ZkN}
\end{align}
It then remains to solve the transport equations. For this, consider
the two Lagrangian manifolds 
\[
\Lambda_{\eta}^{0}=\{(y,\eta):y\in\pi V_{0}\},\quad\Lambda_{\eta}(t)=\{(x,\d_{x}\varphi^{k}(t,x,\eta)):x\in\pi V_{k}\}
\]
where by construction 
\[
\Lambda_{\eta}^{0}(t)=\Phi^{t}(\Lambda_{\eta}^{0}).
\]
Define also the flow 
\[
\phi_{s,s'}^{k}:\begin{cases}
\pi\Lambda_{\eta}(s)\to\pi\Lambda_{\eta}(s+s')\\
x\mapsto\pi\Phi^{s'}(x,\d_{x}\varphi^{k}(s,x,\eta)).
\end{cases}
\]
The first transport equation express the classical fact that seen
as a half-density on $\Lambda_{\eta}(t)$ parametrized by $x$, the
amplitude $a_{0}^{k}|dx|^{\fr12}$ is invariant by this flow since
\[
(\d_{t}+X+\frac{1}{2}\div X)a_{0}^{k}|dx|^{\fr12}=(\d_{t}+\cL_{X})\lp a_{0}^{k}|dx|^{\fr12}\rp
\]
where $\cL_{X}$ denotes the Lie derivative, so 
\[
(\phi_{\tau_{k-1},t}^{k})^{*}a_{0}^{k}=a_{0}^{k}(\tau_{k-1},x)|dx|^{\fr12}\ \Leftrightarrow\ a_{0}^{k}(\tau_{k-1}+t,x)=a_{0}^{k}(\tau_{k-1},(\phi_{\tau_{k-1},t}^{k})^{-1}(x))|\det d_{x}\phi_{\tau_{k-1},t}^{k}|^{-\fr12}.
\]
In particular, since $ $$(\phi_{\tau_{k-1},t}^{k})^{-1}$ maps $x$
to $x_{k-1}$, we have 
\[
\prod_{j=0}^{k-2}(\phi_{\tau_{j},\tau_{j+1}}^{k})^{-1}(\phi_{\tau_{k-1},t}^{k})^{-1}:x\mapsto\d_{\eta}\varphi^{k}(t,x,\eta)
\]
and an immediate induction using the previous equation and \eqref{eq: Transition Principal}
shows that 
\begin{equation}
a_{0}^{k}=\i^{\sigma_{k}}\prod_{i=0}^{k}|\det d_{x}\phi_{\tau_{i-1},\delta t_{i}}^{i}|^{-\fr12}=\i^{\sigma_{k}}\lva\det\d_{x\eta}^{2}\varphi^{k}\rva^{\fr12}\label{eq: principal symbol}
\end{equation}
where 
\[
\i^{\sigma_{k}}=\prod_{i=1}^{k}\e^{\frac{\i\pi}{4}(\sgn\d_{\eta}^{2}\varphi^{i-1}-\sgn\d_{\eta}^{2}\varphi^{i})}\,.
\]
By convenience, we will write 
\[
a_{0}^{k}(\tau_{k-1}+t)=\bT_{\tau_{k-1}}^{t}a_{0}^{k}(\tau_{k-1})
\]
and call $\bT$ the transport operator. The higher order terms are
easily obtained by 
\begin{equation}
a_{j}^{k}(\tau_{k-1}+t)=\bT_{\tau_{k-1}}^{t}a_{j}^{k}(\tau_{k-1})+\int_{0}^{t}\bT_{\tau_{k-1}+s}^{t-s}q_{2}^{w}a_{j-1}^{k}(\tau_{k-1})ds\label{eq: Transport Higher}
\end{equation}
and we can compute the time-dependant symbol as long as we stay in
the framework described at the end of Section \ref{sub: Ham Jac},
namely for $t$ such that $\tau_{k-1}+t\in[t_{k}-h^{\epsilon L_{0}},t_{k}+h^{\epsilon L_{0}}]$.

To evaluate the difference $U_{k}(\tau_{k-1}+t)-U(\tau_{k-1}+t),$
we proceed by induction again: assuming that 
\[
U_{k-1}(t)A_{0}^{w}=U(t)A_{0}^{w}+\cR_{k-1}^{N}(t)A_{0}^{w}
\]
holds true for $t\in[t_{k-1}-h^{\epsilon L_{0}},t_{k-1}+h^{\epsilon L_{0}}]$,
we will construct $\cR_{k}^{N}(t)$ for $t\in[t_{k}-h^{\epsilon L_{0}},t_{k}+h^{\epsilon L_{0}}]$.
If $t=\tau_{k-1}$, it is natural to define in view of \eqref{eq: Chain 1}:
\begin{equation}
\cR_{k}^{N}(\tau_{k-1})\defeq\cS_{k-1}^{N}+\cR_{k-1}^{N}(\tau_{k-1}).\label{eq: Transition RkN}
\end{equation}
Hence we can rewrite the equation satisfied by $U_{k}(t)$: 
\[
\begin{cases}
(hD_{t}+hQ)U_{k}(t)A_{0}^{w}=\cZ_{k}^{N}(t)A_{0}^{w}\\
U_{k}(\tau_{k-1})A_{0}^{w}=U(\tau_{k-1})A_{0}^{w}+\cR_{k}^{N}(\tau_{k-1})A_{0}^{w}.
\end{cases}
\]
This allow to compare $U_{k}(t)A_{0}^{w}$ with $U(t)A_{0}^{w}$ for
$t\defeq\tau_{k-1}+t'\in[t_{k}-h^{\epsilon L_{0}},t_{k}+h^{\epsilon L_{0}}]$
thanks to the Duhamel formula: 
\begin{align*}
U_{k}(\tau_{k-1}+t')A_{0}^{w} & =U(t')U_{k}(\tau_{k-1})A_{0}^{w}+\int_{0}^{t'}U(t'-s)\cZ_{k}^{N}(s)A_{0}^{w}ds\\
 & =U(\tau_{k-1}+t')A_{0}^{w}+U(t')\cR_{k}^{N}(\tau_{k-1})A_{0}^{w}+\int_{0}^{t'}U(t'-s)\cZ_{k}^{N}(s)A_{0}^{w}ds.
\end{align*}
From this we define 
\begin{equation}
\cR_{k}^{N}(\tau_{k-1}+t')=U(t')\cR_{k}^{N}(\tau_{k-1})+\int_{0}^{t'}U(t'-s)\cZ_{k}^{N}(s)ds.\label{eq: Evolution RkN}
\end{equation}

\subsubsection{Estimation of the remainder terms in the trace class norm}

To complete the proof of the proposition, we now have to estimate
the trace of $\Pi\cR_{k}^{N}(t)A_{0}^{w}$ for all $k\leq K$ and
$t\in[t_{k}-h^{\epsilon L_{0}},t_{k}+h^{\epsilon L_{0}}]$. Equations
\eqref{eq: Transition RkN} and \eqref{eq: Evolution RkN} show that
for this purpose, it is enough to estimate the traces of $\Pi\cS_{k-1}^{N}A_{0}^{w}$
and $\Pi\cZ_{k}^{N}(t)A_{0}^{w}$ for all $k\leq K$ and $t\in[t_{k}-h^{\epsilon L_{0}},t_{k}+h^{\epsilon L_{0}}]$.
These operators have compactly supported kernels, and to estimate
their traces in view of \eqref{eq: def SkN}, \eqref{eq: SNk symbol}
and \eqref{eq: ZkN} it is then sufficient to have $C^{0}$ estimates
on their Schwartz kernels, hence on the symbols of $\cS_{k-1}^{N}$
and $\cZ_{k}^{N}$. The definition of these operators clearly indicates
that it is necessary to estimate the $C^{\ell}$ norms of the symbols
$a_{j}^{k}(t)$, which is the purpose of the next lemma. 
\begin{lem}
\label{lem: Symbol Derivatives} With the above notations, let $a_{j}^{k}(t,x,\eta)$
be the $j$-th term of the symbol of $U_{k}(t)$, and denote by $\rho\equiv(y,\eta)$
the unique point in $V_{0}\subset T^{*}M$ such that $\pi\Phi^{t}(y,\eta)=x$.
For the principal symbol, the following estimate hold true: 
\[
\|a_{0}^{k}(t)\|_{C^{0}}\leq C\sup_{\rho\in V_{0}}\e^{-\lambda_{t}^{+}(\rho)},\qquad\|a_{0}^{k}(t)\|_{C^{\ell}}\leq C_{\ell}(k+1)^{\ell},\qquad t\in[t_{k}-h^{\epsilon L_{0}},t_{k}+h^{\epsilon L_{0}}].
\]
For the higher order symbols, we have 
\[
\|a_{j}^{k}(t)\|_{C^{\ell}}\leq C_{j,\ell}(k+1)^{\ell+3j},\qquad t\in[t_{k}-h^{\epsilon L_{0}},t_{k}+h^{\epsilon L_{0}}].
\]
The constants $C,C_{\ell},C_{j,\ell}$ depend only on $M$ and $p.$ \end{lem}
\begin{proof}
We consider first the principal symbol. We will use the following
notations: for $0\leq i\leq k$, we write $\delta t_{i}\defeq\tau_{i}-\tau_{i-1}$
with the convention $\tau_{-1}=0$, and set 
\[
(x_{i},\xi_{i})=\Phi^{\delta t_{i}}(x_{i-1},\xi_{i-1})\in V_{i},\quad0\leq i\leq k
\]
again with the convention $(x_{-1},\xi_{-1})=(y,\eta)$. For sake
of notational simplicity we will consider only the case $t=\tau_{k}$
(hence $x=x_{k}$), the modification of the following proof for $t=\tau_{k-1}+t'\neq\tau_{k}$
being immediate. We recall that $|\det d_{x}\phi_{\tau_{i-1},\delta t_{i}}^{i}|^{-\frac{1}{2}}=\lva\det\frac{\d x_{i-1}}{\d x_{i}}\rva^{\frac{1}{2}}$
and hence we can rewrite the absolute value of the principal symbol
as 
\[
\prod_{i=0}^{k}|\det d_{x}\phi_{\tau_{i-1},\delta t_{i}}^{i}|^{-\fr12}=\prod_{i=0}^{k}\lva\det\frac{\d x_{i-1}}{\d x_{i}}\rva^{\fr12}=\lva\det\frac{\d y}{\d x_{k}}\rva^{\frac{1}{2}},
\]
and the $C^{0}$ estimate is a direct consequence of Proposition \ref{prop: NZ 5}.

Before proceeding to the estimates of the derivatives, we start with
a remark concerning the map $\phi_{\tau_{i-1},\delta t_{i}}^{i}:x_{i-1}\to x_{i}$.
For a given value of the parameter $\eta$, this map is invertible
and since $(x_{i},\xi_{i})=\Phi^{\tau_{i}}(y,\eta)$, it also induces
an invertible map $g_{\tau_{i-1},\delta t_{i}}^{i}:\xi_{i-1}\to\xi_{i}$,
which differential $dg_{\tau_{i-1},\delta t_{i}}^{i}$ is uniformly
contracting in the adapted coordinates \cite[p. 190]{NoZw09_1}. Note
however that in our case, the upper bound on the differential of this
map depends on $h,$ and is actually of the form $1-\cO(h^{L_{0}\epsilon})<1$
since out time-step is $\delta t_{i}=\cO(h^{L_{0}\epsilon}).$ As
a result, the Jacobian matrices $\frac{\d\xi_{i+j}}{\d\xi_{j}}$ are
uniformly bounded for $i,j\geq0$: these estimates are similar to
the estimate \eqref{eq: Uniform bounds horiz} and they are obtained
exactly in the same way, using the fact that $\xi$ are the ``stable''
coordinates as in the standard proof of the stable/unstable manifold
theorem for hyperbolic flows \cite[Chapter 6 and 17]{KaHa95}. Let
us sketch here the main argument : on a Poincaré section $\Sigma$
in a small neighborhood of $\Phi^{\delta t}(x,\xi)$, the Poincaré
map of the flow for time $\delta t$ has a differential in the adapted
coordinates of the form 
\[
d\kappa(\rho_{0})|_{\Sigma}=\left(\begin{array}{cc}
A & 0\\
0 & ^{t}A^{-1}
\end{array}\right)+\cO(\eps),\ \eps\ll1.
\]
In these coordinates, we then have 
\[
\delta x=(A+\cO(\eps))\delta x\quad\mbox{and}\quad\delta\xi=(^{t}A^{-1}+\cO(\eps))\delta\eta.
\]
Due to hyperbolicity, there is $\nu>1$ uniform near $\Gamma$ such
that for $\eps$ small enough, $\|A^{-1}+\cO(\eps)\|,\|^{t}A^{-1}+\cO(\eps)\|\leq\nu$
: namely, these matrices are respectively uniformly expanding and
uniformly contracting. Iterating this property from $\delta\xi_{i}$
to $\delta\xi_{i+j},i+j=\cO(\epsilon\log h^{-1})$, it is possible
to control the remainder terms (and here we refer the reader to \cite{NoZw09_1},
Proposition 5.1) and get for instance that 
\[
\delta\xi_{i+j}=\prod_{k=1}^{j}(^{t}A_{k}^{-1}+\cO_{k}(\eps))\delta\xi_{i}=\cO(1)\delta\xi_{i}.
\]
This implies that the Jacobian matrices $\frac{\d\xi_{i+j}}{\d\xi_{j}}$
are uniformly bounded for $i,j\geq0$.

For higher derivatives, let us call for $0\leq i\leq k$ 
\[
D_{i}=D_{i}(V_{0})=\sup_{x_{i}\in\pi\Phi^{\tau_{i}}(V_{0})}|\det d_{x}(\phi_{\tau_{i-1},\delta t_{i}}^{i})|^{-\frac{1}{2}},\quad f_{i}=|\det d_{x}\phi_{\tau_{i-1},\delta t_{i}}^{i}|^{-\fr12},\quad\bD_{i}=\prod_{0\leq j\leq i}D_{j}.
\]
Note that $D_{i}\leq1$ for all $i$ since the map $\phi_{\tau_{i-1},\delta t_{i}}^{i}:x_{i-1}\to x_{i}$
is uniformly expanding. We start with one derivative, and for this
we consider the two cases 
\[
\begin{cases}
\frac{\d a_{0}^{k}(x_{k},\eta)}{\d x_{k}}=\frac{\d f_{k}}{\d x_{k}}\prod_{i\neq k}f_{i}+\frac{\d f_{k-1}}{\d x_{k-1}}\frac{\d x_{k-1}}{\d x_{k}}\prod_{i\neq k-1}f_{i}+\dots+\frac{\d f_{0}}{\d x_{0}}\frac{\d x_{0}}{\d x_{k}}\prod_{i>0}f_{i}\\
\frac{\d a_{0}^{k}(x_{k},\eta)}{\d\eta}=\frac{\d f_{k}}{\d\xi_{k}}\frac{\d\xi_{k}}{\d\eta}\prod_{i\neq k}f_{i}+\frac{\d f_{k-1}}{\d\xi_{k-1}}\frac{\d\xi_{k-1}}{\d\eta}\prod_{i\neq k-1}f_{i}+\dots+\frac{\d f_{0}}{\d\xi_{0}}\prod_{i>0}f_{i}.
\end{cases}
\]
Using the uniform bounds for the Jacobian matrices $\frac{\d x_{k-j}}{\d x_{k}}$,
$\frac{\d\xi_{k-j}}{\d\eta}$ and the fact that $\frac{\d f_{k-j}}{\d x_{k-j}}$,
$\frac{\d f_{k-j}}{\d\xi_{k-j}}$ are uniformly bounded by constants
that depends only on $M$ and $p$, we end up with 
\[
\|a_{0}^{k}\|_{C_{x}^{1}}\leq C_{0,1}(k+1)\quad\mbox{and}\quad\|a_{0}^{k}\|_{C_{\eta}^{1}}\leq C'{}_{0,1}(k+1).
\]
The same procedure exactly can be applied by induction to show the
bounds for higher derivatives using the chain rule (see \cite[pp. 179--182]{NoZw09_1})
: if $\alpha=(\alpha_{x,1},\dots\alpha_{x,n};\alpha_{\eta,1},\dots\alpha_{\eta,n})$
is a multi index of length $|\alpha|$, we obtain that 
\[
\|a_{0}^{k}\|_{C_{x,\eta}^{\alpha}}\leq C_{0,\alpha}(k+1)^{|\alpha|},
\]
and this concludes the proof for the $C^{\ell}$ estimate of the principal
symbol.

For higher order symbols $a_{j}^{k}$, $j\geq1$, we proceed by induction,
on both $j$ and $k$. The principle is here slightly different compared
to \cite{NoZw09_1} since our transitions at times $\tau_{k}$ are
made differently. Assume the estimates have been established for $a_{j}^{k}(\tau_{k})$
and all $0\leq j\leq N-1$, and consider $a_{j}^{k+1}(\tau_{k+1})$.
Let $\ell\in\N^{2n}$ be a multi-index and consider the $C^{\ell}=C_{x,\eta}^{\ell}$
norm of $a_{j}^{k+1}(\tau_{k+1})$. From the transport equation \eqref{eq: Transport Higher},
the transitions at times $\tau_{k}$ and the induction hypothesis
we have 
\begin{align*}
\|a_{j}^{k+1}(\tau_{k+1})\|_{C^{\ell}} & \leq C_{M,p,\ell}D_{k+1}\lp\|a_{j}^{k+1}(\tau_{k})\|_{C^{\ell}}+\|a_{j-1}^{k+1}(\tau_{k})\|_{C^{\ell+2}}\rp\\
 & \leq C_{M,p,\ell}D_{k+1}\lp\sum_{\nu=0}^{j}\|Z_{j,\nu}^{k+1}a_{j-\nu}^{k}(\tau_{k})\|_{C^{\ell}}+\sum_{\nu=0}^{j-1}\|Z_{j-1,\nu}^{k+1}a_{j-1-\nu}^{k}(\tau_{k})\|_{C^{\ell+2}}\rp\\
 & \leq C_{M,p,\ell}D_{k+1}\lp\sum_{\nu=0}^{j}C_{j,\nu}\|a_{j-\nu}^{k}(\tau_{k})\|_{C^{\ell+2\nu}}+\sum_{\nu=0}^{j-1}\ti C_{j,\nu}\|a_{j-1-\nu}^{k}(\tau_{k})\|_{C^{\ell+2+2\nu}}\rp\\
 & \leq C_{M,p,\ell,j}D_{k+1}\lp\bD_{k}((k+1)^{\ell+3j}+(k+1)^{\ell+3j-1})\rp\\
 & \leq C_{M,p,\ell,j}\bD_{k+1}(k+2)^{\ell+3j}\leq C_{M,p,\ell,j}(k+2)^{\ell+3j}
\end{align*}
and this concludes the proof of the lemma. 
\end{proof}
We now come back to the estimate of the trace of $\Pi\cS_{k-1}^{N}A_{0}^{w}$.
Let $\Pi\cS_{k-1}^{N}A_{0}^{w}(\cdot,\cdot)$ denotes its Schwartz
kernel. Since the support of the diagonal embedding $x\mapsto\Pi\cS_{k-1}^{N}A_{0}^{w}(x,x)$
has a volume of order $\cO((\log h^{-1})^{n})$, it follows from the
properties of $\Pi$ and $A_{0}^{w}$ that there exists $C(h)=\cO((-\log h)^{n})$
independent of $k$ such that 
\[
\lva\Tr\Pi\cS_{k-1}^{N}A_{0}^{w}\rva\leq C(h)\|\Pi\cS_{k-1}^{N}A_{0}^{w}(\cdot,\cdot)\|_{C^{0}}.
\]
Now since the symbol of $\cS_{k-1}^{N}$ satisfies \eqref{eq: SNk symbol},
we can estimate the symbols $a_{j}^{k-1}$ thanks to the preceding
lemma, and we see that since $t\leq\epsilon\log h^{-1}$ and $k\leq K=\cO(h^{-\epsilon L_{0}}\log h^{-1})$,
the Schwartz kernel of $\Pi\cS_{k-1}^{N}A_{0}^{w}$ satisfies 
\[
\|\Pi\cS_{k-1}^{N}A_{0}^{w}(\cdot,\cdot)\|_{C^{0}}\leq C_{M,p,N,\epsilon}\frac{h^{N(1-C\epsilon)}}{(2\pi h)^{n}}=C_{M,p,N,\epsilon}\cO(h^{\frac{N}{2}})
\]
if $\epsilon$ and $N^{-1}$ are small enough. In this case, this
implies that $\lva\Tr\Pi\cS_{k-1}^{N}A_{0}^{w}\rva=\cO(h^{\frac{N}{3}})$
for large enough $N$.

The estimate for $\Tr\Pi\cZ_{k}^{N}(t)A_{0}^{w}$ is obtained in the
very same way. The operator $\cZ_{k}^{N}(t)$ is a Fourier integral
operator with a symbol given by $q_{2}^{w}a_{N-1}^{k}(t)$ and we
immediately get that 
\[
\|\Pi\cZ_{k}^{N}(t)A_{0}^{w}(\cdot,\cdot)\|_{C^{0}}\leq C_{M,p,N,\epsilon}h^{N-n-C\epsilon}\|a_{N-1}^{k}(t)\|_{C^{\alpha}}
\]
where $\alpha=\alpha_{M,p}>0$. Using Lemma \ref{lem: Symbol Derivatives},
we conclude that for $\epsilon,N^{-1}$ small enough, we have 
\[
\|\Pi\cZ_{k}^{N}(t)A_{0}^{w}(\cdot,\cdot)\|_{C^{0}}=\cO_{M,p,N,\epsilon}(h^{\frac{N}{2}})
\]
uniformly for $t\in[t_{k}-h^{\epsilon L_{0}},t_{k}+h^{\epsilon L_{0}}]$,
and this again implies that 
\[
|\Tr\Pi\cZ_{k}^{N}(t)A_{0}^{w}|=\cO_{M,p,N,\epsilon}(h^{N/3})
\]
for $\epsilon,N^{-1}$ small enough and $t\in[t_{k}-h^{\epsilon L_{0}},t_{k}+h^{\epsilon L_{0}}]$.

For the final estimate involving $\cR_{k}^{N}(t)$, we use a recursion
based on \eqref{eq: Transition RkN} and \eqref{eq: Evolution RkN}.
Indeed, if we define $\delta_{k}\defeq\tau_{k}-\tau_{k-1}$, then
we have 
\begin{align*}
\cR_{k}^{N}(\tau_{k-1}+t') & =U(t')\cS_{k-1}^{N}+U(t')\lp U(\delta_{k-1})\cR_{k-1}^{N}(\tau_{k-2})+\int_{0}^{\delta_{k-1}}U(\delta_{k-1}-s)\cZ_{k-1}^{N}(s)ds\rp\\
 & +\int_{0}^{t'}U(t'-s)\cZ_{k}^{N}(s)ds.
\end{align*}
We can iterate this formula further down to $k=0$, to obtain that
\begin{align*}
\lva\Tr\Pi\cR_{k}^{N}(\tau_{k-1}+t')A_{0}^{w}\rva & \leq C_{M,p,N,\epsilon}\sum_{\nu=0}^{k}\lva\Tr\Pi\cS_{k}^{N}A_{0}^{w}\rva+\sup_{t\in[0,\delta_{\nu}]}\lva\Tr\Pi\cZ_{\nu}^{N}(t)A_{0}^{w}\rva\\
 & \leq C_{M,p,N,\epsilon}h^{-L_{0}\epsilon}(\log h^{-1})\cO(h^{N/2})=\cO(h^{N/3})
\end{align*}
if $\epsilon,N^{-1}$ are small enough, thanks to the estimates we
have established just above. This concludes the proof of Proposition
\eqref{prop: The big one}.

\subsection{Microlocal partition and computation of the wave trace}

In this section we complete the proof of Proposition \ref{prop: Long Time Trace Formula}
for the test function $\phi^{(1)}$. In order to use the local representations
of $U(t)$ with $h-$oscillatory integral developed in the preceding
section, we will first define a specific cover of the phase space
in a neighborhood of $\cE^{*}M$.

\subsubsection{Microlocalization around periodic orbits, and completion of the cover}

Recall that $\phi^{(1)}$ has support of the form $[t_{0}-\frac{h^{\epsilon J_{+}}}{2},t_{0}+\frac{h^{\epsilon J_{+}}}{2}]\subset[T-1,T]$.
For the proof of Proposition \ref{prop: Long Time Trace Formula},
it is not necessary to control the elements of the length spectrum
in $\supp\phi^{(1)}$, so we simply assume here that \eqref{eq: Relation beta lambda}
holds true with no further precisions. Since $h^{\epsilon J_{+}}<\delta_{0}$
if $h$ is small enough, we are sure from Lemma \ref{lem: Separation Orbits}
(up to enlarge $L_{0}$ a bit) that if $\gamma\in\cE^{*}M$ satisfies
$\ell(\gamma)\in\supp\phi^{(1)}$, then the tubular neighborhood $\Theta(\gamma,h^{\epsilon L_{0}})$
in $\cE^{*}M$ does not contain another periodic orbit with length
in $\supp\phi^{(1)}.$

For such a $\gamma$ with $\ell(\gamma)\in\supp\phi^{(1)}$, take
a sequence of points $(\rho_{i})_{0\leq i\leq N_{\gamma}}$ of $\gamma$,
such that the open balls $B_{\rho_{i}}(h^{\epsilon L_{0}})$ form
a chain of neighborhoods in $T^{*}M$ covering $\gamma$. In particular,
$N_{\gamma}=\cO(h^{-\epsilon L_{0}}\log h^{-1})$. We will denote
by $B_{\gamma}=\cup_{i=1}^{N_{\gamma}}B_{\rho_{i}}(h^{\epsilon L_{0}})$
the open cover of $\gamma$ obtained with these balls.

Consider now the set 
\[
\cE^{*}M_{res}^{\delta/2}\defeq\lp\cE^{*}M^{\delta/2}\cap T^{*}(\supp\Pi)\rp\setminus\bigcup_{\gamma:\ell(\gamma)\in\supp\phi^{(1)}}\bigcup_{i=1}^{N_{\gamma}}B_{\rho_{i}}(h^{\epsilon L_{0}})
\]
and choose an open cover of $\cE^{*}M_{res}^{\delta/2}$ by open balls
of size $h^{\epsilon L_{0}}$. We denote by $(W_{j})_{j\in J}$ the
open sets of this cover, and we can arrange that $\bigcup_{j}W_{j}\subset\cE^{*}M^{\delta}$,
so $(W_{j})_{j\in J}$ and $(B_{\gamma})_{\ell(\gamma)\in\supp\phi^{(1)}}$
form a cover of $\cE^{*}M^{\delta/2}\cap T^{*}(\supp\Pi)$ which stays
included in $\cE^{*}M^{\delta}$. Without loss of generality, we can
also require that if $\rho\in W_{j}$ for some $j\in J$, then then
$d(\rho,\gamma)>\frac{3}{4}h^{\epsilon L_{0}}$: in this way, the
intersection of each $W_{j}$ with periodic orbits in $\cE^{*}M$
of length in $\supp\phi^{(1)}$ is always empty.

We then choose a partition of unity adapted to the full cover of $\cE^{*}M^{\delta/2}\cap T^{*}\supp\Pi$
we have obtained with the sets $(B_{\gamma})_{\ell(\gamma)\in\supp\phi^{(1)}}$
and $(W_{j})_{j\in J}$. In particular, for each orbit $\gamma$ and
$i\leq N_{\gamma}$, we can choose the points $\rho_{i}$ and the
functions $\chi_{\gamma^{i}}\in C_{0}^{\infty}(B_{\rho_{i}}(h^{\epsilon L_{0}}))$
such that near $\gamma$, the functions $\chi_{\gamma^{i}}$ form
(local) partition of unity with 
\begin{equation}
\sum_{i=1}^{N_{\gamma}}\chi_{\gamma^{i}}(\rho)=1\,,\quad\rho\in\Theta(\gamma,\fr12h^{\epsilon L_{0}}).\label{eq: partition unity gamma-1}
\end{equation}

It is clear that if we choose $\epsilon>0$ small enough such that
$L_{0}\epsilon<\fr12$, the $\chi_{\gamma^{i}}$ can be constructed
by rescaling $h-$independent functions by $h^{\epsilon L_{0}}$,
so we can have $\chi_{\gamma^{i}}\in S_{\epsilon L_{0}}^{0,0}(T^{*}M).$
Remark that with the notations of Section \ref{sub: Local-integral-representations},
we can take $V_{0}\equiv B_{\rho_{i}}(h^{\epsilon L_{0}})$ for each
$i$, and find local generating functions for the flow around $\gamma$,
starting in $V_{0}$. Finally, let 
\[
\chi_{\gamma}\defeq\sum_{i=0}^{N_{\gamma}}\chi_{\gamma^{i}}\,.
\]
We denote by $(\ti\chi_{j})_{j\in J}$ the functions in $C_{0}^{\infty}(T^{*}M)$
supported in the $(W_{j})_{j\in J}$ which complete the partition
of unity. Again, $\ti\chi_{j}\in S_{\epsilon L_{0}}^{0,0}(T^{*}M)$.
By quantization, all these functions produce a quantum partition of
unity microlocalized inside $\cE^{*}M^{\delta}\cap T^{*}\supp\Pi$,
and we have: 
\begin{align*}
\Tr\int\Pi f(hQ)U(t)\e^{\frac{\i}{h}t}\phi^{(1)}(t)dt\Pi & =\Tr\int\Pi U(t)\lp\sum_{\gamma:\ell(\gamma)=\ell_{0}}\Op_{h}(\chi_{\gamma})\rp\e^{\frac{\i}{h}t}\phi^{(1)}(t)dt\\
 & +\Tr\int\Pi U(t)\lp\sum_{j\in J}\Op_{h}(\ti\chi_{j})\rp\e^{\frac{\i}{h}t}\phi^{(1)}(t)dt+\cO(h^{\infty}),
\end{align*}
and the same equality holds true with $(1-\chi)U_{0}(t)(1-\chi)$
replacing $U(t)$.

\subsubsection{Stationary phase near periodic orbits }

In this section also, $C$ will here denote a positive constant depending
only on $M,p$ if written without further notifications. Still for
$\gamma\in\scP$ such that $\ell(\gamma)\in\supp\phi^{(1)}$, let
us define 
\begin{align*}
\Tr(\gamma,i) & \defeq\Tr\int\Pi U(t)\oph(\chi_{\gamma^{i}})\Pi\e^{\frac{\i}{h}t}\phi^{(1)}(t)dt\,,
\end{align*}
and 
\[
\Tr_{\chi_{\gamma}}\defeq\sum_{i=1}^{N_{\gamma}}\Tr(\gamma,i).
\]
The preliminary work in the above sections allows to represent $\Tr(\gamma,i)$
by oscillatory integrals. Locally in $t,x,y,\eta$, we have from Proposition
\ref{prop: The big one}: 
\begin{equation}
\Pi U(t)\Op_{h}(\chi_{\gamma^{i}})\Pi=\frac{1}{(2\pi h)^{2n}}\int\e^{\frac{\i}{h}(\varphi(t,x,\eta)-\moy{y,\eta})}\e^{\frac{\i}{h}\moy{y-z,\xi}}a_{h}(t,x,\eta)\chi_{\gamma^{i}}\lp\frac{y+z}{2},\xi\rp dyd\eta d\xi\label{eq: microloc depart}
\end{equation}
since working with local charts, we have $\Pi\equiv1$ on $\supp\chi_{\gamma^{i}}$.
In the sequel to simplify the notations, we will omit the terms $\Pi$
in the computations. 
\begin{lem}
\label{lem: Phase Stat PO}We have 
\begin{align*}
U(t)\oph(\chi_{\gamma}) & =\frac{1}{(2\pi h)^{n}}\int\e^{\frac{\i}{h}(\varphi(t,x,\eta)-\moy{x,\eta})}a_{h}(t,x,\eta)(\chi_{\gamma^{i}}(x,\eta)+h^{\alpha_{1}}r_{\gamma^{i}}(x,\eta))d\eta
\end{align*}
where $r_{\gamma^{i}}\in S_{\epsilon L_{0}}^{0,0}(T^{*}M)$ and $\alpha_{1}=1-C\epsilon>0$
if $\epsilon$ is small enough. \end{lem}
\begin{proof}
This is a straightforward application of the stationary phase expansion
in the $y,\xi$ variables in \eqref{eq: microloc depart}, noting
that $\chi_{\gamma^{i}}\in S_{\epsilon L_{0}}^{0,0}(T^{*}M)$. 
\end{proof}
The content of the next lemma is classical and consists in the stationary
phase expansion of the trace localized around the periodic orbits.
We refer to \cite{CdV73_trI,CdV73_TrII,DuGu75,SjZw02} for additional
details and references. We recall that in our case, we must check
carefully the expansion in $h=\lambda^{-1}$ as we work until times
which are of Ehrenfest type. 
\begin{lem}
\label{lem: Big Stat Phase}Let $\gamma\in\scP$ with $\ell(\gamma)\in\supp\phi^{(1)}$
as above. Then

\[
\Tr_{\chi_{\gamma}}=\e^{\i\frac{\ell(\gamma)}{h}}\frac{\ell^{\sharp}(\gamma)}{\sqrt{|1-P_{\gamma}|}}\phi^{(1)}(\ell(\gamma))+\cO(h^{\alpha_{2}})
\]
where $\ell^{\sharp}(\gamma)$ is the primitive length of $\gamma$
and $\alpha_{2}=1-C\epsilon>0$ if $\epsilon$ is small enough. 
\begin{proof}
For a given term $\Tr(\gamma,i)$, Proposition \ref{prop: The big one},
Lemma \ref{lem: Phase Stat PO}, and the hypothesis on $\chi_{\gamma^{i}}$
and $\supp\phi^{(1)}$ allow to find a single phase function $\varphi$
and a symbol $a_{h,N}=\sum_{j=0}^{N-1}h^{j}a_{j}$ such that we can
compute the trace via the formula 
\begin{align*}
\Tr(\gamma,i) & =\frac{1}{(2\pi h)^{n}}\int\e^{\frac{\i}{h}(\varphi(t,x,\eta)-\moy{x,\eta}+t)}a_{h,N}(t,x,\eta)\chi_{\gamma^{i}}(x,\eta)\phi^{(1)}(t)d\eta dxdt\\
 & +\frac{h^{\alpha_{1}}}{(2\pi h)^{n}}\int\e^{\frac{\i}{h}(\varphi(t,x,\eta)-\moy{x,\eta}+t)}a_{h,N}(t,x,\eta)r_{\gamma^{i}}(x,\eta)\phi^{(1)}(t)d\eta dxdt+\cO(h^{\frac{N}{3}}).
\end{align*}
Here we have used the fact that $\Pi=1$ on $\supp_{x}a_{h}^{N}$.
We would like to evaluate this expression via the stationary phase
in $x,\eta,t$. The critical points satisfy 
\begin{gather}
\d_{\eta}\varphi(t,x,\eta)=x,\nonumber \\
\d_{x}\varphi(t,x,\eta)=\eta,\nonumber \\
1+\d_{t}\varphi(t,x,\eta)=0.\label{eq: Critical points}
\end{gather}
In view of Proposition \ref{prop: Local projection}, this means that
at the critical points $(x_{c},\eta_{c},t_{c})$ we have $\Phi^{t_{c}}(x_{c},\eta_{c})=(x_{c},\eta_{c})$
and $q(x_{c},\eta_{c})=1$, which is precisely the equation defining
a periodic orbit on the unit energy level of length $t_{c}\in\supp\phi^{(1)}$.
Hence these critical points form a closed, non-degenerate submanifold
of $T^{*}M$ of dimension 1. Following \cite{CdV73_TrII,DuGu75} we
will use the clean version of the stationary phase theorem by using
local coordinates transverse to $\gamma$. Denote by $\rho_{i}$ the
central point of $B_{\rho_{i}}(h^{\epsilon L_{0}})=\supp\chi_{\gamma^{i}}$,
and remark first that because of $\chi_{\gamma^{i}}$, the piece of
orbit that form the critical points of the above integral lies in
a small neighborhood of $\rho_{i}$, say 
\[
W_{\chi_{\gamma}}=\{\Phi^{s}(\rho_{0}),\ \ |s|\leq\cO(h^{\epsilon L_{0}})\}.
\]
Choose local coordinates $(x_{1},\dots,x_{n};\xi_{1},\dots,\xi_{n})$
near $\rho_{i}$ such that $H_{p}=\d/\d_{x_{1}}$ at the periodic
point, and denote by $W_{\chi_{\gamma}}^{\perp}$ the set described
by the transversal coordinates $(x_{2},\dots,\xi_{n})$. We can write
\begin{align}
\Tr(\gamma,i) & =\frac{1}{(2\pi h)^{n}}\int\int_{W_{\chi_{\gamma^{i}}}\times W_{\chi_{\gamma^{i}}}^{\perp}}\e^{\frac{\i}{h}\psi(t,x,y,\eta)}a_{h,N}(t,x,\eta)\chi_{\gamma^{i}}(x,\eta)\phi^{(1)}dxdtd\eta\label{eq: Trace Phase stat finale}\\
 & +\frac{h^{\alpha_{1}}}{(2\pi h)^{n}}\int\int_{W_{\chi_{\gamma^{i}}}\times W_{\chi_{\gamma^{i}}}^{\perp}}\e^{\frac{\i}{h}\psi(t,x,y,\eta)}a_{h,N}(t,x,\eta)r_{\gamma}(x,\eta)\phi^{(1)}dtd\eta dx+\cO(h^{\frac{N}{3}})\nonumber 
\end{align}
where $\psi(t,x,y,\eta)\defeq\varphi(t,x,\eta)-\moy{x,\eta}+t$. We
apply the (usual) stationary phase in the variables $(t,\eta_{1},x_{i},\eta_{i})$
for $i>1$. The Hessian of the phase $\psi(t,x,y,\eta)$ is given
by 
\[
\Hess\psi(t,x,x,\eta)=\left(\begin{array}{ccc}
\d_{t}^{2}\varphi & \d_{xt}^{2}\varphi & \d_{\eta t}^{2}\varphi\\
\d_{tx}^{2}\varphi & \d_{x}^{2}\varphi & \d_{\eta x}^{2}\varphi-\Id\\
\d_{t\eta}^{2}\varphi & \d_{x\eta}^{2}\varphi-\Id & \d_{\eta}^{2}\varphi
\end{array}\right).
\]
Using \eqref{eq: Flow Gen Fn} and the Hamilton-Jacobi equation \eqref{eq: ham Jac},
we find that $\Hess\psi$ is block-diagonal in the decomposition $\Span(\d_{t},\d_{\xi_{1}})\oplus\Span_{i>1}(\d_{x_{i}},\d_{\xi_{i}})$,
and 
\begin{equation}
\Hess\psi|_{W_{\chi_{\gamma^{i}}}^{\perp}}=\left(\begin{array}{cccc}
0 & 1 & 0 & 0\\
1 & 0 & 0 & 0\\
0 & 0 & \d_{x}^{2}\varphi & \d_{\eta x}^{2}\varphi-\Id\\
0 & 0 & \d_{x\eta}^{2}\varphi-\Id & \d_{\eta}^{2}\varphi
\end{array}\right).\label{eq: Hessian big}
\end{equation}
A standard computation \cite{SjZw02} shows that 
\[
d\Phi^{t}|_{W_{\chi_{\gamma^{i}}}^{\perp}}=\left(\begin{array}{cc}
(\d_{\eta x}^{2}\varphi^{\perp})^{-1} & -(\d_{\eta x}^{2}\varphi^{\perp})^{-1}\d_{\eta\eta}^{2}\varphi^{\perp}\\
\d_{xx}^{2}\varphi^{\perp}(\d_{\eta x}^{2}\varphi^{\perp})^{-1} & \d_{x\eta}^{2}\varphi^{\perp}-\d_{xx}^{2}(\d_{\eta x}^{2}\varphi^{\perp})^{-1}\d_{\eta\eta}^{2}\varphi^{\perp}
\end{array}\right)_{\d_{x},\d_{\eta}\in TW_{\chi_{\gamma^{i}}}^{\perp}}
\]
where by convenience, we have written $\d_{x\eta}^{2}\varphi^{\perp}$
when $\d_{x},\d_{\eta}\in TW_{\chi_{\gamma^{i}}}^{\perp}$ . Again,
matrices equalities show that 
\begin{equation}
\det(P_{\gamma}-\Id)\defeq\det(d\Phi^{t}|_{W_{\chi_{\gamma^{i}}}^{\perp}}-\Id)=(\det(\d_{\eta x}^{2}\varphi^{\perp})^{-1})\det\left(\begin{array}{cc}
\d_{x}^{2}\varphi^{\perp} & \d_{\eta x}^{2}\varphi^{\perp}-\Id\\
\d_{x\eta}^{2}\varphi^{\perp}-\Id & \d_{\eta}^{2}\varphi^{\perp}
\end{array}\right)\label{eq: Poincare det}
\end{equation}
where $P_{\gamma}$ is the Poincaré map of the orbit $\gamma$. Note
also that in the chosen coordinate system, $\det\d_{\eta x}^{2}\varphi^{\perp}=\det\d_{\eta x}^{2}\varphi$
since $\d_{x_{1}\eta_{1}}^{2}\varphi=1$ and $\d_{x_{1}\eta_{i}}^{2}\varphi=0$
if $i>1$.

Without loss of generality we can assume that $\tau_{N_{\gamma}}<t_{N_{\gamma}}\defeq\ell(\gamma)$,
$t_{N_{\gamma}}$ being the last time in the subdivision chosen to
apply Proposition \ref{prop: The big one}. Recall that 
\[
\sigma_{N_{\gamma}}=\sum_{i=1}^{N_{\gamma}}\frac{1}{2}(\sgn\d_{\eta}^{2}\varphi^{i-1}-\sgn\d_{\eta}^{2}\varphi^{i}).
\]
From Proposition \ref{prop: The big one}, the principal symbol has
the form 
\[
a_{0}=\i^{\sigma_{N_{\gamma}}}|\det\d_{x\eta}^{2}\varphi|^{\frac{1}{2}}=\i^{\sigma_{N_{\gamma}}}|\det\d_{x\eta}^{2}\varphi^{\perp}|^{\fr12}
\]
as noted above. Hence from \eqref{eq: Poincare det} we have 
\[
|\det(\d_{x\eta}^{2}\varphi^{\perp})\det(P_{\gamma}-\Id)|^{\fr12}=|\det\Hess\psi|_{W_{\chi_{\gamma^{i}}}^{\perp}}|^{\fr12}\ \Rightarrow\frac{a_{0}}{|\det\Hess\psi|_{W_{\chi_{\gamma^{i}}}^{\perp}}|^{\fr12}}=\frac{\i^{\sigma_{N_{\gamma}}}}{|\det(P_{\gamma}-\Id)|^{\fr12}}.
\]
Applying the stationary phase principle in the integral \eqref{eq: Trace Phase stat finale},
we will then get an expansion where the leading term coming from $a_{h,N}$
gives precisely $|\det(P_{\gamma}-\Id)|^{-\fr12}$ modulo a phase
factor. From the symbol estimates in Lemma \ref{lem: Symbol Derivatives}
and the property 
\[
\frac{d^{j}}{dt^{j}}\phi^{(1)}=\cO(\beta^{j\epsilon J_{+}})=\cO(h^{-j\epsilon J_{+}}),
\]
the remainder in the stationary phase expansion at first order together
with the higher order symbols $a_{j}$ with $2\leq j\leq N-1$ at
the critical points have a contribution of order $\cO(h^{\alpha})$
with $\alpha=1-C_{M,p,\phi}\epsilon<0$ if $\epsilon$ is small enough,
and finally 
\[
\Tr(\gamma,i)=\int_{W_{\chi_{\gamma^{i}}}}\frac{\i^{\sigma_{N_{\gamma}}}\e^{\frac{\i\pi}{4}\sgn\Hess\psi|_{W_{\chi_{\gamma^{i}}}^{\perp}}}}{|\det(1-P_{\gamma})|^{\fr12}}\phi^{(1)}(\ell(\gamma))dx_{1}+\cO(h^{1-C\epsilon})+\cO(h^{\frac{N}{3}}).
\]
To conclude the proof of the long time trace formula, we need now
to take care of the phase factor in the above expression. 
\begin{lem}
The integer $\mu(\gamma)\defeq\sigma_{N_{\gamma}}+\frac{1}{2}\sgn\Hess\psi|_{W_{\chi_{\gamma^{i}}}^{\perp}}$
is the Maslov index of the orbit $\gamma$, and the hyperbolicity
of the flow implies that $\mu(\gamma)=0.$ 
\begin{proof}
Let us first recall a couple of facts about the Hörmander and Hörmander-Kashiwara
indices \cite{LioVer80_maslov,Dui76_morse}. If $L_{0},L_{1},L_{2}$
is a triple of Lagrangian planes in a symplectic vector space $(S,\omega)$,
consider the quadratic form on $L_{0}\oplus L_{1}\oplus L_{2}$ given
by 
\[
Q(v_{0},v_{1},v_{2})=\omega(v_{0},v_{1})+\omega(v_{1},v_{2})+\omega(v_{2},v_{0}).
\]
The Hörmander-Kashiwara index of this triple is defined by 
\[
\sgn(L_{0},L_{1},L_{2})\defeq\sgn Q.
\]
In the particular case where $L_{1}$ is transversal to $L_{2}$,
the projection $\pi_{12}$ on $L_{2}$ along $L_{1}$ is well defined
and we are reduced to compute the signature of a quadratic form on
$L_{0}$ only \cite{LioVer80_maslov} : 
\begin{equation}
\sgn(L_{0},L_{1},L_{2})=-\sgn\omega(\pi_{12}\cdot,\cdot)|_{L_{0}\times L_{0}}.\label{eq: HorKashi reduced}
\end{equation}
In this transversal situation, one also has the property that $\sgn$
is antisymmetric in its variables. Finally, if $M_{1},M_{2};L_{1},L_{2}$
are Lagrangian planes, the Hörmander index is defined by 
\[
s(M_{1},M_{2};L_{1},L_{2})=\frac{1}{2}(\sgn(M_{1},M_{2},L_{1})-\sgn(M_{1},M_{2},L_{2})).
\]
Now let $V=\{(0,\delta\xi)\}$ be the tangent space to the fiber in
$T^{*}M$ and $H_{i}=\{(\delta x,0)\}$ the horizontal space in the
$i-$th local representation of $U(t)$. These Lagrangian planes are
transversal, and from 
\[
(\d_{\eta}\varphi^{i}(t,x,\eta),\eta)\xrightarrow{\Phi^{t_{i}}}(x,\d_{x}\varphi^{i}(t,x,\eta)),\qquad y=\d_{\eta}\varphi^{i}(t,x,\eta),\ \ \xi=\d_{x}\varphi^{i}(t,x,\eta),
\]
we have $\d_{\eta}^{2}\varphi_{i}:V\to H_{i}$ and its graph is precisely
$(d\Phi^{t_{i}})^{-1}(V)$. Then \eqref{eq: HorKashi reduced} implies
that 
\[
\sgn\d_{\eta}^{2}\varphi_{i}=\sgn(H_{i},V,(d\Phi^{t_{i}})^{-1}(V))
\]
and this yields (see also \cite{DuGu75}) : 
\[
\frac{1}{2}(\sgn\d_{\eta}^{2}\varphi^{i-1}-\sgn\d_{\eta}^{2}\varphi^{i})=s(H_{i-1},H_{i};(d\Phi^{t_{i-1}})^{-1}(V),V).
\]
Consider the product symplectic manifold $(X,\sigma)\defeq(T^{*}M\times T^{*}M,\omega_{1}-\omega_{2})$
where $\omega_{i}$ are the canonical symplectic forms on the factors,
and set 
\begin{equation}
H\times V\defeq(H\oplus\{0\})\times(\{0\}\oplus V)\subset T_{(x,\xi;y,\eta)}^{*}(T^{*}M\times T^{*}M).\label{eq: H x V}
\end{equation}
Since $V$ and $H$ are transversal, it follows that $V$ and $(d\Phi^{t_{i}})^{-1}(V)$
are transversal too, so \cite{Dui76_morse}, Corollary 3.3 says that
if $\Delta$ denotes the diagonal in $T^{*}X$, then 
\begin{align*}
s(H_{i-1},H_{i};(d\Phi^{t_{i-1}})^{-1}(V),V) & =-s(H_{i-1}\times V,H_{i}\times V;\Delta,\Graph(d\Phi^{t_{i-1}}))\\
 & =-s(\Graph(d\Phi^{t_{i-1}}),\Delta;H_{i-1}\times V,H_{i}\times V)\\
 & =\frac{1}{2}\sgn(\Graph(d\Phi^{t_{i-1}}),\Delta,H_{i}\times V)\\
 & -\frac{1}{2}\sgn(\Graph(d\Phi^{t_{i-1}}),\Delta,H_{i-1}\times V)
\end{align*}
where we used that $\Delta$ is always transversal to $H\times V$
and the antisymmetry of $\sgn$ in that case.

In view of our choice \eqref{eq: H x V}, consider now the general
graph $(\Phi^{t}(\rho),\rho)\subset X$ near a base point $\rho=(y,\eta)$,
$t\leq t_{N_{\gamma}}$. We can here apply \eqref{eq: HorKashi reduced}
to compute $\sgn(\Graph(d\Phi^{t}),\Delta,H\times V)$, using for
$\pi_{12}$ the projection on $\Delta$ along $H\times V$. As above,
we can take $(x,\eta)$ as coordinates on the graph of $\Phi^{t}$,
and then: 
\[
\Graph d\Phi^{t}=(\delta x,(\d_{x}^{2}\varphi)\delta x+(\d_{x\eta}^{2}\varphi)\delta\eta;{}^{t}(\d_{x\eta}^{2}\varphi)\delta x+(\d_{\eta}^{2}\varphi)\delta\eta,\delta\eta).
\]
A straightforward computation using \eqref{eq: HorKashi reduced}
with $L_{0}=\Graph(d\Phi^{t})$, $L_{1}=\Delta$ and $L_{2}=H\times V$,
the canonical form $\omega_{1}-\omega_{2}$ on $T^{*}X$ and the decomposition
\[
T^{*}M=\Span(H_{p})\oplus\Span(\d_{t},\d_{\xi_{1}})\oplus\Span_{i>1}(\d_{x_{i}},\d_{\xi_{i}})
\]
yields directly to 
\[
\sgn(\Graph(d\Phi^{t}),\Delta,H\times V)=-\sgn\left(\begin{array}{cc}
\d_{x}^{2}\varphi & \d_{\eta x}^{2}\varphi-\Id\\
\d_{x\eta}^{2}\varphi-\Id & \d_{\eta}^{2}\varphi
\end{array}\right).
\]
If we take $t=t_{N_{\gamma}}$ in the above equation, we obtain from
\eqref{eq: Hessian big} the identity 
\[
\sgn(\Graph(d\Phi^{t_{N_{\gamma}}}),\Delta,H_{N_{\gamma}}\times V)=-\sgn\Hess\psi|_{W_{\chi_{\gamma^{i}}}^{\perp}}.
\]
We can sum up the preceding discussion to get : 
\begin{align}
\mu(\gamma) & =-\frac{1}{2}\sgn(\Graph(d\Phi^{t_{N_{\gamma}}}),\Delta,H_{N_{\gamma}}\times V)\nonumber \\
 & +\frac{1}{2}\sum_{i=1}^{N_{\gamma}}\sgn(\Graph(d\Phi^{t_{i}-1}),\Delta,H_{i}\times V)-\sgn(\Graph(d\Phi^{t_{i}-1}),\Delta,H_{i-1}\times V)\nonumber \\
 & =-\sgn(\Graph(d\Phi^{t_{0}}),\Delta,H_{0}\times V)\nonumber \\
 & -\sum_{i=1}^{N_{\gamma}}\sgn(\Graph(d\Phi^{t_{i}}),\Delta,H_{i}\times V)-\sgn(\Graph(d\Phi^{t_{i-1}}),\Delta,H_{i}\times V)\nonumber \\
 & =\sum_{i=0}^{N_{\gamma}}s(\Delta,H_{i}\times V;\Graph(d\Phi^{t_{i}}),\Graph(d\Phi^{t_{i-1}}))\label{eq: indice de maslov}
\end{align}
where we have set $t_{-1}=0$ and used the fact that 
\[
\sgn(d\Phi^{0},\Delta,H_{0}\times V)=\sgn(\Delta,\Delta,H_{0}\times V)=0.
\]
The integer in the last line of \eqref{eq: indice de maslov} is called
the \emph{Maslov index} of the symplectic curve $t\mapsto\Graph\Phi^{t}$
for $0\leq t\leq t_{N_{\gamma}}$, and this number is actually independant
of the subdivision $t_{0},\dots,t_{N_{\gamma}}$ \cite{Dui76_morse,LioVer80_maslov}.

Consider then in $X$ the reference Lagrangian manifold $\Lambda_{0}=M\times F$
where $M$ is the base manifold in the first factor, and $F$ the
fibre in the second factor. In local coordinates, $\Lambda_{0}=\{(x,\{\xi\},\{y\},\eta)\}$.
The tangent space of $\Lambda_{0}$ at $(x,\xi;y,\eta)$ is a Lagrangian
plane in $T_{(x,\xi;y,\eta)}^{*}X$ which is precisely $H\times V$
in the above notations. Now the mapping 
\begin{equation}
t\mapsto\sgn(\Graph(d\Phi^{t}),\Delta,H\times V)\label{eq: singular mapping}
\end{equation}
is \emph{not} continuous : it jumps by $\pm2$ if $\Graph(d\Phi^{t})$
crosses the singular cycle $\Lambda^{1}(X,\Lambda_{0})$ attached
to $\Lambda_{0}$, in other words, if there is $t$ such that the
projection $\pi_{G}:\Graph d\Phi^{t}\to H\times V$ is singular \cite{Dui96}.
But this never happens for hyperbolic flows, as the vertical fibre
bundle $F$ is transverse to the weak unstable foliation. Indeed,
in the coordinate system adapted to the dynamics introduced above,
the weak unstable manifold $\R H_{p}(\rho)\oplus E_{\rho}^{u}$ is
canonically identified to $H_{t=0}\subset T_{\rho}^{*}M$. Now $\Lambda(0)=(\{y\},\eta)$
is a Lagrangian manifold in $T^{*}M$ which has tangent space precisely
equal to $F$, which is transverse to $H_{t=0}$ and Proposition \ref{prop: NZ 5}
then says that $\Lambda(t)=\Phi^{t}(\Lambda(0))$ is a Lagrangian
manifold in $T^{*}M$ that projects diffeomorphically (locally) onto
the weak unstable manifold $\R H_{p}(\Phi^{t}(\rho))\oplus E_{\Phi^{t}(\rho)}^{u}\equiv H_{t}$
at all times $t>0$. Equivalently, this means that $\Graph\Phi^{t}$
always projects locally diffeomorphically to $M\times F$ for all
times. Hence the mapping \eqref{eq: singular mapping} is constant
: 
\[
\forall i\geq0,\qquad\sgn(\Graph(d\Phi^{t_{i+1}}),\Delta,H_{i}\times V)=\sgn(\Graph(d\Phi^{t_{i}}),\Delta,H_{i}\times V)\,,
\]
and we immediately get that $\mu(\gamma)=0$. 
\end{proof}
\end{lem}
Finally, by the principle of stationary phase at the first order,
we get 
\begin{align*}
\Tr(\gamma,i) & =\int_{W_{\chi_{\gamma^{i}}}}\frac{1}{|\det(1-P_{\gamma})|^{\fr12}}(1+\cO(h^{1-C\epsilon}))\phi^{(1)}(\ell(\gamma))dx_{1}+\cO(h^{\frac{N}{3}})\\
 & =\frac{1}{|1-P_{\gamma}|^{\fr12}}\e^{\frac{\i\ell(\gamma)}{h}}\phi^{(1)}(\ell(\gamma))\int_{\gamma}\chi_{\gamma^{i}}(1+\cO(h^{1-C\epsilon}))ds+\cO(h^{\frac{N}{3}}).
\end{align*}
Summing up over $i$, we are left in view of \eqref{eq: partition unity gamma-1}
with 
\[
\Tr\Pi U(t)\Op_{h}(\chi_{\gamma})\phi^{(1)}(t)\Pi=\frac{\ell^{\sharp}(\gamma)}{|1-P_{\gamma}|^{\fr12}}\e^{\frac{\i\ell(\gamma)}{h}}\phi^{(1)}(\ell(\gamma))+\cO(h^{1-C\epsilon})+\cO(h^{\frac{N}{3}}),
\]
and this concludes the proof of the Lemma for large enough $N$. 
\end{proof}
\end{lem}
Finally, observing that the number of periodic orbit that have length
in $\supp\phi^{(1)}$ can not exceed $\e^{h_{\top}\ell_{0}}=\cO(h^{-C\epsilon})$,
we have 
\begin{align*}
\Tr\int\Pi f(hQ)U(t)\Pi\e^{\frac{\i}{h}t}\phi^{(1)}(t)dt= & \sum_{\gamma:\ell(\gamma)\in\supp\phi^{(1)}}\e^{\i\frac{\ell(\gamma)}{h}}\frac{\ell^{\sharp}(\gamma)}{\sqrt{|1-P_{\gamma}|}}\phi^{(1)}(\ell(\gamma))+\cO(h^{1-C\epsilon})\\
 & +\Tr\int\Pi U(t)\lp\sum_{j\in J}\Op_{h}(\ti\chi_{j})\rp\e^{\frac{\i}{h}t}\phi^{(1)}(t)dt
\end{align*}
and it remains to check that the last term does not contribute to
the wave trace in the limit $h\to0$, as it is expected.

\subsubsection{Remaining contributions to the wave trace}

To complete the proof of Proposition \ref{prop: Long Time Trace Formula},
we indicate briefly how to deal with the terms microlocalized outside
the periodic orbits with length in $\supp\phi^{(1)}$. The operator
$\Pi U(t)\sum_{j\in J}\Op_{h}(\ti\chi_{j})$ is a sum of semiclassical
Fourier integral operators, and because of the crucial fact that the
time involved is always $\leq\epsilon\log h^{-1}$ for sufficiently
small $\epsilon$, they can again be represented by local oscillatory
integrals with a compactly supported Schwartz kernel of the form 
\[
\int\e^{\frac{\i}{h}S_{j}(t,x,y,\eta)}b_{j}(t,x,y,\eta,h)d\eta.
\]
Here $S_{j}$ is a generating function of the canonical relation 
\begin{align*}
C & =\{\lp(t,e),(x,\xi),(y,\eta)\rp;(x,\xi),(y,\eta)\in T^{*}M\setminus0,\\
 & (t,e)\in T^{*}\R\setminus0,\ e+q(x,\xi)=0,\ (x,\xi)=\Phi^{t}(y,\eta)\}
\end{align*}
and $b_{j}\sim\sum h^{k}b_{j,k}$ . Such an integral representation
can for instance be obtained by composing $\cO(t)$ times a Fourier
integral operator quantizing the time 1 symplectic transformation
$\Phi^{1}:T^{*}M\to T^{*}M$ after microlocalizing with $\Op_{h}(\ti\chi_{j})$.
From the transport equations \eqref{eq: Transport Higher} one can
show by applying crudely the chain rule inductively (exactly as in
Lemma \ref{lem: Symbol Derivatives}) that there is a positive constant
$\cC=\cC_{M,p}$ depending only on $M$ and $p$ (via $\Phi^{1}$)
such that 
\begin{equation}
\|b_{j,k}\|_{C^{\ell}}\leq C_{j,k,\ell}(t+1)^{\ell+3k}\cC^{t(\ell+1)}.\label{eq: estimates symbols outside orbits}
\end{equation}
Actually this type of estimate is true even without any hyperbolicity
assumption, which essentially allows to replace the term $\cC^{t(\ell+1)}$
by $\e^{-tC_{M,p}}$ for some $C_{M,p}>0$ using Proposition \ref{prop: NZ 5}.
It follows that we can perform integrations by parts in the integral
\[
\int\e^{\frac{\i}{h}(S_{j}(t,x,x,\eta)+t)}b_{j}(t,x,x,\eta,h)\phi^{(1)}(t)d\eta dxdt
\]
since the critical equations 
\[
\d_{\eta}S_{j}(t,x,x,\eta)=x,\ \d_{x}S_{j}(t,x,x,\eta)+\d_{y}S_{j}(t,x,x,\eta)=0,\ 1+\d_{t}S_{j}(t,x,x,\eta)=0
\]
can not be satisfied, due to the fact that $\supp\ti\chi_{j}$ do
not contain any point belonging to a periodic orbit with unit energy
and length in $\supp\phi^{(1)}$. As a result, we finally have 
\[
\Tr\int\Pi f(hQ)U(t)\lp\sum_{j\in J}\Op_{h}(\ti\chi_{j})\rp\phi^{(1)}(t)dt=\cO(h^{\infty})
\]
or more precisely, $\cO_{N}(h^{N(1-C\epsilon)})$ for some $C>0$
and any $N\in\N$ in view of \eqref{eq: estimates symbols outside orbits}.
Indeed the number of terms in the sum is bounded above by $\cO(h^{-C\epsilon})$,
due to the fact that the support of $\Pi$ has size $\log h^{-1}$
and that the $\ti\chi_{j}$ are localized in balls of size $\cO(h^{\epsilon L_{0}})$.

Finally, the preceding construction applies identically when replacing
$U(t)$ by the truncated free wave group $(1-\chi)U_{0}(1-\chi)$,
and for the same reasons as above we have: 
\[
\Tr\int\Pi f(hQ)(1-\chi)U_{0}(t)(1-\chi)\Pi\e^{\i\frac{t}{h}}\phi^{(1)}(t)dt=\cO(h^{\infty}).
\]
To summarize, we have shown in this section that 
\begin{align*}
\Tr\int\Pi f(hQ)(U(t)-(1-\chi)U_{0}(t)(1-\chi))\Pi\e^{\i\frac{t}{h}}\phi^{(1)}(t)dt & =\sum_{\gamma:\ell(\gamma)=\ell_{0}}\frac{\e^{\i\frac{\ell(\gamma)}{h}}\ell^{\sharp}(\gamma)}{\sqrt{|1-P_{\gamma}|}}\phi^{(1)}(\ell(\gamma))\\
 & +\cO(h^{1-C\epsilon})+\cO(h^{\infty})\,,
\end{align*}
and this concludes the proof of Proposition \ref{prop: Long Time Trace Formula}
for $\phi^{(1)}$.

\subsection{The case of $\phi^{(2)}$}

Even if the principle of the proof of Proposition \ref{prop: Long Time Trace Formula}
is the same with $\phi^{(2)}$, some significant changes must be made
in order to take into account the fact that now $\supp\phi^{(2)}=[T-1,T]$
contains exponentially many periodic orbits as $T\to+\infty$.

Let $\delta_{0}$ be as in Lemma \ref{lem: Separation Orbits}. To
divide the time interval $[T-1,T]$ into subintervals of size $\delta_{0}$,
let $K_{0}=\left\lfloor \delta_{0}^{-1}\right\rfloor +1$ and 
\[
(J_{k})_{0\leq k\leq K_{0}-1}=[T-1+k\delta_{0},T-1+(k+1)\delta_{0}],\quad J_{K_{0}}=[T-1+\left\lfloor \delta_{0}^{-1}\right\rfloor \delta_{0},T].
\]
Now we can find some functions $f_{k}\in C_{0}^{\infty}(\R)$ such
that $f_{k}$ is supported near $J_{k}$ and furthermore they realize
a partition of unity near $\supp\phi^{(2)}$, namely 
\begin{equation}
\forall t\in[T-1,T],\qquad\sum_{k=0}^{K_{0}}f_{k}(t)\phi^{(2)}(t)=\phi^{(2)}(t).\label{eq: Partition time windows}
\end{equation}
Up to shrink $\delta_{0}$ and enlarge $K_{+}$ further, we can assume
without loss of generality if $T$ is large enough that if $\gamma,\gamma'\in\scP$
are such that $\ell(\gamma),\ell(\gamma')\in\supp f_{k}$, then the
neighborhoods $\Theta(\gamma,\e^{-K_{+}T})$ and $\Theta(\gamma',\e^{-K_{+}T})$
are disjoint. Note also that since $T\leq\log\log h^{-1}$, these
neighborhoods are now of size $\cO(-1/\log h)$.

We can then proceed to the proof of the long time trace formula as
in the preceding section, working first with 
\[
\phi_{k}\defeq f_{k}\phi^{(2)}
\]
and then summing up the contributions according to \eqref{eq: Partition time windows},
the point being that by construction, we can isolate microlocally
periodic orbits whose length is in $\supp\phi_{k}$. For each $\gamma$
such that $\ell(\gamma)\in\supp\phi_{k}$, we can define again the
cutoff functions $\chi_{\gamma^{i}}\in C_{0}^{\infty}(B_{\rho_{i}}(\e^{-TK_{+}}))$
forming a partition of unity around $\gamma$ in phase space, and
write: 
\begin{align*}
\Pi f(hQ)U(t)\phi^{(2)}(t)\Pi & =\sum_{k}\Pi U(t)\lp\sum_{\gamma:\ell(\gamma)\in\supp\phi_{k}}\Op_{h}(\chi_{\gamma})+\sum_{j\in J_{k}}\Op_{h}\ti\chi_{j}\rp\phi_{k}(t)\Pi
\end{align*}
where again the $\ti\chi_{j}$ form a partition of unity associated
to a cover of 
\[
\cE^{*}M_{res_{k}}^{\delta/2}\defeq\lp\cE^{*}M^{\delta/2}\cap T^{*}(\supp\Pi)\rp\setminus\bigcup_{\gamma:\ell(\gamma)\in\supp\phi_{k}}\bigcup_{i=1}^{N_{\gamma}}B_{\rho_{i}}(\e^{-TK_{+}})
\]
which stays away in $T^{*}M$ from the orbits with length in $\supp\phi_{k}$.
We can define as well 
\[
\Tr(\gamma,k,i)\defeq\Tr\int\Pi f(hQ)U(t)\oph(\chi_{\gamma^{i}})\Pi\e^{\frac{\i}{h}t}\phi_{k}(t)dt\,,
\]
Since the orbits $\gamma$ entering in $\Tr(\gamma,k,i)$ are microlocally
isolated from each other, we can perform the stationary phase as in
the preceding section to get 
\[
\Tr(\gamma,k,i)=\frac{1}{|1-P_{\gamma}|^{\fr12}}\e^{\frac{\i\ell(\gamma)}{h}}\phi_{k}(\ell(\gamma))\int_{\gamma}\chi_{\gamma^{i}}(1+\cO(h^{1-C\epsilon}))ds
\]
and adding the contributions microlocalized outside the periodic orbits,
this yields again to 
\[
\Tr\int\Pi f(hQ)U(t)\Pi\e^{\frac{\i}{h}t}\phi_{k}(t)dt=\sum_{\gamma:\ell(\gamma)\in\supp\phi_{k}}\e^{\i\frac{\ell(\gamma)}{h}}\frac{\ell^{\sharp}(\gamma)}{\sqrt{|1-P_{\gamma}|}}\phi_{k}(\ell_{\gamma})+\cO(h^{1-C\epsilon})+\cO(h^{\infty}).
\]
Summing up the contributions in $k$ conclude the proof in view of
\eqref{eq: Partition time windows}.

\subsection{Proof of Proposition \ref{prop: Trace Simplifiee} \label{sub: proof trace simple}. }

The arguments of the above sections shows that we have 
\[
\Tr\int\Pi f(hQ)U(t)\Pi\e^{-\i\lambda t}\phi^{(i)}(t)dxdt=\cO(\lambda^{-\infty})
\]
because the critical equations in the stationary phase expansion cannot
be satisfied in this case since $t,\lambda>0$. So 
\begin{align*}
\Re\left(\Tr\int\Pi f(hQ)U(t)\Pi\e^{\i\lambda t}\phi^{(i)}(t)dt\right) & =\Re\left(\Tr\int\Pi f(hQ)U(t)\Pi2\cos(\lambda t)\phi^{(i)}(t)dt\right)\\
 & +\cO(\lambda^{-\infty})\\
 & =2\Tr\int\Pi f(hQ)u(t)\Pi\cos(\lambda t)\phi^{(i)}(t)dt+\cO(\lambda^{-\infty}).
\end{align*}
From the above sections we also know that 
\[
\Tr\int_{\R}\Pi(1-\chi)f(hQ)U_{0}(t)(1-\chi)\Pi\e^{\pm\i\lambda t}\phi^{(i)}(t)dt=\cO(\lambda^{-\infty})
\]
which allows to conclude the proof of the proposition.

\medskip{}

\noun{Acknowledgements. } This work has been partially supported by
the Agence Nationale de la Recherche, under the grant Gerasic-ANR-13-BS01-0007-0.

\end{document}